\documentclass[11pt]{elsarticle}

\usepackage{xcolor}
\usepackage[colorlinks=true,linkcolor=blue,citecolor=red]{hyperref}
\usepackage{amsmath,amssymb,amsthm}
\usepackage{esint}
\usepackage{enumitem}
\usepackage{mathrsfs}
\usepackage{tikz}
\usetikzlibrary{arrows.meta, positioning, decorations.pathreplacing, calc}
\usepackage{booktabs}
\usepackage{graphicx}

\usepackage{breakurl}

\newcommand{\LCLossInit}{9.6\times 10^{-2}}
\newcommand{\LCLossFinal}{4.1\times 10^{-14}}
\newcommand{\LCIntSelfCheck}{4.7\times 10^{-10}}
\newcommand{\LCMeanDH}{0.0414}
\newcommand{\LCnDense}{5\,000}
\newcommand{\LCResidMeanDense}{3.2\times 10^{-7}}
\newcommand{\LCResidMaxDense}{6.0\times 10^{-4}}
\newcommand{\LCtimingReps}{5}
\newcommand{\LCtimeQHmin}{$241$\,ms}
\newcommand{\LCtimeQHmax}{$251$\,ms}
\newcommand{\LCtimeNoQHfirst}{$258$\,ms}
\newcommand{\LCtimeNoQHlast}{$3.14$\,s}
\newcommand{\LCslowdown}{12.5\times}
\newcommand{\LCtimeQHdSix}{$1.00$\,s}
\newcommand{\LCtimeNoQHdSix}{$1.13$\,s}
\newcommand{\LCtimeQHdEight}{$0.91$\,s}
\newcommand{\LCtimeNoQHdEight}{$0.91$\,s}

\newcommand{\EllipMismA}{2.0\times 10^{-5}}
\newcommand{\EllipMeanLA}{0.972}
\newcommand{\EllipMaxLA}{0.982}
\newcommand{\EllipMismB}{4.8\times 10^{-6}}
\newcommand{\EllipMeanLB}{0.971}
\newcommand{\EllipMaxLB}{0.982}
\newcommand{\EllipMismC}{9.7\times 10^{-6}}
\newcommand{\EllipMeanLC}{0.968}
\newcommand{\EllipMaxLC}{0.982}
\newcommand{\EllipDenseMaxL}{0.987}
\newcommand{\EllipNDense}{20\,001}
\newcommand{\EllipExtremalXtwo}{-0.05}
\newcommand{\EllipCompLoss}{6.9\times 10^{-19}}
\newcommand{\EllipCompDenseMean}{2.1\times 10^{-9}}
\newcommand{\EllipCompDenseMax}{1.4\times 10^{-6}}
\newcommand{\EllipCompMeanL}{0.923}
\newcommand{\EllipCompMaxL}{1.000006}
\newcommand{\EllipCompGradCheck}{7.2\times 10^{-11}}

\setenumerate{label=(\alph*)}
\setenumerate[2]{label=(\Alph*)}

\newtheorem{theorem}{Theorem}[section]
\newtheorem{proposition}[theorem]{Proposition}
\newtheorem{lemma}[theorem]{Lemma}

\theoremstyle{definition}
\newtheorem{definition}[theorem]{Definition}

\theoremstyle{remark}
\newtheorem{remark}[theorem]{Remark}

\numberwithin{equation}{section}

\def\dist{\operatorname{dist}}

\def\gr{\operatorname{gr}}

\DeclareMathOperator{\argmin}{argmin}
\DeclareMathOperator{\Sign}{Sign}

\let\oldtocsection=\tocsection

\let\oldtocsubsection=\tocsubsection

\let\oldtocsubsubsection=\tocsubsubsection

\newcommand{\tocsection}[2]{\hspace{0em}\oldtocsection{#1}{#2}\textbf}
\newcommand{\tocsubsection}[2]{\hspace{1em}\oldtocsubsection{#1}{#2}}
\newcommand{\tocsubsubsection}[2]{\hspace{2em}\oldtocsubsubsection{#1}{#2}}

\begin{document}
	
	\begin{frontmatter}
		
		\title{Distance-Residual Physics-Informed Neural Networks:\\
			A Deep Learning Framework for Differential and\\
			Partial Differential Inclusions}
		
		\author[verkin]{Maria Filipkovska}
		\author[upct]{Juan Jos\'e Mar\'{\i}n}
		\author[itu]{Isil Oner}
		\author[upct]{Francisco Periago\corref{cor}}
		\author[coper]{Krzysztof Rykaczewski}
		
		\cortext[cor]{Corresponding author}
		
		\address[verkin]{B. Verkin Institute for Low Temperature Physics and Engineering of the National Academy of Sciences of Ukraine,\, Nauky Ave. 47, Kharkiv 61103, Ukraine.}
		
		\address[upct]{Department of Applied Mathematics and Statistics, Universidad Polit\'ecnica de Cartagena, member of the European University of Technology EUT+, Campus Muralla del Mar, 30202 Cartagena, Murcia, Spain.}
		
		\address[itu]{Department of Mathematics, Faculty of Sciences, Gebze Technical University, Kocaeli 41400, Turkey.}
		
		\address[coper]{Faculty of Mathematics and Computer Science, Nicolaus Copernicus University in Toru\'n, Chopina 12/18, 87-100 Toru\'n, Poland.}
		
		\begin{abstract}
		We introduce \emph{Distance-Residual Physics-Informed Neural Networks} (DR-PINNs), a physics-informed learning framework for approximating solutions of ordinary and partial differential inclusions, governing laws in which a differential operator is constrained to lie in a set-valued map 	rather than equaling a prescribed function.  The method replaces the classical pointwise PDE/ODE residual by the squared distance from the differential operator to the admissible set. This distance vanishes exactly when the inclusion is satisfied and measures the infimal correction needed for the operator to enter the admissible set. For a fixed closed convex admissible set, the squared distance is differentiable with respect to the operator value, with gradient given by the metric projection. When the admissible set also depends on the network state, this dependence is included through the chain rule. The framework encompasses ordinary differential inclusions and partial differential inclusions with set-valued reaction terms. For both settings we prove \emph{consistency}: under the stated closedness, measurability, convexity, and growth assumptions, any sequence of candidates satisfying the initial (and, in the parabolic case, boundary) conditions whose continuous distance-residual functional tends to zero admits a subsequence converging to an exact solution of the target inclusion. These are conditional statements for the continuous distance-residual functional; they do not cover the finite-collocation training loss, the behaviour of the optimizer, or convergence rates. The proofs combine a priori estimates (Gr\"onwall inequality and parabolic energy estimates), weak compactness arguments (Arzel\`a-Ascoli and an Aubin-Lions-type compact embedding), and the lower semicontinuity of normal convex integrands.  For admissible sets given as convex hulls of finitely many vertices, projection onto the set reduces to a small convex quadratic program, making the loss efficiently computable inside the training loop. Numerical experiments demonstrate high accuracy on the differential-inclusion benchmarks; in the parabolic relay benchmark, trained against an $\varepsilon$-banded relaxation of the relay, the learned solutions reproduce the finite-time extinction dynamics quantitatively in the band-entry observable (seed-stable, persistent, and matching a reference solver to the resolution of the evaluation grid) while the distance residual remains concentrated in a thin layer ahead of the discontinuous extinction front.
		\end{abstract}
		
		
	\end{frontmatter}
	
	\tableofcontents
	\newpage
	
	\section{Introduction}
	\label{sec:intro}
	
	Physics-Informed Neural Networks (PINNs) approximate solutions of
	differential equations by incorporating the governing law directly into
	the training loss \cite{raissi2019}. Since the work of Raissi,
	Perdikaris, and Karniadakis, they have been used for nonlinear PDEs in
	fluid mechanics, heat transfer, and wave propagation, as well as inverse
	problems, parameter identification, operator learning, and
	high-dimensional stochastic differential equations. In these settings the
	governing law is single-valued, for example an ODE
	$\dot{x}(t)=f(t,x(t))$ or a PDE $\mathcal{D}(u)=f$, and the PINN loss is
	built from the pointwise residual obtained after substituting a neural
	network surrogate into that equality.
	
	The same residual construction is not available when the governing law is
	set-valued:
	\[
	\mathcal{D}(u) \in \mathcal{F}(u),
	\]
	where $\mathcal{F}(u)$ is a multifunction (set-valued map).  Concrete
	instances include:
	\begin{itemize}
		\item \emph{Differential inclusions}: $\dot{x}(t)\in F(t,x(t))$,
		where $F(t,x)$ is a set of admissible velocities;
		\item \emph{Parabolic partial differential inclusions}:
		$\partial_t u - \Delta u \in \Phi(u)$, where $\Phi(u)$ is a
		set-valued reaction term.
	\end{itemize}
	In these cases there is no single-valued right-hand side against which to
	form the usual residual. For a set-valued law
	$\mathcal{D}(u)\in\mathcal{F}(u)$, consider instead the quantity
	\[
	\dist\bigl(\mathcal{D}(u),\,\mathcal{F}(u)\bigr)
	:= \inf_{v\in \mathcal{F}(u)} \|\mathcal{D}(u) - v \|
	\]
	which is well defined whenever $\mathcal{F}(u)$ is nonempty and
	closed, and it vanishes if and only if the inclusion is satisfied.  This
	suggests replacing the classical PINN residual by a \emph{distance
		residual} and training a neural network to minimize the loss
	\[
	\mathcal{L}(\theta)
	=
	\frac{1}{N}\sum_{i=1}^{N}
	\dist^2\bigl(\mathcal{D}(u_\theta)(z_i),\,
	\mathcal{F}(u_\theta)(z_i)\bigr),
	\]
	where $u_\theta$ is a neural network surrogate of $u$, and $\{z_i\}_{i=1}^N$ are collocation points.  The resulting methodology, which
	we call \emph{Distance-Residual Physics-Informed Neural Networks}
	(DR-PINNs), extends the PINN paradigm to genuinely set-valued
	differential laws, while reducing to the classical formulation when
	$\mathcal{F}(u)=\{f(u)\}$ is a singleton.
	
	\vspace{0.3cm}
	
	\noindent\textbf{Related work.}
	Distance and projection operators have also been used in PINNs, but for different purposes.  Deguchi and Asai \cite{deguchi2025} use $R$-function-based approximate distance fields to impose Dirichlet
	boundary conditions exactly on non-convex domains, so the distance
	encodes the geometry of the boundary rather than an admissible set for
	the dynamics.  Several recent architectures equip a PINN with a
	projection layer that maps an unconstrained network output onto a
	feasible region to guarantee hard-constraint satisfaction; see the
	KKT-based projection of Iftakher et al.\ \cite{iftakher2025} or the
	integral-conservation projection of Baez et al.\ \cite{baez2025}.  In
	all of these constructions the projection acts once, on the network
	output, to enforce a fixed constraint; it does not serve as the training
	residual, and none of them addresses a genuinely set-valued evolution law.
	
	Wu, Lisser, and Chamoin
	\cite{wu2024cmame,wu2023,wu2024neuropinn,wu2025errorbound} study
	variational inequalities and complementarity problems using neurodynamic
	optimization and PINNs. They first rewrite the variational inequality as
	an ODE whose right-hand side contains a metric projection onto the fixed
	feasible set, and then train a PINN with the ordinary pointwise residual
	of that auxiliary equation. In this formulation the set-valued aspect is
	removed before training. DR-PINNs instead leave
	$\dot{x}(t)\in F(t,x(t))$ multivalued and measure the distance to the
	state-dependent set $F(t,x_\theta(t))$, so the projection is recomputed
	along the current trajectory during training.
	
	Kasyanov, Kapustyan, Levenchuk, and Novykov \cite{kasyanov2024}
	proposed a machine-learning method (including PINNs and the deep Galerkin
	method) for reacti\-on-diffusion equations with \emph{multivalued}
	interaction functions. Their method first replaces the multifunction by a
	single-valued Lipschitz approximation (a
	Pasch--Hausdorff-type regularization), and a classical pointwise
	residual is then trained for the regularized, single-valued equation;
	the set-valued character of the law is thus removed \emph{before}
	training, and the quality of the approximation is governed by an
	additional regularization parameter. A follow-up study
	\cite{paliichuk2026} analyzes the Monte Carlo estimation of the
	corresponding loss functionals and shows that the variance of the
	standard estimator degenerates as the regularization parameter grows,
	which further underlines the cost of the prior single-valued
	regularization. DR-PINNs instead keep the inclusion genuinely
	multivalued and use the distance to the (possibly state-dependent)
	admissible set itself as the training residual, with no regularization
	parameter to drive to a limit; this requires the evaluation of a
	metric projection inside the training loop. A direct experimental comparison with the regularization-based approach of
	\cite{kasyanov2024} on common benchmarks remains to be carried out.
	
	The theory of differential inclusions is well developed
	\cite{aubin_cellina_1984,aubin_frankowska_2009,deimling_1992}, whereas
	PINN-type methods for set-valued problems have so far, to the best of our
	knowledge, used a prior single-valued regularization of the multifunction
	\cite{kasyanov2024,paliichuk2026}. The contribution here is to use the
	distance from the differential operator to the admissible set directly as
	the training residual, without prior single-valued regularization, and to
	prove consistency results for both ordinary and parabolic differential
	inclusions.
	
	\vspace{0.3cm}
	
	\noindent\textbf{Contributions. }
	The paper makes the following contributions:
	\begin{enumerate}
		\item A general \emph{distance-residual principle} (Section~\ref{sec:principle})
		that provides a unified method for constructing PINN losses for
		set-valued differential laws with nonempty closed values. The later
		properties use stronger hypotheses: differentiability of the loss and
		computability of the projection require closed \emph{convex} values,
		while the consistency results additionally use the measurability, growth,
		and compactness assumptions stated in
		Sections~\ref{subsec:consistency} and~\ref{sec:pde}.
		\item \emph{Conditional consistency theorems} for both ordinary differential
		inclusions (Section~\ref{subsec:consistency}) and partial
		differential inclusions (Section~\ref{sec:pde-consistency}), showing that
		vanishing \emph{continuous} distance-residual functionals certify
		admissible solutions. Their relation to finite-collocation training is
		discussed in Remark~\ref{rem:collocation-gap}; the theorems do not
		address optimizer convergence or convergence rates.
		\item A \emph{practical projection algorithm} (Section~\ref{sec:projection-extreme-points})
		for admissible sets given as convex hulls, based on a small quadratic
		program combined with the Quickhull algorithm.
		\item \emph{Numerical experiments} (Section~\ref{sec:numerics}) on
		benchmark differential and partial differential inclusions.
	\end{enumerate}
	
%
%
	
	\section{Preliminaries}
	\label{sec:preliminaries}
	
	\subsection{Set-valued functions}
	
	We collect the notation for set-valued functions used below, including the
	notions of graph and semicontinuity.
	
	\begin{definition}
		Let $X$ and $Y$ be nonempty sets. A \emph{set-valued function} (or \emph{multifunction}) from $X$ to $Y$ is a mapping
		\[
		\mathcal{F}:X \rightrightarrows Y,
		\]
		which assigns to each $x\in X$ a subset $\mathcal{F}(x)\subseteq Y$. The \emph{graph} of $\mathcal{F}$ is defined as
		\[
		\gr(\mathcal{F}):=\{(x,y)\in X\times Y:\, y\in \mathcal{F}(x)\}.
		\]
		For a subset $A\subseteq X$, the image of $A$ under $\mathcal{F}$ is given by
		\[
		\mathcal{F}(A):=\bigcup_{x\in A}\mathcal{F}(x).
		\]
	\end{definition}

	\begin{definition}
		Let $X$ and $Y$ be topological spaces, and
		$\mathcal{F}:X\rightrightarrows Y$ a set-valued map.
		We say that $\mathcal{F}$ is \emph{upper semicontinuous} at $x\in X$ whenever for every open set
		$U$ containing $\mathcal{F}(x)$ there exists a neighborhood $V$ of $x$ such that
		\[
		\mathcal{F}(V)\subset U.
		\]
		The map $\mathcal{F}$ is said to be \emph{lower semicontinuous} at $x\in X$
		if for every $y\in \mathcal{F}(x)$ and every neighborhood $U$ of $y$, there exists a neighborhood $V$ of $x$ such that
		\[
		\mathcal{F}(x') \cap U \neq \varnothing \quad \text{for every } x' \in V.
		\]
		$\mathcal{F}$ is said to be upper semicontinuous (respectively, lower semicontinuous) if it is upper semicontinuous (respectively, lower semicontinuous) at every point $x\in X$.
		Finally, $\mathcal{F}$ is called \emph{continuous} if it is both upper and
		lower semicontinuous.
	\end{definition}

	Following \cite{aubin_frankowska_2009}, we formulate properties of a set-valued map in terms of its graph: the map is said to satisfy property $P$ when its graph has that property.
	For instance, a set-valued map is said to be closed (respectively convex, measurable) if and only if its graph is closed (respectively convex, measurable).

	In contrast, $\mathcal{F}$ is said to be \emph{closed-valued}, \emph{convex-valued}, \emph{bounded-valued}, \emph{compact-valued}, or \emph{strict} if for every $x\in X$, the set \(\mathcal{F}(x)\) is closed, convex, bounded, compact, or nonempty, respectively.
It is easy to show that every strict, closed, set-valued function
$\mathcal{F}:X\rightrightarrows Y$ between metric spaces is closed-valued
(see, e.g., \cite[p.\,22]{BGMO1986}).

	We conclude by recalling the notions of selection and integral of a set-valued function.
	
	\begin{definition}
		Let $X$ and $Y$ be nonempty sets and
		assume that $\mathcal{F}:X\rightrightarrows Y$ is a strict set-valued function.
		A selection from $\mathcal{F}$ is a function $f:X \to Y$
		such that $f(x)\in\mathcal{F}(x)$ for every $x\in X$.
	\end{definition}
	
	\begin{definition}
		Let $I\subseteq\mathbb{R}$ be an interval and
		assume that  $\mathcal{F}:I\rightrightarrows \mathbb{R}^d$ is strict and upper semicontinuous. The integral of $\mathcal{F}$ over $I$ is defined as the set
		\[
		\int_I \mathcal{F}(s) \, ds := \Big\{ \int_I f(s)\,ds : f \text{ is an integrable selection from } \mathcal{F} \Big\}.
		\]
	\end{definition}

	\subsection{Projection onto a set} \label{sec:projection}
	
	Fix a normed space $X$.
	If $C \subseteq X$ is a nonempty closed set, then the infimum
	\begin{equation*}
		\text{dist}(x_0,C) := \inf\{\|x_0 - x\|_X : x \in C\}, \quad x_0 \in X,
	\end{equation*}
	is always well defined; it is attained, for instance, whenever $X$ is finite-dimensional or, more generally, whenever $C$ is boundedly compact (in a general infinite-dimensional normed space a nearest point may fail to exist). We say that $z\in C$ is a \textit{projection} of $x_0$ onto $C$ whenever $\text{dist}(x_0,C)=\|z - x_0\|_X$. In general there can be several projections of $x_0$ onto $C$, that is, there might be more than one point in $C$ closest to $x_0$.
	A fundamental result in convex analysis (see, e.g., \cite[Theorem~5.2, p.\,132]{brezis_2011}) establishes that, whenever $X$ is a Hilbert space, the projection is unique provided that $C$ is nonempty, closed, and convex.
	
	In the case $X=\mathbb{R}^n$, we assume throughout this work that the norm is the Euclidean norm $\|\cdot\|_2$. To simplify the notation, we write $|\cdot|$ instead of $\|\cdot\|_2$. In this scenario, if $C\subset\mathbb{R}^n$ is nonempty, closed, and \emph{convex}, then the Euclidean projection of $x_0$ onto $C$ is attained and unique, and will be denoted by $\Pi_C(x_0)$. In other words,
	\begin{equation}
		\Pi_{C}(x_0):=\argmin_{x\in C}|x_0-x|.
	\end{equation}
For a nonempty closed set that is \emph{not} convex, a nearest point still exists in $\mathbb{R}^n$, but it need not be unique; whenever we write $\Pi_C$ below, the set $C$ is assumed to be convex.

	\subsection{Background on Differential Inclusions} \label{sec:basicsDI}
	
	Differential inclusions extend ordinary differential equations by
	allowing the velocity of the state variable to belong to a set of
	admissible directions.  Let $T>0$, $x_0\in\mathbb{R}^d$, and assume that
	\[
	F:[0,T]\times\mathbb{R}^d\rightrightarrows\mathbb{R}^d
	\]
	is a set-valued map.  Consider the initial value problem (IVP)
	for a \textit{differential} \textit{inclusion}:
	\begin{equation}\label{eq:inclusion}
		\dot{x}(t)\in F(t,x(t)),\qquad x(0)=x_0.
	\end{equation}
	For each pair $(t,x)$, the set $F(t,x)$ represents the admissible
	velocities available to the system at time $t$ and state $x$.
	
	\begin{definition}
		A function $x:[0,T]\to\mathbb{R}^d$ is a \emph{solution} of the IVP
		\eqref{eq:inclusion} if it is absolutely continuous, satisfies
		$x(0)=x_0$, and $\dot{x}(t)\in F(t,x(t))$ for almost every $t\in[0,T]$.
	\end{definition}
	
	Absolute continuity guarantees that $\dot{x}(t)$
	exists almost everywhere. Thus, for any solution $x$, we have
	the exact representation
	\[
	x(t) = x_0 + \int_{0}^{t} \dot{x}(s) \, ds.
	\]

	However, to translate the inclusion $\dot{x}(t) \in F(t,x(t))$ into
	an integral characterization in terms of $F$, regularity conditions
	are needed (see Lemma \ref{lem:integral_repr} below).
	
	\begin{lemma}{\rm (Integral Representation, \cite[Lemma~1, p.\,99]{aubin_cellina_1984})} \label{lem:integral_repr}
		Assume that $F:[0,T] \times \mathbb{R}^d \rightrightarrows \mathbb{R}^d$ is upper semicontinuous, compact-valued, and convex-valued. Then, a continuous function
		$x(\cdot)$ is a solution on $[0,T]$ of the differential inclusion
		\[
		\dot{x}(t) \in F(t, x(t))
		\]
		if and only if for every pair $(t_1, t_2)\in [0,T]$ with $t_1 < t_2$,
		\[
		x(t_2) \in x(t_1) + \int_{t_1}^{t_2} F(s, x(s)) \, ds.
		\]
	\end{lemma}
	
	The proof of the following result appears in
	\cite[Theorem~5.2, p.\,58]{deimling_1992}.
	
	\begin{theorem}[Basic existence]\label{thm:existence}
		Let $F:[0,T]\times\mathbb{R}^d\rightrightarrows\mathbb{R}^d$ be
		strict, closed-valued, and convex-valued.
		Assume that $F(t,x)$ is measurable in $t$ for each $x\in\mathbb{R}^d$,
		upper semicontinuous in $x$ for each $t\in[0,T]$,
		and satisfies the linear growth condition
		\begin{equation} \label{eq:existence-growth}
		\sup_{v\in F(t,x)}|v|\le m(t)(1+|x|)
		\end{equation}
		for some $m\in L^1(0,T)$ and each $(t,x)\in [0,T]\times \mathbb{R}^d$.  Then, for every $x_0\in\mathbb{R}^d$, there
		exists a solution of \eqref{eq:inclusion}.
	\end{theorem}
	
	\begin{remark}
		Strictness, closedness, and convexity of $F$ support the measu\-rable-selection
		and fixed-point arguments used in the existence theory; see
		\cite{aubin_cellina_1984,aubin_frankowska_2009,deimling_1992}. They also
		ensure the properties of the distance to $F(t,x)$ used later: finiteness,
		attainment, and, under convexity, differentiability.
	\end{remark}
	
	Unlike ordinary differential equations, uniqueness of solutions of
	\eqref{eq:inclusion} is not expected in general.  For a given initial
	condition $x_0$, there may exist many admissible trajectories.  The
	\emph{solution set}
	\begin{equation}\label{eq:solution-DI}
		\mathcal{S}(x_0)
		=
		\bigl\{
		x\in {\rm AC}([0,T];\mathbb{R}^d):
		x(0)=x_0,\;
		\dot{x}(t)\in F(t,x(t))\ \text{a.e. on } [0,T]
		\bigr\},
	\end{equation}
	where ${\rm AC}([0,T];\mathbb{R}^d)$ denotes the collection of
	absolutely continuous functions from $[0,T]$ into $\mathbb{R}^d$,
	is naturally set-valued.  Consequently, a trained DR-PINN approximates
	\emph{one} admissible trajectory, not a distinguished one; this
	interpretation is made precise by the consistency result in
	Section~\ref{subsec:consistency}.
	
\subsection{Background on Parabolic Partial Differential Inclusions} \label{sec:basicsPDI}

Next, we consider partial differential inclusions. Although, for the sake of simplicity, in what follows we only address the parabolic case, most of the main ideas and results extend to other types of partial differential inclusions.

Let $\Omega\subset\mathbb{R}^d$ be a bounded domain with $C^{1,1}$ boundary, or a bounded convex domain, let $T>0$, and set
$Q:=(0,T)\times\Omega$.  Consider the initial-boundary value problem
(IBVP) for the parabolic partial differential inclusion:
\begin{equation}\label{eq:parabolic-inclusion-background}
	\left\{
	\begin{array}{ll}
		\partial_t u(t,x)-\Delta u(t,x)\in\Phi(u(t,x)), & (t,x)\in Q,\\[4pt]
		u(0,x)=u_0(x), & x\in\Omega,\\[4pt]
		u(t,x)=0, & (t,x)\in(0,T)\times\partial\Omega,
	\end{array}
	\right.
\end{equation}
where $\Phi:\mathbb{R}\rightrightarrows\mathbb{R}$ is a set-valued reaction
term and $u_0\in H_0^1(\Omega)$.

The natural functional framework for strong solutions of
\eqref{eq:parabolic-inclusion-background} is the \emph{parabolic energy
space}
\[
\mathcal{W}(0,T)
:=
L^2\bigl(0,T;H^2(\Omega)\cap H_0^1(\Omega)\bigr)
\cap
H^1\bigl(0,T;L^2(\Omega)\bigr),
\]
equipped with the norm
\[
\|u\|_{\mathcal{W}}^2:=\|u\|_{L^2(0,T;H^2(\Omega))}^2
+\|\partial_t u\|_{L^2(Q)}^2.
\]
Setting $V:=H^2(\Omega)\cap H_0^1(\Omega)$
and $W:=L^2(\Omega)$, both Hilbert spaces, we have
\[
\mathcal{W}(0,T)=\Big\{v\in L^2(0,T;V):\frac{dv}{dt}\in L^2(0,T;W)\Big\},
\]
so by \cite[Theorem~II.5.14]{BoyerFabrie2013}, the embedding
\[
\mathcal{W}(0,T)\hookrightarrow
C^0\bigl([0,T];[V,W]_{1/2}\bigr)
\]
is continuous, where $[V,W]_{1/2}$ denotes the interpolation space of
order $1/2$ between $V$ and $W$ (see also the
intermediate-derivatives and trace theory of
\cite[Chapter~1, Theorem~3.1]{LionsMagenes}). Under our standing
regularity assumptions on $\Omega$, the interpolation space is
identified by the following lemma, which yields the continuous
embedding
\begin{equation}\label{eq:W-cont-embedding}
	\mathcal{W}(0,T)\hookrightarrow C\bigl([0,T];H_0^1(\Omega)\bigr).
\end{equation}

\begin{lemma}[Interpolation identity]\label{lem:interp-identity}
	Let $\Omega\subset\mathbb{R}^d$ be a bounded domain with $C^{1,1}$
	boundary, or a bounded convex domain. Then,
	\begin{equation}\label{eq:interp-identity}
		\bigl[H^2(\Omega)\cap H_0^1(\Omega),\,L^2(\Omega)\bigr]_{1/2}
		=
		H_0^1(\Omega).
	\end{equation}
\end{lemma}

\begin{proof}
	Apply the abstract framework of
	\cite[Chapter~1, \S2.1]{LionsMagenes} with $X=H_0^1(\Omega)$
	(equipped with the gradient inner product) and $Y=L^2(\Omega)$. The
	associated positive self-adjoint operator
	(\cite[Chapter~1, \S2.1, eq.~(2.4)]{LionsMagenes}) is precisely the
	Dirichlet Laplacian $-\Delta_D$, i.e., the negative Laplacian $-\Delta$
	restricted to the domain $D(-\Delta_D)=\{u\in H_0^1(\Omega):-\Delta u\in L^2(\Omega)\text{ weakly}\}$. By
	\cite[Chapter~1, \S2.4, Proposition~2.1, eq.~(2.42)]{LionsMagenes},
	\[
	\bigl[D(-\Delta_D),\,L^2(\Omega)\bigr]_{1/2}=H_0^1(\Omega).
	\]
	It therefore remains to identify
	$D(-\Delta_D)=H^2(\Omega)\cap H_0^1(\Omega)$, which is exactly the
	statement of elliptic $H^2$-regularity for the Dirichlet Laplacian.
	This equality holds in both admissible regimes for $\Omega$:
	if $\Omega$ is bounded and convex (with no smoothness assumed on the
	boundary), by \cite[Theorem~3.2.1.2]{Grisvard1985}, and if $\Omega$
	is bounded with $C^{1,1}$ boundary (not necessarily convex), by
	\cite[Theorem~2.4.2.5]{Grisvard1985}, in both cases applied to
	$-\Delta$ with homogeneous Dirichlet conditions. Combining the two
	results yields \eqref{eq:interp-identity}.
\end{proof}

The compactness statement needed below is stronger than the usual
Aubin--Lions conclusion, which gives compactness in
$L^2(0,T;L^2(\Omega))$. We state the required version separately.

\begin{lemma}[{Compact embedding into $C([0,T];L^2(\Omega))$}]
	\label{lem:compact-embedding}
	The embedding
	\[
	\mathcal{W}(0,T)
	\hookrightarrow\hookrightarrow
	C\bigl([0,T];L^2(\Omega)\bigr)
	\]
	is compact.
\end{lemma}

\begin{proof}
	Let $\{u_n\}_{n\in\mathbb{N}}\subset\mathcal{W}(0,T)$ be bounded, that is,
	there exists a constant $M>0$ such that
	$\|u_n\|_{\mathcal{W}}\le M$ for every $n\in\mathbb{N}$. First, for all
	$0\le s\le t\le T$, the fundamental theorem of calculus in
	$L^2(\Omega)$ and the Cauchy--Schwarz inequality give
	\begin{align*}
	\|u_n(t)-u_n(s)\|_{L^2(\Omega)}
	& \le\int_s^t\|\partial_t u_n(\tau)\|_{L^2(\Omega)}\,d\tau
	\\	
	&\le|t-s|^{1/2}\,\|\partial_t u_n\|_{L^2(Q)}
	\le M\,|t-s|^{1/2},
	\end{align*}
	so the sequence is uniformly equicontinuous with values in
	$L^2(\Omega)$. Second, by the continuous embedding
	\eqref{eq:W-cont-embedding} there is a constant $C_e>0$ such that
	$\sup_{t\in[0,T]}\|u_n(t)\|_{H_0^1(\Omega)}\le C_eM$ for every $n$;
	hence, for each fixed $t\in[0,T]$, the set
	$\{u_n(t):n\in\mathbb{N}\}$ is bounded in $H_0^1(\Omega)$ and
	therefore, by the Rellich--Kondrachov theorem, relatively compact in
	$L^2(\Omega)$. The Arzel\`a--Ascoli theorem in the Banach space
	$C([0,T];L^2(\Omega))$ then provides a uniformly convergent
	subsequence, which proves the claimed compactness.
\end{proof}

The (compact, hence continuous) embedding into $C([0,T];L^2(\Omega))$
ensures in particular that the initial condition $u(0,\cdot)=u_0$ is well
defined pointwise in time for every $u\in\mathcal{W}(0,T)$.

\begin{definition}\label{def:strong-solution-pde}
	A function $u\in\mathcal{W}(0,T)$ is a \emph{strong solution} of
	\eqref{eq:parabolic-inclusion-background} if
	\begin{enumerate}
		\item[\rm(a)] $u(0,\cdot)=u_0$ on $\Omega$,
		\item[\rm(b)] $u(t,\cdot)|_{\partial\Omega}=0$ for a.e.\ $t\in(0,T)$,
		\item[\rm(c)] $\partial_t u(t,x)-\Delta u(t,x)\in\Phi(u(t,x))$ for
		a.e.\ $(t,x)\in Q$.
	\end{enumerate}
	Since $u\in\mathcal{W}(0,T)$ implies $\partial_t u,\,\Delta u\in L^2(Q)$,
	the inclusion in~(c) is meaningful pointwise almost everywhere.
	Moreover, by the continuous embedding \eqref{eq:W-cont-embedding},
	every $u\in\mathcal{W}(0,T)$ satisfies $u(t,\cdot)\in H_0^1(\Omega)$
	for \emph{every} $t\in[0,T]$, so condition~(b) holds automatically
	(for every $t$, in fact) and is retained only for emphasis.
\end{definition}

The \emph{solution set} of \eqref{eq:parabolic-inclusion-background} is
\begin{equation}\label{eq:solution-PDI}
	\mathcal{S}(u_0)
	:=
	\bigl\{u\in\mathcal{W}(0,T):u\text{ is a strong solution of }
	\eqref{eq:parabolic-inclusion-background}\bigr\}.
\end{equation}
As in the ODE setting, strong solutions need not be unique, i.e. $\mathcal{S}(u_0)$
may contain multiple elements.

The following existence result guarantees that $\mathcal{S}(u_0)$ is
nonempty under standing assumptions on the set-valued reaction term
$\Phi:\mathbb{R}\rightrightarrows\mathbb{R}$ that are \emph{exactly} the
hypotheses of the consistency analysis of
Section~\ref{sec:pde-consistency}:

\begin{enumerate}[label=(P\arabic*)]
	\item \label{P1} $\Phi$ is strict, closed-valued, and convex-valued.
	
	\item \label{P2} $\Phi$ is upper semicontinuous.
	
	\item \label{P3} $\Phi$ satisfies the linear growth condition
	\eqref{eq:existence-growth} with $m\equiv m_0$ constant, i.e.,
	there exists a constant $m_0\ge 0$ such that
	\[
	\sup_{r\in\Phi(s)}|r|\le m_0(1+|s|)
	\qquad\forall\, s\in\mathbb{R}.
	\]
\end{enumerate}

Since \ref{P2} requires only \emph{upper} semicontinuity, $\Phi$
need not admit any continuous selection; the relay term of
Section~\ref{subsec:parabolic-PDI} is a case in point. Selection
theorems of Michael type (see, e.g., \cite[Lemma~2.1]{deimling_1992}),
which require lower semicontinuity, therefore cannot be used to reduce
\eqref{eq:parabolic-inclusion-background} to a single-valued semilinear
equation. We instead work directly with the multivalued problem and
obtain existence from the fixed-point theorem of Bohnenblust and Karlin
\cite{bohnenblust_karlin_1950}, an extension of Kakutani's theorem
\cite{kakutani_1941} from compact convex sets in finite-dimensional
spaces to closed convex sets in Banach spaces.

\paragraph{Linear solvability}
For $h\in L^2(Q)$ and $u_0\in H_0^1(\Omega)$,
let $S(h)\in\mathcal{W}(0,T)$ denote the unique
solution of the linear Dirichlet problem
\[
\partial_t u-\Delta u=h \ \text{ on } Q,
\qquad u=u_0 \ \text{ on }\{0\}\times\Omega,
\qquad u=0 \ \text{ on }(0,T)\times\partial\Omega.
\]
By parabolic maximal $L^2$-regularity, $S$ is well defined and there
exists a constant $C_{\mathrm{par}}>0$ such that
\begin{equation}\label{eq:max-reg-linear}
	\|S(h)\|_{\mathcal{W}(0,T)}
	\le
	C_{\mathrm{par}}\bigl(\|h\|_{L^2(Q)}+\|u_0\|_{H_0^1(\Omega)}\bigr);
\end{equation}
see \cite[Section~7.1.3, Theorem~5]{Evans} for smooth domains; the
$H^2$ elliptic regularity required under our standing assumptions on
$\Omega$ is supplied by \cite[Theorem~3.2.1.2]{Grisvard1985} in the
convex case and by \cite[Theorem~2.4.2.5]{Grisvard1985} in the
$C^{1,1}$-boundary case (cf.\ the proof of
Lemma~\ref{lem:interp-identity}). We write $S_0$ for the solution
operator of the same problem with zero initial datum and
$w_0:=S(0)\in\mathcal{W}(0,T)$ for the solution of the zero-source
problem with initial datum $u_0$. Then, we have $S(h)=S_0(h)+w_0$; by
\eqref{eq:max-reg-linear} (applied with $u_0=0$), $S_0:L^2(Q)\to
\mathcal{W}(0,T)$ is a bounded linear operator, and, after composition
with the continuous embedding $\mathcal{W}(0,T)\hookrightarrow L^2(Q)$,
it is continuous from $L^2(Q)$ into $L^2(Q)$.

We begin by showing that, for any fixed $w\in L^2(Q)$, the
Nemytskii-type multimap $(t,x)\mapsto\Phi(w(t,x))$ admits at least one
$L^2(Q)$-selection.

\begin{proposition}[Nonemptiness of the selection set]
	\label{prop:bk-selection}
	Assume {\rm \ref{P1}--\ref{P3}} and let $w\in L^2(Q)$. Then
	\[
	S^2_{\Phi,w}
	:=
	\bigl\{v\in L^2(Q) : v(t,x)\in\Phi(w(t,x))\ \text{a.e.\ on } Q\bigr\}
	\]
	is nonempty, closed, and convex in $L^2(Q)$.
\end{proposition}

\begin{proof}
	Let $C\subset\mathbb{R}$ be closed. By \ref{P2}, the set
	$\{s\in\mathbb{R}:\Phi(s)\subset\mathbb{R}\setminus C\}$ is open,
	hence its complement
	\[
	\Phi^{-1}(C):=\{s\in\mathbb{R}:\Phi(s)\cap C\neq\varnothing\}
	\]
	is closed, and in particular Borel. Writing an open set
	$U\subset\mathbb{R}$ as a countable union of closed sets shows that
	$\Phi^{-1}(U)$ is Borel for every open $U$; thus $\Phi$ is a Borel
	measurable multifunction with nonempty closed values, by
	\ref{P1}.
	
	For every closed $C\subset\mathbb{R}$,
	\[
	\bigl\{(t,x)\in Q : \Phi(w(t,x))\cap C\neq\varnothing\bigr\}
	= w^{-1}\bigl(\Phi^{-1}(C)\bigr)
	\]
	is Lebesgue measurable, being the preimage of a closed (hence
	Borel) set under the Lebesgue measurable function $w:Q\to\mathbb{R}$.
	Hence $F(t,x):=\Phi(w(t,x))$ is a measurable multifunction from the
	complete $\sigma$-finite measure space $(Q,\mathscr{L}(Q))$ into the
	nonempty closed subsets of $\mathbb{R}$.
	Here and elsewhere throughout this work, $\mathscr{L}(Q)$ denotes
	the collection of all Lebesgue measurable subsets of $Q$.
	
	By the measurable selection theorem
	\cite[Theorem~8.1.3]{aubin_frankowska_2009}, there exists a
	measurable $v:Q\to\mathbb{R}$ with $v(t,x)\in\Phi(w(t,x))$ a.e.\ on
	$Q$. By \ref{P3},
	$|v(t,x)|\le m_0\bigl(1+|w(t,x)|\bigr)$ a.e., hence
	$|v|^2\le2m_0^2(1+|w|^2)$ a.e.\ and, since $|Q|<\infty$,
	\begin{equation}\label{eq:bk-selection-bound}
		\|v\|_{L^2(Q)}^2\le 2m_0^2\bigl(|Q|+\|w\|_{L^2(Q)}^2\bigr)<\infty.
	\end{equation}
	Therefore $v\in S^2_{\Phi,w}$, so $S^2_{\Phi,w}\neq\varnothing$.
	
	Convexity of $S^2_{\Phi,w}$ follows from the pointwise convexity of
	$\Phi(w(t,x))$ in \ref{P1}. For closedness, if $v_n\in
	S^2_{\Phi,w}$ and $v_n\to v$ in $L^2(Q)$, a subsequence converges
	a.e., and the pointwise closedness of $\Phi(w(t,x))$ gives
	$v\in S^2_{\Phi,w}$.
\end{proof}

The fixed-point theorem we use is the following.

\begin{theorem}[Bohnenblust--Karlin]
	\label{thm:bk-fixed-point}
	Let $X$ be a Banach space and let $D\subset X$ be a nonempty closed
	convex set. Let $\mathcal{G}:D\rightrightarrows X$ be a strict, closed-valued, convex-valued multifunction such that
	$\mathcal{G}(D)\subset D$. Suppose that:
	\begin{enumerate}
		\item[\textup{(a)}] The graph of $\mathcal{G}$ is sequentially
		closed in $D\times X$: if $x_n\to x$ in $D$, $y_n\to y$ in $X$,
		and $y_n\in\mathcal{G}(x_n)$ for every $n$, then
		$y\in\mathcal{G}(x)$.
		\item[\textup{(b)}] $\mathcal{G}(D)$ is contained in a
		sequentially compact subset of $X$.
	\end{enumerate}
	Then $\mathcal{G}$ has a fixed point, i.e., there exists $x_0\in D$ such that
	$x_0\in\mathcal{G}(x_0)$.
\end{theorem}

\begin{proof}
	This is \cite[Theorem~4, p.~159]{bohnenblust_karlin_1950}, restated
	in the notation of the present paper. Bohnenblust and Karlin's proof
	uses the convexity of the values $\mathcal{G}(x)$ through the fact
	that the closed $\varepsilon$-neighbourhood of a \emph{convex} set is
	again convex, which is what allows the finite-dimensional Kakutani
	theorem to be invoked on a finite $\varepsilon$-net, even though
	convexity of the values is not restated as a separately numbered
	hypothesis in their theorem statement; it is made explicit as a
	standing hypothesis in the modern formulation of \cite[Corollary~9.8, p.\,452]{Zeidler1986}, which we have followed above.
\end{proof}

The key analytic ingredient, beyond Proposition~\ref{prop:bk-selection},
is the following weak--strong closedness property of the selection sets.
When passing to the limit along $w_n\to w$ strongly in $L^2(Q)$ and
$v_n\in S^2_{\Phi,w_n}$ such that $v_n\rightharpoonup v$ only weakly in
$L^2(Q)$, it guarantees that the weak limit $v$ remains an admissible
selection of $\Phi(w(\cdot,\cdot))$; without it, one could not pass to
the limit in the set-valued reaction term.

\begin{lemma}[Weak-strong closedness of the selection sets]
	\label{lem:bk-weak-closedness}
	Assume {\rm \ref{P1}--\ref{P3}}. Suppose that
	$\{w_n\}_{n\in\mathbb{N}},\{v_n\}_{n\in\mathbb{N}}\subset L^2(Q)$
	satisfy:
	\begin{enumerate}
		\item[\textup{(i)}] $w_n\to w$ strongly in $L^2(Q)$,
		\item[\textup{(ii)}] $v_n\in S^2_{\Phi,w_n}$ for every
		$n\in\mathbb{N}$,
		\item[\textup{(iii)}] $v_n\rightharpoonup v$ weakly in $L^2(Q)$.
	\end{enumerate}
	Then, $v\in S^2_{\Phi,w}$.
\end{lemma}

\begin{proof}
	Since $w_n\to w$ strongly in $L^2(Q)$, we may pass to a subsequence,
	for which we use the same labeling as for the original sequence,
	along which $w_n(t,x)\to w(t,x)$ for a.e.
	$(t,x)\in Q$. Dropping finitely many terms does not affect a weak
	limit, so for every $n$ the tail sequence $\{v_k\}_{k\ge n}$ also
	converges weakly to $v$. Applying Mazur's lemma (see, e.g.,
	\cite[Corollary~3.8]{brezis_2011}) to this tail sequence yields, for
	each $n$, a finite convex combination supported on indices $k\ge n$,
	\[
	\hat v_n := \sum_{k=n}^{N(n)} \lambda_k^{(n)} v_k,
	\qquad
	\lambda_k^{(n)}\ge0,
	\quad
	\sum_{k=n}^{N(n)}\lambda_k^{(n)}=1,
	\]
	with $\|\hat v_n-v\|_{L^2(Q)}\le 1/n$, so that $\hat v_n\to v$
	strongly in $L^2(Q)$; passing to a further subsequence, not
	relabeled, we may also assume $\hat v_n(t,x)\to v(t,x)$ for a.e.\
	$(t,x)\in Q$.
	
	For each $k$, hypothesis \textup{(ii)} gives
	$v_k(t,x)\in\Phi(w_k(t,x))$ for all $(t,x)$ outside a null set
	$N_k\subset Q$; set $N:=\bigcup_k N_k$, again a null set. Fix any
	$(t,x)\in Q\setminus N$ at which, in addition, both of the above
	a.e.\ convergences hold; such points form a set of full measure. Fix
	$\varepsilon>0$. By \ref{P1}, $\Phi(w(t,x))$ is closed and
	convex, so $\Phi(w(t,x))+\varepsilon B$, with $B$ the open unit ball
	of $\mathbb{R}$, is open and convex. By \ref{P2}, $\Phi$ is
	upper semicontinuous at $w(t,x)$, so there exists $\delta>0$ such
	that $|s-w(t,x)|<\delta$ implies
	$\Phi(s)\subset\Phi(w(t,x))+\varepsilon B$. Since
	$w_k(t,x)\to w(t,x)$, there is $N_1=N_1(t,x,\varepsilon)$ with
	\[
	v_k(t,x)\in\Phi(w_k(t,x))\subset\Phi(w(t,x))+\varepsilon B,
	\qquad k\ge N_1.
	\]
	As $\Phi(w(t,x))+\varepsilon B$ is convex, it contains every convex
	combination of the points $v_k(t,x)$ with $k\ge N_1$; since
	$\hat v_n$ is supported on indices $k\ge n$, we get
	$\hat v_n(t,x)\in\Phi(w(t,x))+\varepsilon B$ for all $n\ge N_1$.
	Letting $n\to\infty$ and using $\hat v_n(t,x)\to v(t,x)$ together
	with the closedness of $\Phi(w(t,x))+\varepsilon\overline{B}$, we
	obtain $v(t,x)\in\Phi(w(t,x))+\varepsilon\overline{B}$. Since
	$\varepsilon>0$ was arbitrary and $\Phi(w(t,x))$ is closed,
	\[
	v(t,x)\in\bigcap_{\varepsilon>0}
	\bigl(\Phi(w(t,x))+\varepsilon\overline{B}\bigr)
	=\Phi(w(t,x)).
	\]
	This holds for a.e.\ $(t,x)\in Q$, the exceptional set being a
	countable union of null sets, so $v\in S^2_{\Phi,w}$.
\end{proof}

\paragraph{The solution multimap and an invariant set}
We seek a strong solution as a fixed point of the set-valued map
\begin{equation}\label{eq:bk-G-def}
	\mathcal{G}:L^2(Q)\rightrightarrows L^2(Q),
	\qquad
	\mathcal{G}(w):=\bigl\{\,S_0(v)+w_0 : v\in S^2_{\Phi,w}\,\bigr\}.
\end{equation}
Indeed, if $w^\ast\in\mathcal{G}(w^\ast)$, then
$w^\ast=S_0(v^\ast)+w_0$ for some $v^\ast\in S^2_{\Phi,w^\ast}$, and
$u:=w^\ast\in\mathcal{W}(0,T)$ solves $\partial_t u-\Delta u=v^\ast\in
\Phi(w^\ast)=\Phi(u)$ a.e.\ on $Q$ with $u(0)=u_0$, i.e.\
$u\in\mathcal{S}(u_0)$.

As is standard when the reaction term has linear (not small) growth, a
self-mapping estimate for $\mathcal{G}$ in the ordinary $L^2(Q)$-norm
need not close, since the growth constant of \ref{P3} enters the
linear energy estimate with a factor that need not be smaller than $1$.
We circumvent this with a Bielecki-type exponentially weighted norm.

\begin{lemma}[Weighted a priori bound]
	\label{lem:bk-bielecki}
	Let $\lambda>0$ and for every $g\in L^2(Q)$, define
	\[
	\|g\|_\lambda^2:=\int_0^T e^{-2\lambda t}
	\|g(t,\cdot)\|_{L^2(\Omega)}^2\,dt.
	\]
	Note that, since $e^{-2\lambda T}\le e^{-2\lambda t}\le1$ on $[0,T]$,
	$\|\cdot\|_\lambda$ is a norm on $L^2(Q)$ equivalent to
	$\|\cdot\|_{L^2(Q)}$ with
	\begin{equation}\label{eq:bk-norm-equiv}
		e^{-\lambda T}\|g\|_{L^2(Q)}\le\|g\|_\lambda\le\|g\|_{L^2(Q)}
		\qquad\text{for all } g\in L^2(Q).
	\end{equation}
	Assume {\rm \ref{P1}--\ref{P3}}. If $w\in L^2(Q)$, $v\in S^2_{\Phi,w}$,
	and $u:=S_0(v)+w_0$, then for every $\lambda>0$,
	\[
	\|u\|_\lambda
	\le
	\|w_0\|_\lambda
	+\frac{\sqrt2\,m_0\sqrt{|Q|}}{\lambda}
	+\frac{\sqrt2\,m_0}{\lambda}\,\|w\|_\lambda.
	\]
\end{lemma}

\begin{proof}
	By \ref{P3}, as in the proof of
	Proposition~\ref{prop:bk-selection},
	\[
	|v(t,x)|^2\le2m_0^2(1+|w(t,x)|^2) \quad\text{a.e.\ on }Q.
	\]
	Multiplying by
	$e^{-2\lambda t}\le1$ and integrating over $Q$ gives
	\begin{equation}\label{eq:bk-v-weighted}
		\|v\|_\lambda\le\sqrt2\,m_0\bigl(\sqrt{|Q|}+\|w\|_\lambda\bigr).
	\end{equation}
	Let $\tilde u:=S_0(v)$, so that $\tilde u(0)=0$ and
	$\partial_t\tilde u-\Delta\tilde u=v$. Multiplying the energy
	identity
	\[
	\tfrac12\tfrac{d}{dt}\|\tilde u(t)\|_{L^2(\Omega)}^2
	+\|\nabla\tilde u(t)\|_{L^2(\Omega)}^2
	=\int_\Omega v(t,x)\,\tilde u(t,x)\,dx
	\]
	by $e^{-2\lambda t}$ one has
	\[
	e^{-2\lambda t}\tfrac{d}{dt}\|\tilde u\|_{L^2(\Omega)}^2
	=\tfrac{d}{dt}\bigl(e^{-2\lambda t}\|\tilde u\|_{L^2(\Omega)}^2\bigr)
	+2\lambda e^{-2\lambda t}\|\tilde u\|_{L^2(\Omega)}^2.
	\]
	Applying Young's inequality,
	\[
	\|v(t)\|_{L^2(\Omega)}\|\tilde u(t)\|_{L^2(\Omega)}
	\le\tfrac{\lambda}{2}\|\tilde u(t)\|_{L^2(\Omega)}^2
	+\tfrac{1}{2\lambda}\|v(t)\|_{L^2(\Omega)}^2,
	\]
	to the source term, integrating over $(0,T)$, using $\tilde u(0)=0$,
	and discarding the nonnegative terms
	$e^{-2\lambda T}\|\tilde u(T)\|_{L^2(\Omega)}^2$ and the gradient
	term, gives $\lambda\|\tilde u\|_\lambda^2\le
	\lambda^{-1}\|v\|_\lambda^2$, i.e.
	\begin{equation}\label{eq:bk-tildeu-weighted}
		\|\tilde u\|_\lambda\le\frac{1}{\lambda}\|v\|_\lambda.
	\end{equation}
	Combining
	\eqref{eq:bk-v-weighted}--\eqref{eq:bk-tildeu-weighted} with the
	triangle inequality
	$\|u\|_\lambda\le\|\tilde u\|_\lambda+\|w_0\|_\lambda$ yields the
	claim.
\end{proof}

Fix $\lambda_0:=2\sqrt2\,m_0+1$, so that
$\sqrt2\,m_0/\lambda_0\le\tfrac12$, and abbreviate
$\|\cdot\|_\ast:=\|\cdot\|_{\lambda_0}$. Set
\[
R:=2\Bigl(\|w_0\|_\ast+\frac{\sqrt2\,m_0\sqrt{|Q|}}{\lambda_0}\Bigr),
\qquad
K:=\bigl\{w\in L^2(Q):\|w\|_\ast\le R\bigr\}.
\]
The set $K$ is nonempty (it contains $0$), closed, bounded, and convex
in $L^2(Q)$; by \eqref{eq:bk-norm-equiv}, every $w\in K$ satisfies
$\|w\|_{L^2(Q)}\le e^{\lambda_0 T}R=:R'$.

\begin{lemma}[Invariance]
	\label{lem:bk-self-map}
	$\mathcal{G}(K)\subset K$.
\end{lemma}

\begin{proof}
	Let $w\in K$ and $u\in\mathcal{G}(w)$. Then, $\|w\|_\ast\le R$
, $u=S_0(v)+w_0$ with $v\in S^2_{\Phi,w}$, and due to Lemma~\ref{lem:bk-bielecki} with $\lambda=\lambda_0$ and the choice
	of $\lambda_0$,
	\[
	\|u\|_\ast
	\le\|w_0\|_\ast+\frac{\sqrt2\,m_0\sqrt{|Q|}}{\lambda_0}
	+\tfrac12\|w\|_\ast
	=\frac{R}{2}+\tfrac12\|w\|_\ast
	\le\frac{R}{2}+\frac{R}{2}=R.
	\]
	Hence $u\in K$.
\end{proof}

\begin{lemma}[Values of $\mathcal{G}$]
	\label{lem:bk-values}
	For every $w\in K$, the set $\mathcal{G}(w)$ is nonempty, closed,
	and convex in $L^2(Q)$.
\end{lemma}

\begin{proof}
	It follows from
	Proposition~\ref{prop:bk-selection} that $\mathcal{G}(w)$ is nonempty.
	Indeed, picking any $v\in S^2_{\Phi,w}$ gives $S_0(v)+w_0\in\mathcal{G}(w)$.
	
	We next turn to the proof of convexity. If $u_i=S_0(v_i)+w_0$ with $v_i\in S^2_{\Phi,w}$,
	$i=1,2$, and $\mu\in[0,1]$, then, by convexity of $S^2_{\Phi,w}$
	(Proposition~\ref{prop:bk-selection}) and linearity of $S_0$,
	\[
	\mu u_1+(1-\mu)u_2
	=S_0\bigl(\mu v_1+(1-\mu)v_2\bigr)+w_0\in\mathcal{G}(w).
	\]
	Therefore, $\mathcal{G}(w)$ is convex.
	
	Finally, let $\{u_n\}_{n\in\mathbb{N}}\subset\mathcal{G}(w)$ with $u_n\to u$
	strongly in $L^2(Q)$, and write $u_n=S_0(v_n)+w_0$ with
	$v_n\in S^2_{\Phi,w}$. Since $w$ is fixed, \ref{P3} and
	\eqref{eq:bk-selection-bound} give a uniform bound
	$\|v_n\|_{L^2(Q)}\le C(w)$. By reflexivity of $L^2(Q)$, a
	subsequence of $\{v_n\}_{n\in\mathbb{N}}$, for which we use the same labeling,
	satisfies $v_n\rightharpoonup v$ weakly
	in $L^2(Q)$. By weak continuity of $S_0$,
	$S_0(v_n)\rightharpoonup S_0(v)$ weakly in $L^2(Q)$; since
	$u_n\to u$ strongly, uniqueness of weak limits yields
	$u=S_0(v)+w_0$. Applying Lemma~\ref{lem:bk-weak-closedness} with the
	stationary sequence $w_n\equiv w$ gives $v\in S^2_{\Phi,w}$, whence
	$u\in\mathcal{G}(w)$. We conclude that $\mathcal{G}(w)$ is closed.
\end{proof}

\begin{lemma}[Relative compactness of the range]
	\label{lem:bk-compact}
	$\overline{\mathcal{G}(K)}$ is a compact subset of $L^2(Q)$.
\end{lemma}

\begin{proof}
	Let $u\in\mathcal{G}(w)$ for some $w\in K$. Then $u=S_0(v)+w_0$ with
	$v\in S^2_{\Phi,w}$ and by \eqref{eq:bk-selection-bound} and estimate
	$\|w\|_{L^2(Q)}\le R'$ we have the bound
	\[
	\|v\|_{L^2(Q)}^2\le2m_0^2\bigl(|Q|+R'^2\bigr)=:M^2,
	\]
	which is uniform over $w\in K$ and over the choice of selection $v$.
	Since $S_0:L^2(Q)\to\mathcal{W}(0,T)$ is bounded,
	\[
	\|u\|_{\mathcal{W}(0,T)}
	\le\|S_0\|\,M+\|w_0\|_{\mathcal{W}(0,T)},
	\]
	so $\mathcal{G}(K)$ is bounded in $\mathcal{W}(0,T)$ (here $\|S_0\|$ denotes the norm of $S_0$ in the space of continuous linear operators acting from $L^2(Q)$ to $\mathcal{W}(0,T)$). By the compact
	embedding $\mathcal{W}(0,T)\hookrightarrow\hookrightarrow
	C([0,T];L^2(\Omega))$ of Lemma~\ref{lem:compact-embedding}, and hence
	into $L^2(Q)$, bounded subsets of $\mathcal{W}(0,T)$ are relatively
	compact in $L^2(Q)$. Thus $\overline{\mathcal{G}(K)}$ is compact in
	$L^2(Q)$.
\end{proof}

\begin{lemma}[Sequential closedness of the graph]
	\label{lem:bk-graph}
	If $\{w_n\}_{n\in\mathbb{N}}\subset K$, $u_n\in\mathcal{G}(w_n)$, $w_n\to w$ and
	$u_n\to u$ strongly in $L^2(Q)$, then $u\in\mathcal{G}(w)$.
\end{lemma}

\begin{proof}
	Write $u_n=S_0(v_n)+w_0$ with $v_n\in S^2_{\Phi,w_n}$. By
	\eqref{eq:bk-selection-bound} and the boundedness of
	$\{w_n\}_{n\in\mathbb{N}}\subset K$ in $L^2(Q)$, the sequence $\{v_n\}_{n\in\mathbb{N}}$ is bounded in
	$L^2(Q)$. By reflexivity, a subsequence of $\{v_n\}_{n\in\mathbb{N}}$, for which we use the same labeling, satisfies
	$v_n\rightharpoonup v$ weakly in $L^2(Q)$. By weak continuity
	of $S_0$, $S_0(v_n)\rightharpoonup S_0(v)$ weakly in $L^2(Q)$; on
	the other hand $S_0(v_n)=u_n-w_0\to u-w_0$ strongly, hence weakly.
	By uniqueness of weak limits, $u=S_0(v)+w_0$. Since $w_n\to w$
	strongly and $v_n\rightharpoonup v$ weakly,
	Lemma~\ref{lem:bk-weak-closedness} gives $v\in S^2_{\Phi,w}$, hence
	$u\in\mathcal{G}(w)$.
\end{proof}

\begin{theorem}[Existence of strong solutions]
	\label{thm:existence-pde}
	Assume {\rm \ref{P1}--\ref{P3}}. Then problem
	\eqref{eq:parabolic-inclusion-background} admits at least one strong
	solution, i.e., $\mathcal{S}(u_0)\neq\varnothing$.
\end{theorem}

\begin{proof}
	$K$ is a nonempty closed convex subset of the Banach space $L^2(Q)$.
	By Lemma~\ref{lem:bk-values}, $\mathcal{G}(w)$ is nonempty, closed,
	and convex for every $w\in K$, and by Lemma~\ref{lem:bk-self-map},
	$\mathcal{G}(K)\subset K$. By Lemma~\ref{lem:bk-graph}, the graph of
	$\mathcal{G}$ is sequentially closed in $K\times L^2(Q)$, which is
	hypothesis \textup{(a)} of Theorem~\ref{thm:bk-fixed-point}; by
	Lemma~\ref{lem:bk-compact}, $\mathcal{G}(K)$ is contained in the
	compact set $\overline{\mathcal{G}(K)}\subset L^2(Q)$, which is
	hypothesis \textup{(b)}. Theorem~\ref{thm:bk-fixed-point}, applied
	with $X=L^2(Q)$ and $D=K$, yields $w^\ast\in K$ with
	$w^\ast\in\mathcal{G}(w^\ast)$, i.e.\ $w^\ast=S_0(v^\ast)+w_0$ for
	some $v^\ast\in S^2_{\Phi,w^\ast}$. Setting $u:=w^\ast$, we obtain
	$u\in\mathcal{W}(0,T)$ with $\partial_t u-\Delta u=v^\ast\in\Phi(u)$
	a.e.\ on $Q$ and $u(0)=u_0$, so $u\in\mathcal{S}(u_0)$ by
	Definition~\ref{def:strong-solution-pde}.
\end{proof}

\begin{remark}
	Theorem~\ref{thm:existence-pde} asserts existence only. Since $\Phi$
	is merely upper semicontinuous, $\mathcal{S}(u_0)$ need not be a
	singleton in general, and the fixed-point argument makes no claim
	about, and does not need, uniqueness of the selection $v^\ast$ or of
	the resulting strong solution. For the relay example of
	Section~\ref{subsec:parabolic-PDI}, uniqueness does hold, but for a
	structural reason unrelated to the fixed-point argument: the
	equation is a subgradient flow of a convex energy; see
	Remark~\ref{rem:relay-wellposed}.
\end{remark}

\begin{remark}
	If $\Phi$ is, in addition, \emph{lower} semicontinuous, a
	Michael-type selection theorem (\cite[Lemma~2.1]{deimling_1992},
	applied with $D=X=\mathbb{R}$) provides a continuous selection
	$f(s)\in\Phi(s)$, which by \ref{P3} inherits the linear growth
	condition $|f(s)|\le m_0(1+|s|)$; problem
	\eqref{eq:parabolic-inclusion-background} then reduces to the
	semilinear equation $\partial_t u-\Delta u=f(u)$, solvable by a
	standard Schaefer fixed-point argument. Requiring both upper and
	lower semicontinuity throughout would, however, make the
	multifunction continuous and would exclude the relay nonlinearity of
	Section~\ref{subsec:parabolic-PDI}, which admits no continuous
	selection; this is precisely why Theorem~\ref{thm:existence-pde} is
	formulated under upper semicontinuity alone.
\end{remark}

	\section{A General Distance-Residual Principle}
	\label{sec:principle}
	
	Assume that $X$ and $Y$ are Banach spaces and consider a differential operator $\mathcal{D}:X\to Y$. Let $\mathcal{F}:X\rightrightarrows Y$ be strict and closed-valued. We consider the abstract
	set-valued differential problem
	\begin{equation}\label{eq:abstract-inclusion}
		\mathcal{D}(u)\in\mathcal{F}(u).
	\end{equation}
	Special cases of \eqref{eq:abstract-inclusion} include differential
	inclusions $\dot{x}(t)\in F(t,x(t))$ and parabolic partial differential
	inclusions $\partial_t u-\Delta u\in\Phi(u)$.
	
	\begin{definition}[Distance residual]
		Let $X$ and $Y$ be Banach spaces and assume that $\mathcal{D}:X\to Y$ is a differential operator. Let $\mathcal{F}:X\rightrightarrows Y$ be strict and
		closed-valued.
		For any $u\in X$, the \emph{distance residual} of $u$ with respect to
		\eqref{eq:abstract-inclusion} is defined as
		\[
		R(u):=\dist^2\bigl(\mathcal{D}(u),\,\mathcal{F}(u)\bigr)
		= \inf_{v\in \mathcal{F}(u)} \|\mathcal{D}(u) - v \|_Y^2.
		\]
		In the concrete settings considered below, the elements of $X$ and
		$Y$ are functions of a variable $z$ ranging over a time interval or a
		space--time cylinder $Z$, and the multifunction $\mathcal{F}$ acts
		through a pointwise set-valued map: there is a multifunction
		$G:Z\times\mathbb{R}^m\rightrightarrows\mathbb{R}^k$ such that
		$\mathcal{F}$ is the associated Nemytskii (superposition) operator,
		\[
		\mathcal{F}(u)
		=
		\bigl\{v\in Y:\ v(z)\in G(z,u(z))\ \text{for a.e. } z\in Z\bigr\},
		\]
		and the inclusion \eqref{eq:abstract-inclusion} is imposed pointwise.
		In that case we define the \emph{pointwise distance residual}
		\[
		r(u;z)
		:=
		\dist^2\bigl((\mathcal{D}u)(z),\,G(z,u(z))\bigr),
		\]
		and the associated DR-PINN loss is
		\[
		\mathcal{L}(\theta)
		:=
		\frac{1}{N}\sum_{i=1}^{N} r(u_\theta;z_i),
		\]
		where $u_\theta$ is a neural-network surrogate and $\{z_i\}_{i=1}^N$ are
		collocation points. Note that $R(u)$ is a single number attached to $u$,
		whereas $r(u;\cdot)$ is a function of the collocation variable; the loss
		$\mathcal{L}$ is an empirical collocation (quadrature) approximation of
		the continuous residual functional $\int_Z r(u_\theta;z)\,dz$
		(in the experiments below, both deterministic uniform grids and
		adaptively re-weighted samples are used).
	\end{definition}

	\begin{remark}[Collocation loss versus continuous functional]
		\label{rem:collocation-gap}
		The consistency results of
		Sections~\ref{subsec:consistency} and~\ref{sec:pde-consistency} are
		formulated for the continuous distance-residual functional
		$\mathcal{J}$, whereas training minimizes the finite-collocation loss
		$\mathcal{L}$. An implication
		passing from $\mathcal{L}(\theta_n)\to0$ to
		$\mathcal{J}(u_{\theta_n})\to0$ requires additional control of the
		sampling scheme and of the regularity of $r(u_\theta;\cdot)$; it is not
		part of the consistency theorems. One sufficient bridge is the
		following. If $Z$ is compact with $|Z|<\infty$, the maps
		$z\mapsto r(u_{\theta_n};z)$ share a Lipschitz constant $M$ on $Z$,
		the collocation sets $\{z_i^{(n)}\}$ have fill distance
		$h_n:=\sup_{z\in Z}\min_i|z-z_i^{(n)}|\to0$, and the
		\emph{maximum} pointwise residual satisfies
		$\max_i r(u_{\theta_n};z_i^{(n)})\to0$, then
		\[
		\sup_{z\in Z}r(u_{\theta_n};z)
		\le\max_i r(u_{\theta_n};z_i^{(n)})+Mh_n\longrightarrow0,
		\]
		and hence $\int_Z r(u_{\theta_n};z)\,dz\to0$, as $n\to\infty$. This condition controls
		the maximum collocation residual; the mean loss $\mathcal{L}$ alone does
		not give the same bound. Accordingly, the dense-grid validations in
		Section~\ref{sec:numerics} report both mean and maximal residuals.
	\end{remark}

	\begin{remark}
		Under the hypothesis of the previous definition,
		the distance residual possesses the following fundamental properties:
		\begin{enumerate}
			\item Non-negativity: $R(u)\ge 0$ for all $u\in X$.
			\item Exactness: if $\mathcal{F}(u)$ is closed, then
			$R(u)=0 \iff \mathcal{D}(u)\in\mathcal{F}(u)$.
		\end{enumerate}
		Moreover, $R(u)^{1/2}$ admits a natural geometric interpretation: it represents the distance from $\mathcal{D}(u)$ to the admissible set $\mathcal{F}(u)$, i.e., the \emph{infimal} correction needed to satisfy the inclusion at $u$. In a general Banach space the infimum need not be attained; it is attained (and the correction is then a genuine minimum, realized by the metric projection) when $\mathcal{F}(u)$ is closed and convex and $Y$ is a Hilbert space, cf.\ Section~\ref{sec:projection}.
	\end{remark}
	
	\begin{remark}
		Since we are assuming that $\mathcal{F}$ is strict and closed-valued,
		when $\mathcal{F}(u)$ is convex and $Y$ is a Hilbert space, the
		projection $P_{\mathcal{F}(u)}\bigl(\mathcal{D}(u)\bigr)$ onto $\mathcal{F}(u)$ is uniquely
		defined (cf. Section~\ref{sec:projection}), and
		\[
		R(u)
		=
		\bigl\|\mathcal{D}(u)-P_{\mathcal{F}(u)}\bigl(\mathcal{D}(u)\bigr)\bigr\|_Y^{2}.
		\]
		For a \emph{fixed} closed convex set $C\subset Y$, the map
		$y\mapsto\dist^2(y,C)$ is Fr\'echet differentiable and
		\[
		\nabla_y\dist^2(y,C)=2\bigl(y-P_C(y)\bigr).
		\]
		This projection-gradient identity is used in the numerical
		implementations. It differentiates with respect to the first argument
		while holding the set fixed. When $\mathcal{F}(u)$ depends on $u$,
		computing $\nabla R(u)$ also requires differentiating that state
		dependence (or rewriting it explicitly, as in the translated-set
		examples below).
	\end{remark}
	
	When $\mathcal{F}(u)=\{f(u)\}$ is a singleton,
	$R(u)=\|\mathcal{D}(u)-f(u)\|_Y^2$, which is the standard PINN residual.
	
	\subsection{A general consistency result}
	
	The basic closed-graph argument is the following: sufficiently strong
	convergence together with a vanishing distance residual forces the limit
	to satisfy the target inclusion.
	
	\begin{lemma}\label{lem:closed-graph}
		Let $X$ and $Y$ be Banach spaces and let $\mathcal{F}: X \rightrightarrows Y$ be a strict set-valued function with closed graph.
		Suppose that $x_n \to x$ in $X$, $y_n \to y$ in $Y$, and
		\begin{equation} \label{eq:closed-graph-dist}
			\dist(y_n, \mathcal{F}(x_n)) \to 0.
		\end{equation}
		Then $y \in \mathcal{F}(x)$.
	\end{lemma}
	
	\begin{proof}
		By the definition of the distance to a set,
		for each $n \in \mathbb{N}$ there exists $a_n \in \mathcal{F}(x_n)$ such that
		$$
		\|y_n - a_n\|_Y \leq \text{dist}(y_n, \mathcal{F}(x_n)) + \frac{1}{n}.
		$$
		From this and \eqref{eq:closed-graph-dist} we deduce that
		\[
		\|y_n - a_n\|_Y \to 0 \quad \text{as } n \to \infty.
		\]
		Consequently, since $y_n \to y$, we have
		\[
		a_n = y_n - (y_n - a_n) \to y \quad \text{as } n \to \infty.
		\]
		Now, since $a_n \in \mathcal{F}(x_n)$, we have $(x_n, a_n) \in \text{gr}(\mathcal{F})$ for all $n\in\mathbb{N}$.
		Since $\text{gr}(\mathcal{F})$ is closed, $x_n \to x$, and $a_n \to y$, it follows that $(x, y) \in \text{gr}(\mathcal{F})$,
		which gives $y \in \mathcal{F}(x)$.
	\end{proof}
	
	\begin{theorem}[General consistency]\label{thm:general-consistency}
		Let $X$, $Y$ be Banach spaces, assume that $\mathcal{D}:X\to Y$ is continuous,
		and let $\mathcal{F}:X\rightrightarrows Y$ be strict and closed (that is, with closed graph).
		Let $\{u_n\}_{n\in\mathbb{N}}\subset X$ satisfy $u_n\to u$ in $X$ and
		\[
		\dist(\mathcal{D}(u_n),\mathcal{F}(u_n))\to 0.
		\]
		Then
		$\mathcal{D}(u)\in\mathcal{F}(u)$.
	\end{theorem}
	
	\begin{proof}
		Since $\mathcal{D}$ is continuous, $\mathcal{D}(u_n)\to\mathcal{D}(u)$. The result then follows by applying Lemma~\ref{lem:closed-graph} with $x_n=u_n$ and $y_n=\mathcal{D}(u_n)$.
	\end{proof}
	
	Theorem~\ref{thm:general-consistency} requires strong convergence of both
	$u_n$ and $\mathcal{D}(u_n)$. In the ODE and PDE settings of
	Sections~\ref{subsec:consistency} and~\ref{sec:pde}, the differential
	operators converge only weakly. The corresponding consistency results
	therefore use compactness of the states together with weak convergence of
	the operators and lower semicontinuity of the distance-residual functional,
	in the spirit of the direct method of the calculus of variations.

	\section{DR-PINNs for Differential Inclusions}
	\label{sec:distance-residual}
	
	\subsection{Formulation}
	\label{sec:description-formulation}
	
	Consider the differential inclusion \eqref{eq:inclusion} with $F:[0,T]\times\mathbb{R}^d\rightrightarrows\mathbb{R}^d$ satisfying
	the hypotheses of Theorem~\ref{thm:existence}.  The classical PINN
	residual
	\[
	\dot{x}(t)-f(t,x(t))
	\]
	has no analogue when $F$ is set-valued and lacks a single-valued
	selection.  The natural replacement is the distance from the computed
	velocity to the admissible set:
	\begin{equation}\label{eq:inclusion-residual}
		R_F(t;\theta)
		:=
		\dist^2\bigl(\dot{x}_\theta(t),\,F(t,x_\theta(t))\bigr),
	\end{equation}
	where $x_\theta(t)$ is a neural-network surrogate and $\dot{x}_\theta(t)$
	is obtained by automatic differentiation, exactly as in classical PINNs.
	
	\begin{remark}
		The assumptions that $F(t,x)$ is nonempty and closed ensure that the
		distance is finite and attained: for every $v\in\mathbb{R}^d$ there
		exists $\Pi_{F(t,x)}(v)\in F(t,x)$ with
		\[
		\dist(v,F(t,x))=\bigl|v-\Pi_{F(t,x)}(v)\bigr|.
		\]
		Closedness gives the following exactness property:
		\[
		\dist(v,F(t,x))=0 \,\,\iff\,\, v\in F(t,x).
		\]
		When $F(t,x)$ is also convex,
		the projection $\Pi_{F(t,x)}(v)$ is unique (see Section~\ref{sec:projection}).
		The squared distance becomes
		\begin{equation}\label{loss:distance}
			\dist^2(v,F(t,x))=\bigl|v-\Pi_{F(t,x)}(v)\bigr|^2,
		\end{equation}
		and its gradient with respect to $v$ is
		\[
		\nabla_v\dist^2(v,F(t,x))=2(v-\Pi_{F(t,x)}(v)),
		\]
		which is well-defined and Lipschitz. This identity differentiates
		only with respect to $v$, with $F(t,x)$ held fixed. If the admissible
		set depends on the learned state $x_\theta$, the corresponding
		chain-rule term must be retained in the implementation; exact
		formulas for the practically relevant cases, and the limits of plain
		differentiability for general moving sets, are given in
		Remark~\ref{rem:moving-set-gradient} below.
	\end{remark}
	
	The DR-PINN loss for the differential inclusion \eqref{eq:inclusion} is
	\begin{equation}\label{eq:main-distance-loss}
		\mathcal{L}_F(\theta)
		=
		\frac{1}{N}\sum_{i=1}^{N}
		\dist^2\bigl(\dot{x}_\theta(t_i),\,F(t_i,x_\theta(t_i))\bigr),
	\end{equation}
	where $\{t_i\}_{i=1}^N\subset[0,T]$ are collocation points.  Boundary
	and initial conditions are incorporated additively, so that
	the total loss then takes the form
	\[
	\mathcal{L}(\theta)
	=
	\lambda_F\,\mathcal{L}_F(\theta)
	+
	\lambda_0\,|x_\theta(0)-x_0|^2,
	\]
	where $\lambda_F,\lambda_0>0$ are user-specified weights.
	Given a neural network $N_\theta(t)$ we set
	\[
	x_\theta(t)=x_0+t\,N_\theta(t),
	\]
	which enforces $x_\theta(0)=x_0$, so that only the inclusion residual
	remains in the loss \eqref{eq:main-distance-loss}, that is, we can take
	$\mathcal{L}(\theta) = \mathcal{L}_F(\theta)$,
	which is more convenient for numerical approximations.
	Alternatively, the
	initial constraint can be penalized as a distance to a singleton.

	\subsection{Computing the Projection onto a General Convex Set}
	\label{sec:projection-extreme-points}
	
	To evaluate the loss \eqref{eq:main-distance-loss} one must compute
	the projection onto $F(t,x)$.  For simple sets
	(intervals, balls, boxes) this is explicit.  For a large and practically
	relevant class of admissible sets --- those that are, or can be
	approximated by, convex hulls of finitely many vertices --- the
	projection is computable by a small quadratic program, as we now
	describe.
	
	Suppose that $F(t,x)=\operatorname{conv}(V(t,x))$, where $\operatorname{conv}$ denotes the convex hull and
	\[
	V(t,x)=\{v_1(t,x),\dots,v_K(t,x)\}\subset\mathbb{R}^d.
	\]
	This covers
	any set of the controlled form $F(t,x)=\{f(t,x,u):u\in U\}$ with $U$ a
	polytope, \emph{provided} $f(t,x,\cdot)$ is affine in $u$: the image of
	a polytope under an affine map is again a polytope, with vertices among
	the images of the vertices. For a general nonlinear $f$ the image of a
	polytope need not be a polytope, nor even convex, and the construction
	below does not apply verbatim. The framework also covers
	any convex set discretized by a finite point cloud.
	For every $(t,x)\in[0,T]\times\mathbb{R}^d$, the projection of
	$v\in\mathbb{R}^d$ onto
	$F=\operatorname{conv}(V(t,x))$ satisfies
	\begin{equation}\label{eq:projection-as-qp}
		\Pi_F(v)
		=
		\sum_{k=1}^K\lambda_k^\ast v_k(t,x),
		\qquad
		\lambda^\ast
		\in
		\operatorname*{argmin}_{\lambda\in\Delta_K}
		\Bigl|v-\sum_{k=1}^K\lambda_k v_k(t,x)\Bigr|^2,
	\end{equation}
	where $\Delta_K=\{\lambda\in\mathbb{R}^K:\lambda_k\ge0,\;\sum_k\lambda_k=1\}$
	is the probability simplex.  The projected point $\Pi_F(v)$ is unique,
	since the set is convex, but its barycentric representation
	$\lambda^\ast$ need not be: any minimizer of the quadratic program
	yields the same projected point.  Problem \eqref{eq:projection-as-qp} is a
	small convex quadratic program with $K$ variables, one equality
	constraint, and non-negativity bounds.  It can be solved efficiently, even
	inside the training loop, by standard active-set or Frank--Wolfe-type
	methods.
	
	This reformulation shifts the difficulty from ``compute the projection
	onto a general convex set'' to ``produce a vertex set $V(t,x)$ whose
	convex hull approximates $F(t,x)$''.  When $V(t,x)$ is
	redundant, it is advantageous to first reduce it to its extreme points
	using a convex-hull algorithm such as Quickhull \cite{barber1996}.
	Quickhull computes the extreme points $\widetilde{V}(t,x)\subset V(t,x)$
	efficiently in practice (in output-sensitive time in low dimensions
	\cite{barber1996}) and satisfies
	$\operatorname{conv}(\widetilde{V}(t,x))=\operatorname{conv}(V(t,x))=F(t,x)$,
	while typically having far fewer points, which keeps the quadratic
	program \eqref{eq:projection-as-qp} small.
	
	Two regimes are worth distinguishing:
	\begin{enumerate}
		\item $F(t,x)$ is constant, or depending only on $t$.
		The extreme-point set $\widetilde{V}(t)$ can be precomputed
		\emph{offline} on a grid of times.  During training only the small
		quadratic program \eqref{eq:projection-as-qp} is solved online.
		\item $F(t,x)$ depends on the state $x_\theta(t)$.
If $V(t,x)=f(t,x)+A(t,x)V_0$ for a fixed reference set $V_0$
(independent of $(t,x)$), a convex-hull reduction
$\widetilde{V}_0$ of $V_0$ can still be computed once offline.
Indeed, for every fixed $(t,x)$, $A(t,x)$ is simply a linear
map, and
\begin{align*}
\operatorname{conv}\big(A(t,x)\widetilde{V}_0\big)
& =A(t,x)\operatorname{conv}(\widetilde{V}_0)
=A(t,x)\operatorname{conv}(V_0)
\\ &
=\operatorname{conv}\big(A(t,x)V_0\big),
\end{align*}
an identity valid for \emph{every} linear map.
Only the reduced set needs to be transformed --- by the
current $A(t,x)$ and $f(t,x)$ --- at each collocation point
during training.

In Section~\ref{subsec:consistency} below, we impose hypotheses
\ref{H1}--\ref{H3} to prove the consistency of the distance residual.
What is required so that $F$ satisfies these hypotheses is regularity
and growth of $A(\cdot,\cdot)$ and $f(\cdot,\cdot)$. Precisely,
$A(\cdot,x)$ and $f(\cdot,x)$ are measurable in
$t$ for every $x$; $A(t,\cdot)$ and $f(t,\cdot)$ are continuous for
a.e.\ $t$ (which already yields Hausdorff continuity, hence
upper semicontinuity, of $x\mapsto F(t,x)$, since
$d_H(A_1S,A_2S)\le\|A_1-A_2\|\sup_{w\in S}|w|$ for any bounded
$S$); and at most linear growth in $x$,
\[
|f(t,x)|+\|A(t,x)\|\,R_0\le m(t)(1+|x|),
\qquad R_0:=\sup_{w\in V_0}|w|,
\]
for some $m\in L^2(0,T)$, so that $F$ inherits the
linear-growth bound \ref{H3}. Constant $A(t,x)$ (as in
Section~\ref{subsec:linear-control}, where $B$ plays this
role) or $A(t,x)$ affine in $x$ are natural sufficient special
cases; a genuinely nonlinear dependence
$(t,x)\mapsto A(t,x)$ is equally admissible provided this
growth bound holds.
\end{enumerate}
	
	Quickhull is used only as a preprocessing step: it removes redundant
	points before the quadratic program is solved. Combined with
	\eqref{eq:projection-as-qp}, this keeps the evaluation of
	$\dist^2(v,F(t,x))$ practical for convex sets represented by finite point
	clouds.
	
	\begin{remark}[Differentiating the residual when the set depends on the
		state]\label{rem:moving-set-gradient}
		When the admissible set depends on the learned state, that dependence
		also enters the derivative. The two cases used in this paper are as
		follows.
		
		\emph{(i) Translations of a fixed set.} Suppose that
		\[
		F(t,x)=c(t,x)+C(t),
		\]
		where $C(t)\subset\mathbb{R}^d$ is a nonempty closed convex set not
		depending on $x$ and $c(\cdot,\cdot)$ is $C^1$ in $x$. Since the
		distance is translation invariant,
		\[
		r(t,x,v):=\dist^2\bigl(v,c(t,x)+C(t)\bigr)
		=\dist^2\bigl(v-c(t,x),C(t)\bigr),
		\]
		and the set on the right-hand side no longer moves with $x$.
		Writing
		\[
		q:=\bigl(v-c(t,x)\bigr)-\Pi_{C(t)}\bigl(v-c(t,x)\bigr),
		\]
		the Fr\'echet differentiability of $y\mapsto\dist^2(y,C(t))$ (whose
		gradient is $2(y-\Pi_{C(t)}y)$, cf.\ Section~\ref{sec:principle})
		combined with the chain rule gives the \emph{exact} gradients
		\[
		\nabla_v r=2q,
		\qquad
		\nabla_x r=-2\,D_xc(t,x)^{\!\top}q,
		\]
		so $r$ is genuinely $C^1$ in $(x,v)$. In an automatic-differentiation
		implementation, these gradients are obtained by evaluating the
		projection numerically and holding the \emph{projected point}
		constant in the backward pass, while differentiating through the
		argument $v-c(t,x)$. This covers the linear control example of
		Section~\ref{subsec:linear-control}, where
		$F(t,x)=Ax+BU$, $c(t,x)=Ax$, $D_xc=A$, and the ellipsoidal example
		of Section~\ref{subsec:rotating-ellipsoidal}, where $c(t,x)=g(t,x)$.
		
		\emph{(ii) General moving polytopes.} For
		$F(t,x)=\operatorname{conv}\{v_1(t,x),\dots,v_K(t,x)\}$ with
		$v_k(t,\cdot)\in C^1$, the residual is the value function of the
		quadratic program \eqref{eq:projection-as-qp},
		\[
		r(t,x,v)
		=\min_{\lambda\in\Delta_K}
		\Bigl|v-\sum_{k=1}^K\lambda_kv_k(t,x)\Bigr|^2 .
		\]
		If the minimizer $\lambda^\ast=\lambda^\ast(t,x,v)$ is \emph{unique}
		in a neighborhood of the point of interest, then, by Danskin's
		theorem for parametric minimization over a compact set, $r$ is
		differentiable there, with
		\[
		\nabla_vr=2e,
		\qquad
		\nabla_xr=-2\sum_{k=1}^K\lambda_k^\ast\,
		D_xv_k(t,x)^{\!\top}e,
		\qquad
		e:=v-\sum_{k=1}^K\lambda_k^\ast v_k(t,x).
		\]
		If $\lambda^\ast$ is \emph{not} unique (the projected point is
		always unique, but its barycentric representation need not be, cf.\
		the discussion after \eqref{eq:projection-as-qp}), the value function
		is in general only directionally differentiable in $x$, rather than
		$C^1$, so directional derivatives or Clarke subgradients are needed. The plain-gradient statements used here are therefore restricted to fixed sets, sets
		depending only on $t$, translations as in (i), and moving polytopes for
		which the inner quadratic program has a unique, locally stable
		minimizer.
	\end{remark}
	
	\subsection{Consistency}
	\label{subsec:consistency}
	
	We next prove the consistency of the distance residual, as stated in Theorem~\ref{thm:consistency} below. Since solutions of an inclusion need not be unique,
	the result is a compactness/lower-semicontinuity statement rather than a
	Gr\"onwall-type error estimate for a distinguished solution.
	
	We work under the hypotheses of Theorem~\ref{thm:existence}, with two
	mild strengthenings: $m\in L^2(0,T)$, and joint (product)
	measurability of $F$ in place of measurability in $t$ alone
	(see Remark~\ref{rem:H2-product} below):
	
	\begin{enumerate}[label=(H\arabic*)]
		\item \label{H1} $F$ is strict, closed-valued, and convex-valued.
		\item \label{H2} $F$ is \emph{product measurable} (or jointly $\mathscr{L}([0,T])\otimes\mathcal{B}(\mathbb{R}^d)$-measurable), i.e.,
		\[
		\begin{gathered}
		\bigl\{(t,x)\in[0,T]\times\mathbb{R}^d:F(t,x)\cap
		U\neq\varnothing\bigr\}
		\in
		\mathscr{L}([0,T])\otimes\mathcal{B}(\mathbb{R}^d) \\
		\text{for every open }U\subset\mathbb{R}^d,
		\end{gathered}
		\]
		where $\mathcal{L}([0,T])\otimes\mathcal{B}(\mathbb{R}^
d)$ is the product measure space of the Lebesgue $\sigma$-field $\mathscr{L}(
[0,T])$ and the Borel $\sigma$-field $\mathcal{B}(\mathbb{R}^d)$, and $x\mapsto F(t,x)$ is upper
		semicontinuous for \emph{a.e.} $t\in[0,T]$.
		\item \label{H3} There exists $m\in L^2(0,T)$, $m\ge0$, such that
		\[
		\sup_{v\in F(t,x)}|v|\le m(t)(1+|x|) \qquad
		\text{ for a.e. } t\in[0,T] \text{ and every } x\in\mathbb{R}^d.
		\]
	\end{enumerate}
	
	\begin{remark}\label{rem:H2-product}
	
		Fixing $x\in\mathbb{R}^d$ in the product measurability condition in \ref{H2} shows that $t\mapsto F(t,x)$ is Lebesgue measurable for every $x\in\mathbb{R}^d$, hence \ref{H2} implies the
		measurability hypothesis of Theorem~\ref{thm:existence}. Product
		measurability is a mild requirement: it holds, in particular,
		whenever $F$ is jointly upper semicontinuous on
		$[0,T]\times\mathbb{R}^d$ (then $\{F\cap U\neq\varnothing\}$ is an
		$F_\sigma$ set for every open $U$), and whenever $F$ is measurable in
		$t$ and \emph{continuous} in $x$ (Carath\'eodory multifunctions);
		both cases cover all the admissible sets used in
		Section~\ref{sec:numerics}. This strengthening is needed because
		measurability in $t$ together with upper semicontinuity in $x$ does not,
		in general, imply superpositional measurability; see
		\cite{zygmunt_1992}. Product measurability in \ref{H2} supplies the
		measurability used in the proof of Theorem~\ref{thm:consistency}.
	\end{remark}
	
	\begin{definition}\label{def:admissible-class}
		Define the admissible space
		\[
		\mathcal{A}:= \left\{ x \in W^{1,2}(0,T; \mathbb{R}^d) : x(0) = x_0 \right\},
		\]
		where $W^{1,2}(0,T; \mathbb{R}^d)$ is the Sobolev space of absolutely
		continuous functions with weak derivative in $L^2(0,T; \mathbb{R}^d)$.
		For $x\in\mathcal{A}$, define the \emph{distance-residual functional}
		\begin{equation}\label{eq:cont-functional}
			\mathcal{J}(x)
			:=
			\int_0^T\dist^2\bigl(\dot{x}(t),\,F(t,x(t))\bigr)\,dt.
		\end{equation}
	\end{definition}
	
	Every trajectory produced by the hard-constraint ansatz
	$x_\theta(t)=x_0+tN_\theta(t)$, with $N_\theta$ Lipschitz on $[0,T]$
	(as for the tanh networks used below), belongs to $\mathcal{A}$. Finite
	networks represent only a subset of $\mathcal{A}$.
	
	\begin{lemma}\label{lem:normal-integrand}
		Let $F:[0,T]\times\mathbb{R}^d\rightrightarrows\mathbb{R}^d$ be a set-valued function satisfying hypothesis {\rm\ref{H1}} and {\rm\ref{H2}}.
		Consider the function $g:[0,T]\times \mathbb{R}^d \times \mathbb{R}^d \to \mathbb{R}$ defined as
		\[
		g(t, x, v) := \operatorname{dist}^2(v, F(t, x)).
		\]
		Then, $g$ is a normal convex integrand, that is:
		\begin{enumerate}
			\item $g(t, x, \cdot)$ is convex in $v$ for a.e. $t \in [0,T]$ and every $x\in\mathbb{R}^d$.
			\item $g(t, \cdot, \cdot)$ is jointly lower semicontinuous in $(x, v)$ for a.e. $t \in [0,T]$.
			\item $g$ is jointly
			$\mathscr{L}([0,T])\otimes\mathcal{B}(\mathbb{R}^d\times\mathbb{R}^d)$-measurable
			on $[0,T]\times\mathbb{R}^d\times\mathbb{R}^d$.
		\end{enumerate}
	\end{lemma}
	\begin{proof}
		For $(a)$,
		fix $x \in \mathbb{R}^d$. For $v_1, v_2 \in \mathbb{R}^d$ and $\lambda \in [0,1]$,
		since $F(t,x)$ is closed and convex, the projections onto $F(t,x)$ exist and are unique:
		\[
		w_1 := \Pi_{F(t,x)} (v_1),  \quad
		w_2 := \Pi_{F(t,x)} (v_2).
		\]
		As $F(t,x)$ is convex, $\overline{w} := \lambda w_1 + (1-\lambda) w_2 \in F(t,x)$. Consequently,
		\begin{align}\label{eq:normal-integrand-inproof1}
			\dist^2(\lambda v_1 + (1-\lambda)v_2, F(t,x))
			& \le \big|\lambda v_1 + (1-\lambda)v_2 - \overline{w}\big|^2
			\nonumber\\ &
			= \left|\lambda(v_1 - w_1) + (1-\lambda)(v_2 - w_2)\right|^2.
		\end{align}
		On the other hand, since the squared Euclidean norm is convex,
		\begin{equation}\label{eq:normal-integrand-inproof2}
			|\lambda(v_1 - w_1) + (1-\lambda)(v_2 - w_2)|^2
			\le \lambda |v_1 - w_1|^2 + (1-\lambda) |v_2 - w_2|^2.
		\end{equation}
		Gathering \eqref{eq:normal-integrand-inproof1} and \eqref{eq:normal-integrand-inproof2},
		\begin{align*}
			\dist^2(\lambda v_1 + & (1-\lambda)v_2, F(t,x))
			\le \lambda |v_1 - w_1|^2 + (1-\lambda)|v_2 - w_2|^2
			\\ &
			= \lambda \dist^2(v_1, F(t,x)) + (1-\lambda) \dist^2(v_2, F(t,x)).
		\end{align*}
		We conclude that $g(t, x, \cdot)$ is convex in $v$.
		
		For $(b)$,
		we start by proving lower semicontinuity in $x$, with $v\in \mathbb{R}^d$ fixed.
		Suppose, for the sake of contradiction, that there is a sequence $\{x_n\}_{n\in\mathbb{N}}$ such that $x_n \to x \in \mathbb{R}^d$ and
		\[
		\liminf_{n \to \infty} \dist(v, F(t, x_n)) = L < d := \dist(v, F(t, x)).
		\]
		Choose $d' \in (L, d)$. Since $\dist(v, F(t,x)) = d > d'$, all points in $F(t,x)$
		lie outside the closed ball $\overline{B}(v, d')$, that is,
		\[
		F(t, x) \subset U := \mathbb{R}^d \setminus \overline{B}(v, d'),
		\]
		where $U$ is open.
		By upper semicontinuity of $F(t, \cdot)$ at $x$ (hypothesis \ref{H2}),
		there exists a neighborhood $V$ of $x$ such that
		\[
		F(t, x') \subset U \quad \text{for all } x' \in V.
		\]
		This means that $|w - v| > d'$ for all $w \in F(t, x')$ and $x'\in V$. Hence,
		\begin{equation}\label{eq:normal-integrand-inproof3}
			\dist(v, F(t, x')) \ge d' \quad \text{ for all } x'\in V.
		\end{equation}
		Since $x_n \to x$ and $V$ is a neighborhood of $x$, there exists $N\geq 1$ such $x_n \in V$ whenever $n > N$. From this and \eqref{eq:normal-integrand-inproof3}, we have
		\[
		\dist(v, F(t, x_n)) \ge d' \quad \text{ for all } n > N.
		\]
		Hence $\liminf_{n\to\infty}\dist(v, F(t, x_n))\ge d' > L$, which contradicts the definition of $L$ and finishes the proof of the lower semicontinuity in $x$ whenever $v\in\mathbb{R}^d$ is fixed.
		
		Next, for a fixed $x\in\mathbb{R}^d$, one can verify that the map $v \mapsto \text{dist}(v, F(t, x))$ is $1$-Lipschitz continuous in $v$, that is,
		\begin{equation}\label{eq:normal-integrand-inproof4}
			|\dist(v, F(t,x)) - \dist(v', F(t,x))| \le |v - v'|
			\quad \text{ for each } v,v'\in\mathbb{R}^d.
		\end{equation}
		Fix $\{x_n\}_{n\in\mathbb{N}}\subset\mathbb{R}^d$
		and $\{v_n\}_{n\in\mathbb{N}}\subset\mathbb{R}^d$
		such that $(x_n, v_n) \rightarrow (x, v)$ as $n \rightarrow \infty.$ From \eqref{eq:normal-integrand-inproof4}, we have
		\begin{equation}
			\text{dist}(v_n, F(t, x_n)) \ge \text{dist}(v, F(t, x_n)) - |v_n - v| \quad \text{ for all } n\in\mathbb{N}.
		\end{equation}
		Taking the $\liminf_{n \rightarrow \infty}$ on both sides, the term $|v_n - v|$ vanishes and the lower semicontinuity in $x$ for the first term on the right-hand side gives
		\begin{equation}
			\liminf_{n \rightarrow \infty} \text{dist}(v_n, F(t, x_n)) \ge \text{dist}(v, F(t, x)).
		\end{equation}
		The proof of $(b)$ is completed by squaring both sides.
		
		For $(c)$, fix first $v_0\in\mathbb{R}^d$ and
		$c>0$. Then
		\begin{align*}
		\bigl\{(t,x):&\dist(v_0,F(t,x))<c\bigr\}
		\\
		& =
		\bigl\{(t,x):F(t,x)\cap B(v_0,c)\neq\varnothing\bigr\}
		\in
		\mathscr{L}([0,T])\otimes\mathcal{B}(\mathbb{R}^d)
		\end{align*}
		by the product measurability required in \ref{H2}, so
		$(t,x)\mapsto\dist(v_0,F(t,x))$ is jointly measurable for each fixed
		$v_0$. Combining this with the $1$-Lipschitz continuity
		\eqref{eq:normal-integrand-inproof4} in $v$ and with the exact
		representation
		\[
		\dist(v,F(t,x))
		=
		\inf_{q\in\mathbb{Q}^d}
		\bigl[\dist(q,F(t,x))+|v-q|\bigr]
		\]
		(the inequality ``$\le$'' follows from the triangle inequality, and
		``$\ge$'' follows by letting $q\to v$), we express
		$(t,x,v)\mapsto\dist(v,F(t,x))$ as a countable infimum of jointly
		measurable functions. Hence $g(t,x,v)=\operatorname{dist}^2(v, F(t, x))$ is jointly measurable, which
		proves $(c)$.
	\end{proof}

	The next lemma settles the measurability of the superposed
	multifunction and of the associated metric projection appearing in the
	proof of Theorem~\ref{thm:consistency}.
	
	\begin{lemma}[Measurability of superpositions and projections]
		\label{lem:superposition-measurable}
		Let $F:[0,T]\times\mathbb{R}^d\rightrightarrows\mathbb{R}^d$ satisfy
		{\rm \ref{H1}--\ref{H2}}. Then:
		\begin{enumerate}
			\item For every Lebesgue measurable
			$y:[0,T]\to\mathbb{R}^d$, the multifunction
			$t\mapsto F(t,y(t))$ is Lebesgue measurable, strict, closed-valued, and convex-valued.
			\item I f, in addition, $v:[0,T]\to\mathbb{R}^d$ is measurable,
			then the metric projection
			$t\mapsto \Pi_{F(t,y(t))}\bigl(v(t)\bigr)$ is measurable.
		\end{enumerate}
	\end{lemma}
	
	\begin{proof}
		(a) The map $t\mapsto(t,y(t))$ is measurable from
		$([0,T],\mathscr{L}([0,T]))$ into
		$([0,T]\times\mathbb{R}^d,
		\mathscr{L}([0,T])\otimes\mathcal{B}(\mathbb{R}^d))$, since the
		preimage of every product set $A\times B$ is
		$A\cap y^{-1}(B)\in\mathscr{L}([0,T])$ and such products generate
		the product $\sigma$-field. Hence, for every open
		$U\subset\mathbb{R}^d$,
		\[
		\bigl\{t\in[0,T]:F(t,y(t))\cap U\neq\varnothing\bigr\}
		=
		\bigl\{t:(t,y(t))\in
		\{F\cap U\neq\varnothing\}\bigr\}
		\in\mathscr{L}([0,T])
		\]
		by the product measurability in \ref{H2}. The values are
		nonempty, closed, and convex by \ref{H1}.
		
		(b) Since $t\mapsto F(t,y(t))$ is measurable with closed convex
		values and $v(\cdot)$ is measurable, the metric projection is
		measurable by the measurable-projection theorem
		\cite[Theorem~8.2.11]{aubin_frankowska_2009}.
	\end{proof}
	
	Both consistency proofs use the same lower-semicontinuity principle:
	strong convergence in the state variable together with weak convergence
	in the variable in which the integrand is convex.
	
	\begin{lemma}[Strong--weak lower semicontinuity]
		\label{lem:strong-weak-lsc}
		Let $Z\subset\mathbb{R}^N$ be bounded and measurable, and let
		$g:Z\times\mathbb{R}^m\times\mathbb{R}^k\to[0,\infty]$ be a normal
		integrand, i.e., jointly
		$\mathscr{L}(Z)\otimes
		\mathcal{B}(\mathbb{R}^m\times\mathbb{R}^k)$-measurable
		and such that $g(z,\cdot,\cdot)$ is lower semicontinuous on
		$\mathbb{R}^m\times\mathbb{R}^k$ for a.e.\ $z\in Z$. Assume in
		addition that $g(z,x,\cdot)$ is convex for a.e.\ $z\in Z$ and every
		$x\in\mathbb{R}^m$ and that
		\[
		\int_Z g\bigl(z,x_0(z),v_0(z)\bigr)\,dz < \infty
		\]
		for some $x_0\in L^2(Z;\mathbb{R}^m)$, $v_0\in L^2(Z;\mathbb{R}^k)$. If
		\[
		x_n\to x \ \text{ strongly in } L^2(Z;\mathbb{R}^m),
		\qquad
		v_n\rightharpoonup v \ \text{ weakly in } L^2(Z;\mathbb{R}^k),
		\]
		then
		\[
		\int_Z g\bigl(z,x(z),v(z)\bigr)\,dz
		\le
		\liminf_{n\to\infty}\int_Z g\bigl(z,x_n(z),v_n(z)\bigr)\,dz.
		\]
	\end{lemma}
	
	\begin{proof}
The result follows directly from \cite[Theorem~3]{ioffe_1977} since $g$ is nonnegative.
	\end{proof}
	
	\begin{theorem}[Consistency of the distance residual] \label{thm:consistency}
		Let $F:[0,T]\times\mathbb{R}^d\rightrightarrows\mathbb{R}^d$ be a set-valued function satisfying hypotheses {\rm \ref{H1}--\ref{H3}}.
		Consider the differential inclusion \eqref{eq:inclusion}, along with its
		solution set $\mathcal{S}(x_0)$ defined as in \eqref{eq:solution-DI}.
		Then:
		\begin{enumerate}
			\item[\rm(a)] (Exactness) For every $x \in \mathcal{A}$, $\mathcal{J}(x) = 0$ if and only if $x \in \mathcal{S}(x_0)$.
			\item[\rm(b)] (Asymptotic consistency) If $\{x_n\}_{n\in\mathbb{N}} \subset \mathcal{A}$ satisfies $\mathcal{J}(x_n) \to 0$, then there exists a subsequence of $\{x_n\}_{n\in\mathbb{N}}$, for which we use the same labeling, and $x^* \in \mathcal{A}$ such that
			\[ x_n \to x^* \text{ uniformly on }[0,T],
			\qquad
			\dot{x}_n \rightharpoonup \dot{x}^* \text{ weakly in } L^2(0,T; \mathbb{R}^d),
			\]
		 and
		 \[
		 x^* \in \mathcal{S}(x_0).
		 \]
		\end{enumerate}
	\end{theorem}
	
	\begin{proof}
		Part $(a)$ follows immediately from the exactness property of the distance. To prove $(b)$ we start by setting $r_n(t) := \operatorname{dist}(\dot{x}_n(t), F(t, x_n(t)))$ for every $n\in\mathbb{N}$ so that
		\[
		\int_0^T r_n^2(t) \, dt = \mathcal{J}(x_n) =: \varepsilon_n \to 0.
		\]
		Let $\varepsilon_0 := \sup_n \varepsilon_n < \infty$. Since  $F(t, x_n(t))$ is nonempty, closed,  and convex, the projection
		\[
		w_n(t) = \Pi_{F(t, x_n(t))}(\dot{x}_n(t)) \in F(t, x_n(t))
		\]
		is uniquely defined and the minimum is attained, that is,
		\[
		|\dot{x}_n(t) - w_n(t)| =\min_{w\in F(t, x_n(t))} |\dot{x}_n(t) - w| =r_n(t).
		\]
		Moreover, $t\mapsto w_n(t)$ is (Lebesgue) measurable by
		Lemma~\ref{lem:superposition-measurable}, applied with $y=x_n$ and
		$v=\dot{x}_n$.
		By the triangle inequality and \ref{H3}, for a.e. $t\in[0,T]$,
		\begin{equation}\label{eq:consistency-inproof1}
			|\dot{x}_n(t)| \le |\dot{x}_n(t) - w_n(t)|+|w_n(t)|= r_n(t) + m(t)(1 + |x_n(t)|).
		\end{equation}
		On the other hand, applying Cauchy–Schwarz inequality to $r_n(t)$, we obtain
		\[
		\int_0^t r_n(s) \, ds  \le \Big(\int_0^T ds\Big)^{1/2}\Big(\int_0^T r^2_n(s)ds\Big)^{1/2}
		\leq  \sqrt{T \varepsilon_n}\le  \sqrt{T \varepsilon_0} =: R_0.
		\]
		Then, integrating \eqref{eq:consistency-inproof1} we deduce that
		\begin{align*}
			|x_n(t)| & \le |x_0|  + |x_n(t)-x_0|
			\leq |x_0| + \int_0^t |\dot{x}_n(s)|\, ds
			\\ &
			\le |x_0| + R_0 + \int_0^t m(s)(1 + |x_n(s)|) \, ds.
		\end{align*}
		From this and Gr\"{o}nwall's inequality (see, e.g., \cite[pp.\,12--14]{pachpatte1998}) we obtain
		\begin{equation}\label{eq:consistency-inproof2}
			1 + |x_n(t)| \le (1 + |x_0| + R_0 )e^{\|m\|_{L^1(0,T)}}:=C_0
		\end{equation}
		for every $t \in [0,T]$ and $n\in\mathbb{N}$.
		Substituting \eqref{eq:consistency-inproof2} back into \eqref{eq:consistency-inproof1}, squaring, and integrating,
		\[
		\int_0^T |\dot{x}_n(t)|^2 \, dt \le \int_0^T r_n^2(t) \, dt + 2C_0 \int_0^T r_n(t) m(t) \, dt
		+ C_0^2 \int_0^T m^2(t) \, dt.
		\]
		The second term in the right-hand side can be handled by
		applying Cauchy--Schwarz to obtain
		\[
		\int_0^T r_n(t) m(t) \, dt \le \Big(\int_0^T r_n^2(t) \, dt\Big)^{1/2} \cdot \Big(\int_0^T m^2(t) \, dt \Big)^{1/2}
		= \sqrt{\varepsilon_n} \, \|m\|_{L^2(0,T)}.
		\]
		Therefore,
		\begin{align*}
			\int_0^T |\dot{x}_n(t)|^2 \, dt
			&\le \varepsilon_n + 2C_0\sqrt{\varepsilon_n} \, \|m\|_{L^2(0,T)} + C_0^2 \|m\|_{L^2(0,T)}^2.
		\end{align*}
		Since $\varepsilon_0 := \sup_n \varepsilon_n < \infty$, for every $n\in\mathbb{N}$,
		\begin{equation}\label{eq:consistency-inproof3}
			\|\dot{x}_n\|^2_{L^2(0,T)} \le \varepsilon_0 + 2 C_0 \sqrt{\varepsilon_0} \|m\|_{L^2(0,T)} +  C^2 _0\|m\|^2_{L^2(0,T)} =: C_1.
		\end{equation}
		This shows that
		the sequence $\{\dot{x}_n\}_{n\in\mathbb{N}}$ is bounded in $L^2(0,T; \mathbb{R}^d)$.
		In particular, from \eqref{eq:consistency-inproof2} and \eqref{eq:consistency-inproof3} we deduce that
		the sequence $\{x_n\}_{n\in\mathbb{N}}$ is bounded in the Sobolev space $W^{1,2}(0,T; \mathbb{R}^d)$.
		
		Since, by \eqref{eq:consistency-inproof3}, the sequence $\{\dot{x}_n\}_{n\in\mathbb{N}}$ is bounded in the Hilbert space $L^2(0,T; \mathbb{R}^d)$, which is reflexive, \cite[Theorem 3.18]{brezis_2011} provides a subsequence  of $\{\dot{x}_n\}_{n\in\mathbb{N}}$, for which we use the same labeling, and $v^*\in L^2(0,T;\mathbb{R}^d)$ such that
		\[
		\dot{x}_n \rightharpoonup v^* \quad \text{weakly in } L^2(0,T; \mathbb{R}^d).
		\]
		By Cauchy--Schwarz inequality,
		\[
		|x_n(t) - x_n(s)| = \left|\int_s^t \dot{x}_n(\tau) \, d\tau\right|
		\le \|\dot{x}_n\|_{L^2(s,t)} \sqrt{|t - s|}
		\le \|\dot{x}_n\|_{L^2(0,T)} \sqrt{|t - s|}.
		\]
		From this and \eqref{eq:consistency-inproof3}, we have
		\[
		|x_n(t) - x_n(s)| \le \sqrt{C_1|t - s|},
		\]
		hence $x_n$ is uniformly equicontinuous.
		Since $x_n$ is also uniformly bounded as seen from \eqref{eq:consistency-inproof2},
		we may therefore apply the Arzelà–Ascoli theorem, which guarantees the existence of a subsequence, still denoted by $\{x_n\}_{n\in\mathbb{N}}$, such that
		\[
		x_n \to x^* \quad \text{uniformly on } [0,T].
		\]
		Then, passing to the limit in $x_n(t) = x_0 + \int_0^t \dot{x}_n(s) \, ds$ (the integral converges by the weak convergence $\dot{x}_n \rightharpoonup v^*$ in $L^2$, tested against $\mathbf{1}_{[0,t]}$) gives
		\[
		x^*(t) = x_0 + \int_0^t v^*(s) \, ds
		\]
		hence $x^* \in \mathcal{A}$ and $\dot{x}^* = v^*$ a.e.
		
		Finally, set $g(t, x, v) := \operatorname{dist}^2(v, F(t,x))$.
		By Lemma~\ref{lem:normal-integrand}, $g$ is a normal convex
		integrand. Since $x_n\to x^*$ uniformly on $[0,T]$, in particular
		$x_n\to x^*$ strongly in $L^2(0,T;\mathbb{R}^d)$, while
		$\dot{x}_n\rightharpoonup\dot{x}^*$ weakly in
		$L^2(0,T;\mathbb{R}^d)$. Lemma~\ref{lem:strong-weak-lsc}, applied
		with $Z=(0,T)$ and $m=k=d$, yields
		\[
		\mathcal{J}(x^*) \le \liminf_{n \to \infty} \mathcal{J}(x_n) = 0.
		\]
		As $\mathcal{J}$ is non-negative, we conclude that $\mathcal{J}(x^*) = 0$. By part (a), this means that $x^* \in \mathcal{S}(x_0)$.
	\end{proof}
	
	\begin{remark}
		Because solutions need not be unique, Theorem~\ref{thm:consistency}
		identifies a subsequential limit in $\mathcal{S}(x_0)$ but does not
		select a prescribed element of that set. Additional structure, such as
		monotonicity, a one-sided Lipschitz condition, or uniqueness, may yield
		full-sequence convergence, stronger convergence, or quantitative error
		estimates.
	\end{remark}

	\section{DR-PINNs for Partial Differential Inclusions}
	\label{sec:pde}
	
	Consider now the parabolic partial differential
	inclusion (\ref{eq:parabolic-inclusion-background}).
	
	\subsection{The Distance-Based PDE Loss}
	\label{sec:pde-loss}
	
	The classical PDE residual $\partial_t u_\theta-\Delta u_\theta-f(u_\theta)$
	has no analogue when $\Phi$ is set-valued.  The natural replacement is
	\begin{equation}\label{eq:pde-distance-loss}
		\mathcal{L}_{\mathrm{incl}}(\theta)
		=
		\frac{1}{N_r}
		\sum_{i=1}^{N_r}
		\dist^2\Bigl(
		\partial_t u_\theta(t_i,x_i)-\Delta u_\theta(t_i,x_i),\;
		\Phi(u_\theta(t_i,x_i))
		\Bigr),
	\end{equation}
	where $\{(t_i,x_i)\}_{i=1}^{N_r}$ are collocation points in $Q$.  As in
	the case of differential inclusions, $\mathcal{L}_{\mathrm{incl}}(\theta)=0$ if and only if the
	inclusion is satisfied at every collocation point.  The full DR-PINN loss
	incorporates initial and boundary data:
	\begin{align*}
		\mathcal{L}(\theta)
		&=
		\mathcal{L}_{\mathrm{incl}}(\theta)
		+
		\lambda_0\,\mathcal{L}_0(\theta)
		+
		\lambda_b\,\mathcal{L}_b(\theta),
	\end{align*}
	where
	\[
	\mathcal{L}_0(\theta)
	=
	\frac{1}{N_0}\sum_{j=1}^{N_0}\bigl|u_\theta(0,x_j^0)-u_0(x_j^0)\bigr|^2,
	\qquad
	\mathcal{L}_b(\theta)
	=
	\frac{1}{N_b}\sum_{k=1}^{N_b}\bigl|u_\theta(t_k^b,x_k^b)\bigr|^2.
	\]
	
	\subsection{Consistency}
	\label{sec:pde-consistency}
	For the consistency analysis, we assume that the set-valued reaction
	term $\Phi:\mathbb{R}\rightrightarrows\mathbb{R}$ satisfies the
	standing hypotheses \ref{P1}--\ref{P3} of Section~\ref{sec:basicsPDI};
	these are exactly the assumptions under which
	Theorem~\ref{thm:existence-pde} guarantees
	$\mathcal{S}(u_0)\neq\varnothing$, so that existence and consistency
	rely on the same set of hypotheses.


	\begin{definition}\label{def:admissible-pde}
		Let
		\[
		\mathcal{A}
		:=
		\bigl\{u\in\mathcal{W}(0,T):
		u(0,\cdot)=u_0\ \text{in }\Omega,\;
		u(t,\cdot)|_{\partial\Omega}=0\ \forall t\in[0,T]
		\bigr\}.
		\]
		For $u\in\mathcal{A}$, define the \emph{continuous distance-residual
			functional}
		\[
		\mathcal{J}(u)
		:=
		\iint_{Q}\dist^2\bigl(\partial_t u-\Delta u,\,\Phi(u)\bigr)\,dx\,dt.
		\]
	\end{definition}
	The pointwise-in-time boundary condition in the definition of
	$\mathcal{A}$ is in fact automatically satisfied: by the continuous
	embedding \eqref{eq:W-cont-embedding}, every $u\in\mathcal{W}(0,T)$
	admits a representative in $C([0,T];H^1_0(\Omega))$, so
	$u(t,\cdot)\in H_0^1(\Omega)$ --- and hence
	$u(t,\cdot)|_{\partial\Omega}=0$ in the trace sense --- for
	\emph{every} $t\in[0,T]$, even though $u(t,\cdot)\in H^2(\Omega)$
	only for a.e.\ $t$. Consequently,
	\[
	\mathcal{A}=\{u\in\mathcal{W}(0,T):u(0,\cdot)=u_0\}.
	\]
	The boundary condition is retained in the definition only for emphasis.
	
	We also note that $\mathcal{J}(u)<\infty$ for every
	$u\in\mathcal{A}$: by \ref{P3}, choosing for a.e.\ $(t,x)$ any
	$r_0(t,x)\in\Phi(u(t,x))$ gives
	\[
	\dist\bigl(\partial_t u-\Delta u,\,\Phi(u)\bigr)
	\le|\partial_t u-\Delta u-r_0|
	\le|\partial_t u-\Delta u|+m_0(1+|u|)
	\quad\text{a.e.\ on } Q;
	\]
	squaring and integrating, and using $\partial_t u-\Delta u\in L^2(Q)$
	and $u\in L^2(Q)$ (both immediate from $u\in\mathcal{W}(0,T)$),
	yields $\mathcal{J}(u)<\infty$.
\begin{lemma}[Normal convex integrand]\label{lem:normal-integrand-parabolic}
	Let $\Phi:\mathbb{R}\rightrightarrows\mathbb{R}$ satisfy hypotheses \ref{P1} and \ref{P2}.
	Define $g:\mathbb{R}\times\mathbb{R}\to[0,\infty)$ by
	\[
	g(s,z):=\operatorname{dist}^2\bigl(z,\Phi(s)\bigr).
	\]
	Then $g$ is a normal convex integrand, that is:
	\begin{enumerate}
		\item[(a)] $g(s,\cdot)$ is convex for every $s\in\mathbb{R}$;
		\item[(b)] $g$ is jointly lower semicontinuous on $\mathbb{R}\times\mathbb{R}$;
		\item[(c)] $g$ is Borel measurable; consequently, for every pair of measurable
		functions $u,v:Q\to\mathbb{R}$, the composition $(t,x)\mapsto g(u(t,x),v(t,x))$
		is Lebesgue measurable on $Q$.
	\end{enumerate}
\end{lemma}

\begin{proof}
	We treat the three assertions in order.
	
	\emph{Part (a).} Fix $s\in\mathbb{R}$. Since $\Phi(s)$ is nonempty, closed, and convex
	by \ref{P1}, the projection $\Pi_{\Phi(s)}$ is well defined and single-valued (see Section~\ref{sec:projection}). For
	$z_1,z_2\in\mathbb{R}$ and $\lambda\in[0,1]$, set
	$w_1:=\Pi_{\Phi(s)}(z_1)$, $w_2:=\Pi_{\Phi(s)}(z_2)$. By convexity of
	$\Phi(s)$, $w:=\lambda w_1+(1-\lambda)w_2\in\Phi(s)$, hence
	\begin{align*}
	\operatorname{dist}^2\bigl(\lambda z_1+(1-\lambda)z_2,\Phi(s)\bigr)
	& \le |\lambda z_1+(1-\lambda)z_2-w|^2
	\\
	& = |\lambda(z_1-w_1)+(1-\lambda)(z_2-w_2)|^2.
	\end{align*}
	By convexity of $t\mapsto t^2$,
	\[
	|\lambda(z_1-w_1)+(1-\lambda)(z_2-w_2)|^2
	\le \lambda(z_1-w_1)^2+(1-\lambda)(z_2-w_2)^2,
	\]
	and combining the two inequalities gives
	\[
	g(s,\lambda z_1+(1-\lambda)z_2)\le \lambda g(s,z_1)+(1-\lambda)g(s,z_2).
	\]
	
	\emph{Part (b).} We first show that $s\mapsto\operatorname{dist}(z,\Phi(s))$
	is lower semicontinuous for fixed $z\in\mathbb{R}$. Suppose, for contradiction,
	that $s_n\to s$ and
	\[
	\liminf_{n\to\infty}\operatorname{dist}(z,\Phi(s_n)) = L < d:=\operatorname{dist}(z,\Phi(s)).
	\]
	Choose $d'\in(L,d)$. Since $d>d'$, we have $\Phi(s)\subset U:=\mathbb{R}\setminus[z-d',z+d']$,
	an open set. By the upper semicontinuity of $\Phi$ at $s$ (hypothesis \ref{P2}),
	there exists a neighborhood $V$ of $s$ such that $\Phi(s')\subset U$ for all
	$s'\in V$. Hence $|w-z|>d'$ for every $w\in\Phi(s')$, $s'\in V$, so that
	\[
	\operatorname{dist}(z,\Phi(s'))\ge d' \qquad \text{for all } s'\in V.
	\]
	Since $s_n\to s$, there exists $N\in\mathbb{N}$ such that $s_n\in V$ for all
	$n>N$, whence $\operatorname{dist}(z,\Phi(s_n))\ge d'$ for $n>N$. This contradicts
	$\liminf_n \operatorname{dist}(z,\Phi(s_n)) = L < d'$, proving lower semicontinuity
	in $s$.
	
	On the other hand, for fixed $s\in\mathbb{R}$, the map
	$z\mapsto\operatorname{dist}(z,\Phi(s))$ is $1$-Lipschitz:
	\[
	\bigl|\operatorname{dist}(z,\Phi(s))-\operatorname{dist}(z',\Phi(s))\bigr|
	\le |z-z'| \qquad \text{for all } z,z'\in\mathbb{R}.
	\]
	Let $(s_n,z_n)\to(s,z)$. Then
	\[
	\operatorname{dist}(z_n,\Phi(s_n)) \ge \operatorname{dist}(z,\Phi(s_n)) - |z_n-z|.
	\]
	Taking $\liminf_{n\to\infty}$ on both sides, the term $|z_n-z|$ vanishes and
	the lower semicontinuity in $s$ established above gives
	\[
	\liminf_{n\to\infty}\operatorname{dist}(z_n,\Phi(s_n)) \ge \operatorname{dist}(z,\Phi(s)).
	\]
	Squaring both (nonnegative) sides yields $\liminf_n g(s_n,z_n)\ge g(s,z)$,
	which proves (b).
	
	\emph{Part (c).} A jointly lower semicontinuous function on a metric
	space has closed sublevel sets $\{(s,z):g(s,z)\le c\}$ for every
	$c\in\mathbb{R}$, hence $g$ is Borel measurable on $\mathbb{R}\times\mathbb{R}$.
	If $u,v:Q\to\mathbb{R}$ are Lebesgue measurable, then
	$(t,x)\mapsto(u(t,x),v(t,x))$ is measurable from $Q$ into $\mathbb{R}^2$, and
	the composition of a Borel measurable function with a Lebesgue measurable
	function is Lebesgue measurable. Therefore $(t,x)\mapsto g(u(t,x),v(t,x))$
	is measurable on $Q$.
\end{proof}
	
	\begin{theorem}[Consistency: parabolic case]
		\label{thm:consistency-parabolic}
		Assume that $\Phi:\mathbb{R}\rightrightarrows\mathbb{R}$
		satisfies {\rm \ref{P1}--\ref{P3}}. Then:
		
		\begin{itemize}
			\item[(a)] \emph{Exactness.} For every $u\in\mathcal{A}$,
			\[
			\mathcal{J}(u)=0
			\;\Longleftrightarrow\;
			\partial_tu(t,x)-\Delta u(t,x)
			\in\Phi(u(t,x))
			\quad\text{for a.e. }(t,x)\in Q.
			\]
			
			\item[(b)] \emph{Asymptotic consistency.} Let
			$\{u_n\}_{n\in\mathbb{N}}\subset\mathcal{A}$ satisfy
			\[
			\mathcal{J}(u_n)\longrightarrow 0.
			\]
			Then there exists a subsequence, for which we use the same labeling, and
			$u^\ast\in\mathcal{A}$ such that
			\[
			u_n\rightharpoonup u^\ast
			\text{ weakly in }\mathcal{W}(0,T),
			\quad
			u_n\longrightarrow u^\ast
			\text{ strongly in }
			C\bigl([0,T];L^2(\Omega)\bigr),
			\]
			and
			\[
			u^\ast\in\mathcal{S}(u_0).
			\]
		\end{itemize}
	\end{theorem}
	
	\begin{proof}
		Part~\textup{(a)} follows directly from the fact that
		$\Phi(u(t,x))$ is closed for almost every $(t,x)\in Q$.
		
		We now prove part~\textup{(b)}. Set
		\[
		z_n:=\partial_tu_n-\Delta u_n
		\]
		and
		\[
		\rho_n(t,x)
		:=
		\operatorname{dist}
		\bigl(z_n(t,x),\Phi(u_n(t,x))\bigr).
		\]
		Then
		\begin{equation}
		\|\rho_n\|_{L^2(Q)}^2
		=
		\mathcal{J}(u_n)
		\longrightarrow 0.
		\label{eq:rho-convergence}
		\end{equation}
		
		By \ref{P1}, the set $\Phi(s)$ is nonempty, closed, and convex
		for every $s\in\mathbb{R}$. Hence the metric projection onto
		$\Phi(s)$ is well defined and unique. Define
		\[
		w_n(t,x)
		:=
		\Pi_{\Phi(u_n(t,x))}
		\bigl(z_n(t,x)\bigr).
		\]
		The map $w_n$ is (Lebesgue) measurable on $Q$. Indeed, by
		\ref{P2}, for every closed $C\subset\mathbb{R}$ the set
		$\{s\in\mathbb{R}:\Phi(s)\subset\mathbb{R}\setminus C\}$ is open,
		so $\{s:\Phi(s)\cap C\neq\varnothing\}$ is closed; writing an open
		set $U\subset\mathbb{R}$ as a countable union of closed sets shows
		that $\{s:\Phi(s)\cap U\neq\varnothing\}$ is Borel for every open
		$U$, i.e., $\Phi$ is a Borel measurable multifunction with closed
		convex values. Consequently, $(t,x)\mapsto\Phi(u_n(t,x))$ is
		Lebesgue measurable on $Q$, as the composition of a Borel measurable
		multifunction with the measurable function $u_n$, and the metric
		projection of the measurable function $z_n$ onto it is measurable by
		the measurable-projection theorem for measurable, closed- and
		convex-valued multifunctions; see
		\cite[Theorem~8.2.11]{aubin_frankowska_2009}. It satisfies
		\[
		w_n(t,x)\in\Phi(u_n(t,x))
		\quad\text{for a.e. }(t,x)\in Q,
		\]
		together with
		\[
		|z_n(t,x)-w_n(t,x)|
		=
		\rho_n(t,x).
		\]
		Setting
		\[
		r_n:=z_n-w_n,
		\]
		we obtain
		\begin{equation}
		\partial_tu_n-\Delta u_n=w_n+r_n
		\quad\text{on }Q,
		\label{eq:pde-decomposition}
		\end{equation}
		and, by \eqref{eq:rho-convergence},
		\begin{equation}
		\|r_n\|_{L^2(Q)}^2
		=
		\mathcal{J}(u_n)
		\longrightarrow 0.
		\label{eq:rn-convergence}
		\end{equation}
		
		We first derive an $L^2$ energy estimate. Testing
		\eqref{eq:pde-decomposition} with $u_n(t)$ and using the homogeneous Dirichlet
		boundary condition, we obtain, for almost every $t\in(0,T)$,
		\begin{equation}
		\frac{1}{2}\frac{d}{dt}
		\|u_n(t)\|_{L^2(\Omega)}^2
		+
		\|\nabla u_n(t)\|_{L^2(\Omega)}^2
		=
		\int_\Omega
		\bigl(w_n(t,x)+r_n(t,x)\bigr)u_n(t,x)\,dx.
		\label{eq:energy-identity}
		\end{equation}
		By \ref{P3},
		\begin{equation}
		|w_n(t,x)|
		\leq m_0\bigl(1+|u_n(t,x)|\bigr)
		\quad\text{for a.e. }(t,x)\in Q.
		\label{eq:wn-growth}
		\end{equation}
		Consequently,
		\begin{equation}
		\|w_n(t)\|_{L^2(\Omega)}
		\leq
		m_0\Bigl(
		|\Omega|^{1/2}
		+
		\|u_n(t)\|_{L^2(\Omega)}
		\Bigr).
		\label{eq:wn-L2}
		\end{equation}
		Using the Cauchy--Schwarz and Young inequalities in
		\eqref{eq:energy-identity}, together with \eqref{eq:wn-L2}, we find a constant
		$C>0$, independent of $n$, such that
		\begin{equation}
		\frac{d}{dt}
		\|u_n(t)\|_{L^2(\Omega)}^2
		+
		2\|\nabla u_n(t)\|_{L^2(\Omega)}^2
		\leq
		C\Bigl(
		1+\|u_n(t)\|_{L^2(\Omega)}^2
		\Bigr)
		+
		\|r_n(t)\|_{L^2(\Omega)}^2.
		\label{eq:gronwall-ineq}
		\end{equation}
		Since $u_n(0)=u_0$ for every $n$, Gronwall's inequality and
		\eqref{eq:rn-convergence} yield
		\begin{equation}
		\sup_{n\in\mathbb{N}}
		\|u_n\|_{L^\infty(0,T;L^2(\Omega))}
		<\infty.
		\label{eq:Linf-bound}
		\end{equation}
		Integrating \eqref{eq:gronwall-ineq} over $(0,T)$ also gives
		\[
		\sup_{n\in\mathbb{N}}
		\|u_n\|_{L^2(0,T;H_0^1(\Omega))}
		<\infty.
		\]
		
		From \eqref{eq:wn-growth} and \eqref{eq:Linf-bound}, it follows that
		\[
		\sup_{n\in\mathbb{N}}
		\|w_n\|_{L^2(Q)}
		<\infty.		
		\]
		Together with \eqref{eq:rn-convergence}, this implies
		\begin{equation}
		\sup_{n\in\mathbb{N}}
		\|w_n+r_n\|_{L^2(Q)}
		<\infty.
		\label{eq:forcing-bound}
		\end{equation}
		
		We now apply the standard $L^2$ maximal-regularity estimate for the
		Di\-richlet heat equation to \eqref{eq:pde-decomposition}; see
		\cite[Section~7.1.3, Theorem~5]{Evans} for smooth domains, with the
		underlying $H^2$ elliptic estimate supplied by
		\cite[Theorem~3.2.1.2]{Grisvard1985} on convex domains and by
		\cite[Theorem~2.4.2.5]{Grisvard1985} on $C^{1,1}$ domains. Since
		$u_0\in H_0^1(\Omega)$ and $\partial\Omega$ is $C^{1,1}$ or $\Omega$ is convex (which covers the square domain used in Section~\ref{subsec:parabolic-PDI}), there exists a constant
		$C_{\mathrm{par}}>0$, independent of $n$, such that
		\[
		\|u_n\|_{\mathcal{W}(0,T)}
		\leq
		C_{\mathrm{par}}
		\left(
		\|w_n+r_n\|_{L^2(Q)}
		+
		\|u_0\|_{H_0^1(\Omega)}
		\right).
		\]
		By \eqref{eq:forcing-bound},
		\[
		\sup_{n\in\mathbb{N}}
		\|u_n\|_{\mathcal{W}(0,T)}
		<\infty.
		\]
		
		Since $\mathcal{W}(0,T)$ is reflexive, there exist a subsequence,
		not relabeled, and $u^\ast\in\mathcal{W}(0,T)$ such that
		\begin{equation}
		u_n\rightharpoonup u^\ast
		\quad\text{weakly in }\mathcal{W}(0,T).
		\label{eq:weak-convergence-W}
		\end{equation}
		By the compact parabolic embedding	$\mathcal{W}(0,T)
		\hookrightarrow\hookrightarrow
		C\bigl([0,T];L^2(\Omega)\bigr)$ of
		Lemma~\ref{lem:compact-embedding},
		\begin{equation}
		u_n\longrightarrow u^\ast
		\quad\text{strongly in }
		C\bigl([0,T];L^2(\Omega)\bigr).
		\label{eq:strong-convergence-C}
		\end{equation}
		In particular,
		\begin{equation}
		u_n\longrightarrow u^\ast
		\quad\text{strongly in }L^2(Q).
		\label{eq:strong-convergence-L2}
		\end{equation}
		The convergence in \eqref{eq:strong-convergence-C} and the identity
		$u_n(0)=u_0$ imply
		\[
		u^\ast(0)=u_0.
		\]
		Moreover, since
		\[
		u_n\in L^2\bigl(0,T;H^2(\Omega)\cap H_0^1(\Omega)\bigr)
		\]
		and the latter is a closed linear subspace, the weak limit satisfies
		the homogeneous Dirichlet boundary condition. Hence
		$u^\ast\in\mathcal{A}$.
		
		Define
		\[
		g:\mathbb{R}\times\mathbb{R}\longrightarrow[0,\infty),
		\qquad
		g(s,z):=\operatorname{dist}^2(z,\Phi(s)).
		\]
		The linear operator
		\[
		\mathcal{D}:
		\mathcal{W}(0,T)\longrightarrow L^2(Q),
		\qquad
		\mathcal{D}u:=\partial_tu-\Delta u,
		\]
		is continuous. Therefore, \eqref{eq:weak-convergence-W} implies
		\begin{equation}
		\partial_tu_n-\Delta u_n
		\rightharpoonup
		\partial_tu^\ast-\Delta u^\ast
		\quad\text{weakly in }L^2(Q).
		\label{eq:weak-convergence-operator}
		\end{equation}
		By Lemma~\ref{lem:normal-integrand-parabolic}, $g$ is Borel
		measurable, jointly lower semicontinuous, and convex in its second
		argument. Hence $\widetilde g(z,s,\zeta):=g(s,\zeta)$, regarded as an
		integrand on $Q\times\mathbb{R}\times\mathbb{R}$ independent of
		$z=(t,x)$, satisfies the hypotheses of
		Lemma~\ref{lem:strong-weak-lsc} with $Z=Q$ and $m=k=1$. Applying
		that lemma with \eqref{eq:strong-convergence-L2}
		($u_n\to u^\ast$ strongly in $L^2(Q)$) and
		\eqref{eq:weak-convergence-operator}
		($\partial_t u_n-\Delta u_n\rightharpoonup
		\partial_t u^\ast-\Delta u^\ast$ weakly in $L^2(Q)$), we obtain
		\[
		\begin{aligned}
			\mathcal{J}(u^\ast)
			&=
			\int_Q
			g\bigl(
			u^\ast,
			\partial_tu^\ast-\Delta u^\ast
			\bigr)\,dx\,dt
			\\
			&\leq
			\liminf_{n\to\infty}
			\int_Q
			g\bigl(
			u_n,
			\partial_tu_n-\Delta u_n
			\bigr)\,dx\,dt
			\\
			&=
			\liminf_{n\to\infty}
			\mathcal{J}(u_n)
			=
			0.
		\end{aligned}
		\]
		Since $\mathcal{J}$ is nonnegative, $\mathcal{J}(u^\ast)=0$.
		
		Part~\textup{(a)} then gives
		\[
		\partial_tu^\ast-\Delta u^\ast
		\in\Phi(u^\ast)
		\quad\text{for a.e. }(t,x)\in Q.
		\]
		Thus $u^\ast$ is a strong solution of
		\eqref{eq:parabolic-inclusion-background}, and consequently $u^\ast\in\mathcal{S}(u_0)$.	
	\end{proof}

\section{Numerical Simulations}
\label{sec:numerics}
We consider three numerical examples: two differential inclusions and one
parabolic partial differential inclusion.
\subsection{Linear control system with polytopic input set}
\label{subsec:linear-control}

\paragraph{Problem formulation}
We consider the two-dimensional linear differential inclusion
\begin{equation}
  \label{eq:linear-control}
  \dot{x}(t) \in Ax(t) + BU,
  \qquad x(0) = x_0,
  \qquad t \in [0, T],
\end{equation}
where the admissible-velocity set is
\begin{equation}
  \label{eq:Fx-linear}
  F(t,x) = \mathrm{conv}(V(x)),
  \qquad
  V(x) = \{Ax + Bu_k\}_{k=1}^{K}.
\end{equation}
This is an instance of regime~(b) in Section~\ref{sec:projection-extreme-points}:
the vertex set $V(x)$ depends on the current state through an affine image of
the fixed polytope $U = \mathrm{conv}\{u_1,\dots,u_K\}$.

\paragraph{System data}
The state dimension is $d = 2$ and the control dimension is $m = 2$.
The drift matrix and the control-influence matrix are
\begin{equation}
  \label{eq:AB-matrices}
  A =
  \begin{pmatrix}
    -0.5 &  1.0 \\
    -1.0 & -0.5
  \end{pmatrix},
  \qquad
  B = I_2,
\end{equation}
where $I_2$ denotes the $2\times 2$ identity matrix.
The eigenvalues of $A$ are $-0.5 \pm i$, so the uncontrolled system is a
stable spiral.
The initial condition is $x_0 = (1, 0)^\top$ and the time horizon is $T = 1.5$.

\paragraph{Control set and Quickhull preprocessing}
The control set is the $\ell^1$ diamond in $\mathbb{R}^2$,
\begin{equation}
  \label{eq:U-diamond}
  U = \mathrm{conv}\{e_1, -e_1, e_2, -e_2\},
\end{equation}
where $e_1, e_2$ are the standard basis vectors.
To illustrate the computational benefit of the Quickhull reduction, we
augment the four true extreme vertices with eight additional points drawn
uniformly from $[-0.35, 0.35]^2$, which lie in the interior of
$\mathrm{conv}(U)$.
The resulting raw vertex set has $K_{\mathrm{raw}} = 12$ points.
Applying the Quickhull algorithm (Section~\ref{sec:projection-extreme-points},
regime~(b)) to $B U_{\mathrm{raw}}$ recovers the four true extreme points,
reducing the QP~\eqref{eq:projection-as-qp} from $K_{\mathrm{raw}} = 12$ to
$\widetilde{K} = 4$ variables (see Figure \ref{fig:control-set}).
\begin{figure}[htbp]
	\centering
	\includegraphics[width=0.65\textwidth]{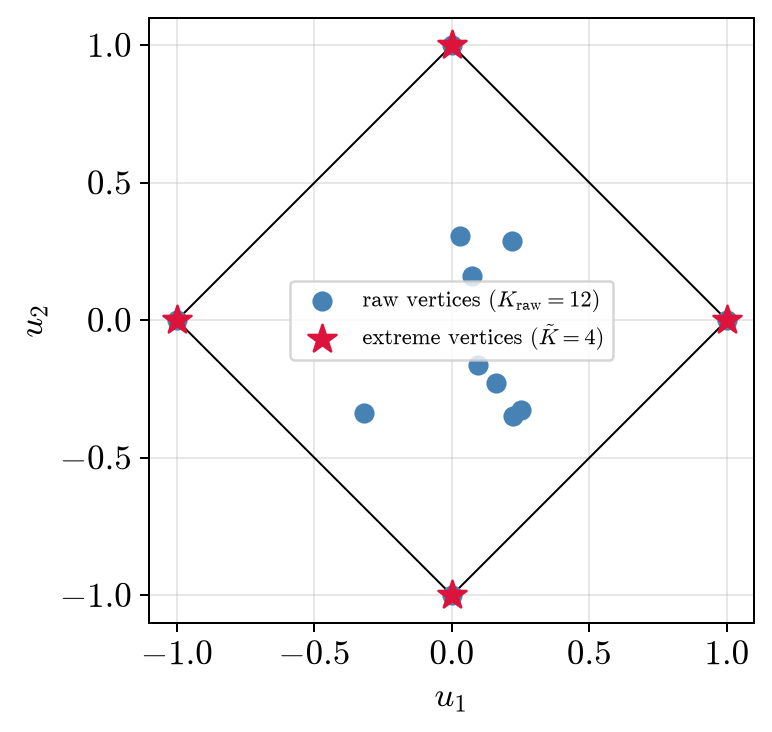}
	\caption{Raw vertex set $U_{\mathrm{raw}}$ ($K_{\mathrm{raw}}=12$, blue circles: the
		four true extreme vertices of the $\ell^1$ diamond plus eight redundant
		interior points) and the four extreme vertices $\widetilde{K}=4$ recovered
		by Quickhull (red stars), together with the convex hull of $U_{\mathrm{raw}}$.}
	\label{fig:control-set}
\end{figure}

Since $B = I_2$, the offline-reduced set is
$\widetilde{U} = \{e_1, -e_1, e_2, -e_2\}$,
and the vertex set at each collocation point is built online as
\begin{equation}
  \label{eq:Vtilde-online}
  \widetilde{V}(x_\theta(t_i))
  = \bigl\{Ax_\theta(t_i) + \tilde{u}_k\bigr\}_{k=1}^{\widetilde{K}}.
\end{equation}
The projection $\Pi_{F_i}\!\bigl(\dot{x}_\theta(t_i)\bigr)$
is then computed by solving the QP~\eqref{eq:projection-as-qp}
with $\widetilde{K} = 4$ variables using the SLSQP method.

\paragraph{Neural network and training}
The trial solution takes the form
\begin{equation}
  \label{eq:trial-linear}
  x_\theta(t) = x_0 + t\,N_\theta(t),
\end{equation}
where $N_\theta : \mathbb{R} \to \mathbb{R}^2$ is a fully connected neural
network with $\tanh$ activation, three hidden layers of width 64, and
Glorot-normal initialization.
The parametrization~\eqref{eq:trial-linear} enforces $x_\theta(0) = x_0$
exactly for every $\theta$, so the DR-PINN loss reduces to the pure inclusion
term
\begin{equation}
  \label{eq:loss-linear}
  \mathcal{L}(\theta)
  = \frac{1}{N} \sum_{i=1}^{N}
    \dist^2\!\bigl(\dot{x}_\theta(t_i),\, F(x_\theta(t_i))\bigr),
\end{equation}
evaluated at $N = 300$ collocation points distributed uniformly on $[0, T]$.
Since $F(x_\theta(t_i))=Ax_\theta(t_i)+B\widetilde{U}$ is a translation of
the \emph{fixed} polytope $B\widetilde{U}$, the residual is evaluated in the
equivalent translated form
$\dist^2\!\bigl(\dot x_\theta(t_i)-Ax_\theta(t_i),\,B\widetilde{U}\bigr)$ of
Remark~\ref{rem:moving-set-gradient}(i): the projection is computed
numerically by the QP~\eqref{eq:projection-as-qp} and the projected point is
held constant in the backward pass, so automatic differentiation returns the
exact gradients $2q$ and $-2A^{\top}q$ with respect to
$\dot x_\theta$ and $x_\theta$, respectively (the same practice as in the
DR-PINN companion run of Section~\ref{subsec:rotating-ellipsoidal}).
Training is performed with the Adam optimizer at learning rate $10^{-3}$ for
$1\,000$ epochs. Figure~\ref{fig:loss-linear} shows the inclusion loss $\mathcal{L}(\theta)$
of~\eqref{eq:loss-linear} along training. Starting from
$\mathcal{L}(\theta) = \LCLossInit$ at initialization, the loss
decreases to $\LCLossFinal$ after $1\,000$ epochs, confirming that the
learned velocity $\dot x_\theta(t)$ becomes admissible, i.e.\
$\dot x_\theta(t) \in F(x_\theta(t))$, at (numerically) every collocation
point. Collocation-point accuracy does not control the residual between
collocation points, so we also report an independent dense-grid validation
below.

\begin{figure}[htbp]
	\centering
	\includegraphics[width=\textwidth]{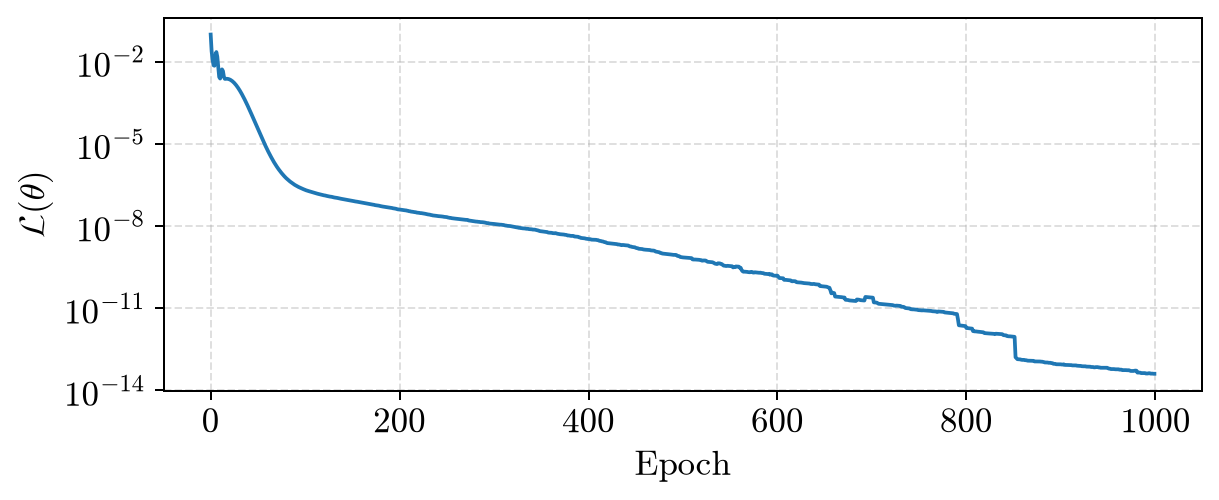}
	\caption{Training history of the inclusion loss $\mathcal{L}(\theta)$ for
		the single DR-PINN of this section (semi-logarithmic scale).}
	\label{fig:loss-linear}
\end{figure}

\paragraph{Reference reachable tube}
A reference inner approximation of the reachable tube
\begin{equation}
  \label{eq:reachable-tube}
  \mathcal{R}(t)
  =
  \bigl\{
    x(t)
    :
    \dot{x}(s) = Ax(s) + Bu(s),\;
    u(s) \in U \text{ a.e.},\;
    x(0) = x_0
  \bigr\}
\end{equation}
is constructed by integrating the ODE $\dot{x} = Ax + Bu$ with two families
of admissible controls:
\begin{enumerate}
	\item[(a)] \emph{Constant bang-bang controls}: one trajectory per extreme
	vertex $u_k \in \widetilde{U}$, totalling four reference trajectories.
	By Pontryagin's maximum principle, since the Hamiltonian of a linear
	system is affine in the control, boundary trajectories of
	$\mathcal{R}(t)$ are generated by bang-bang controls taking values in
	the extreme points $\widetilde{U}$, possibly \emph{switching} among
	them in time (see, e.g., \cite{LeeMarkus1967}, Chapter 4, or
	\cite{aubin_cellina_1984}, Chapter 8). The four \emph{constant} vertex
	controls therefore provide representative extremal trajectories; they do
	not, in general, trace the entire reachable-set boundary.
  \item[(b)] \emph{Piecewise-constant random controls}: 300 trajectories obtained
        by randomly switching among the vertices of $U$ at uniformly
        distributed times. Each constant-control segment
        $[t_j,t_{j+1}]$ is integrated over its \emph{exact} interval with a
        fifth-order Runge--Kutta scheme (relative tolerance $10^{-9}$,
        absolute tolerance $10^{-11}$) and dense output; the state handed to
        the next segment is the state at exactly the switching time
        $t_{j+1}$, and the trajectory is then evaluated at the global grid
        times falling inside the segment. Since the system is linear with
        piecewise-constant control, the exact solution is available in
        closed form via the segment-wise matrix-exponential propagation
        $x(t_{j+1}) = e^{A\Delta t_j}x(t_j) + A^{-1}(e^{A\Delta t_j}-I)Bu_j$;
        the replication script cross-checks every integrated trajectory
        against this closed form and aborts if the deviation exceeds
        $10^{-7}$ (observed maximum over all trajectories and grid times:
        $\LCIntSelfCheck$).
        This family provides an empirical sample of the interior of
        $\mathcal{R}(t)$ (it does not densely cover it in any
        mathematical sense).
\end{enumerate}
The union of all 304 trajectories evaluated on a uniform time grid of 200
points constitutes the \emph{sampled reference cloud} used for comparison:
at each grid time $t$,
\[
  \widehat{\mathcal{R}}_{\rm ref}(t)
  :=
  \bigl\{x^{(j)}(t):\ j=1,\dots,304\bigr\}
  \subset
  \mathcal{R}(t).
\]
We use separate notation: $\mathcal{R}(t)$ denotes the
exact reachable set of \eqref{eq:reachable-tube}, while
$\widehat{\mathcal{R}}_{\rm ref}(t)$ is its finite empirical sample.
Figure~\ref{fig:phase-linear} overlays the trained trajectory $x_\theta(t)$
with the sampled reference cloud $\widehat{\mathcal{R}}_{\rm ref}(t)$
approximating the reachable tube $\mathcal{R}(t)$
of~\eqref{eq:reachable-tube}: the left panel shows the phase portrait, the
right panel the two state components against time. The learned trajectory
remains within the visual envelope of the sampled reference trajectories
throughout $[0,T]$, in line with the consistency mechanism of
Theorem~\ref{thm:consistency} (which concerns the continuous residual
functional; cf.\ Remark~\ref{rem:collocation-gap}): a small inclusion
loss --- corroborated here by the independent dense-grid residual
validation reported below --- produces a trajectory close to empirically
generated admissible ones.
\begin{figure}[htbp]
	\centering
	\includegraphics[width=\textwidth]{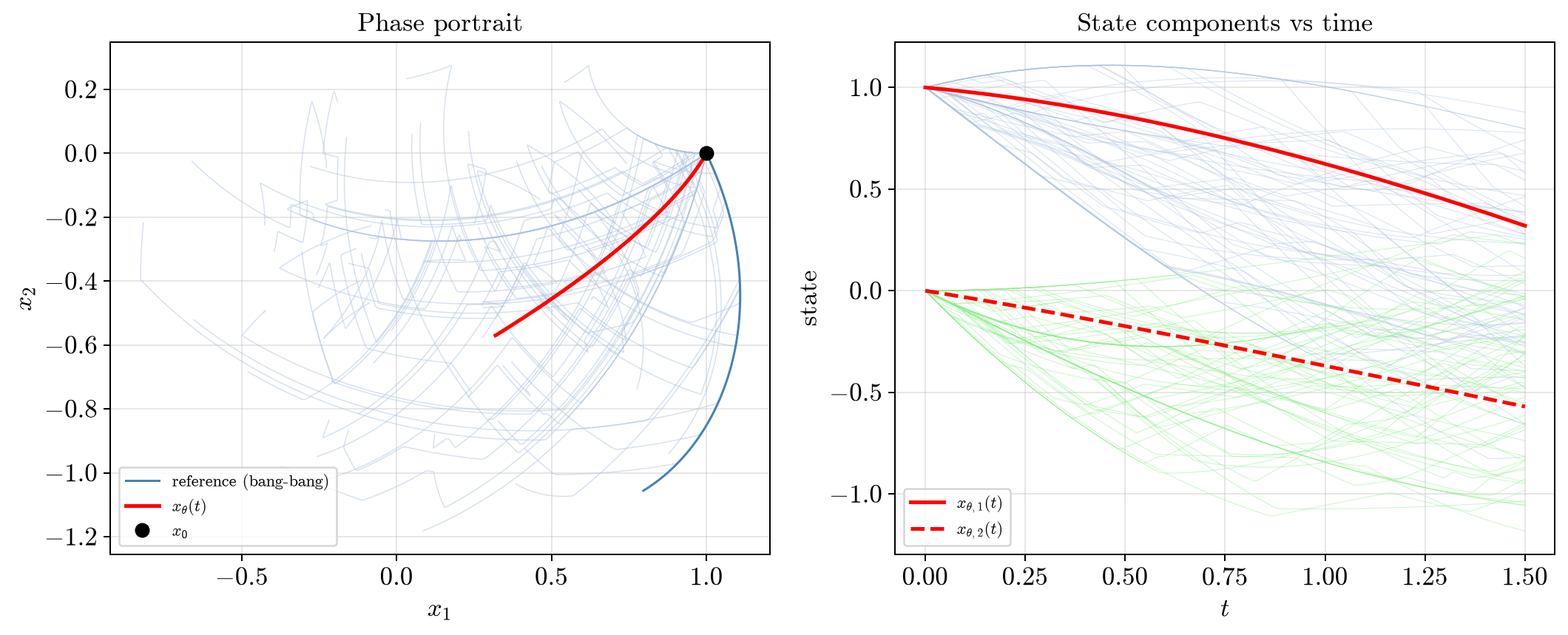}
	\caption{Phase portrait (left) and state components versus time (right)
		for the trained trajectory $x_\theta(t)$, compared with the sampled
		reference cloud $\widehat{\mathcal{R}}_{\rm ref}(t)$ (304 trajectories,
		light blue) and one representative bang-bang extremal (dark blue).}
	\label{fig:phase-linear}
\end{figure}

\paragraph{Ensemble and discrepancy from the sampled reference cloud}
Since a single DR-PINN approximates one element of the solution set
$\mathcal{S}(x_0)$ (cf.\ Section~\ref{subsec:consistency}), its distance to
the sampled reference cloud is not, by itself, a robust check. We therefore
form an ensemble of $n_{\mathrm{ens}} = 5$ trajectories (networks trained
with different random seeds) and compute, at each time, the maximum
distance of an ensemble member from the sampled reference cloud; this
quantity is then averaged over time. Precisely, the ensemble of trained
trajectories at time $t$ is
\[
  \mathcal{R}_\theta(t) = \bigl\{x_\theta^{(s)}(t) : s = 1,\dots,n_{\mathrm{ens}}\bigr\},
\]
and the discrepancy from the sampled reference cloud is measured by the
one-sided Hausdorff distance
\[
  d_H^+\!\bigl(\mathcal{R}_\theta(t),\,\widehat{\mathcal{R}}_{\rm ref}(t)\bigr)
  = \max_{a \in \mathcal{R}_\theta(t)}\min_{b \in \widehat{\mathcal{R}}_{\rm ref}(t)}\|a-b\|,
\]
which, at each time $t$, already takes the maximum over the ensemble
trajectories; averaging this quantity over the uniform time grid gives the
reported value
\[
  \overline{d_H^+}(\mathcal{R}_\theta,\widehat{\mathcal{R}}_{\rm ref}) = \LCMeanDH,
\]
which is small relative to the scale of the problem (the control set $U$ and
the reachable tube $\mathcal{R}(t)$ both have diameter of order $1$).
Because $\widehat{\mathcal{R}}_{\rm ref}(t)$ is only a finite trajectory
sample, $\overline{d_H^+}$ measures one-sided discrepancy from that sample,
not admissibility with respect to the exact reachable set $\mathcal{R}(t)$.
Its value $d_H^+$ also depends on the density of the cloud; a sparser cloud
can increase the discrepancy even for admissible trajectories. We therefore use this
quantity only to visualize the reachable tube. A certified comparison
would require the sets $\mathcal{R}(t)$ themselves, for example through
support functions or a guaranteed reachability solver. Admissibility in
this experiment is assessed instead by the dense-grid residual validation
reported next.

\paragraph{Independent dense-grid residual validation}
As announced above, and analogously to the dense-grid check of
Section~\ref{subsec:rotating-ellipsoidal}, the inclusion residual of each
ensemble member is additionally evaluated at $\LCnDense$ uniformly random
times \emph{off} the collocation grid, by solving the
QP~\eqref{eq:projection-as-qp} for
$\dist\!\bigl(\dot x_\theta(t)-Ax_\theta(t),\,B\widetilde U\bigr)$ at each
point. Over the whole ensemble, the worst-member mean inclusion distance is
$\LCResidMeanDense$ and the worst pointwise maximum is
$\LCResidMaxDense$, confirming that the trained velocities remain
(numerically) admissible between collocation points as well.

\paragraph{Computational scalability}
Computational efficiency is assessed by a scalability study that varies
$K_{\mathrm{raw}} \in \{8, 16, 32, 64, 128\}$ (at fixed $d = 2$) and
$d \in \{2, 4, 6, 8\}$ (at fixed $K_{\mathrm{raw}} = 32$), comparing the
per-epoch training time with and without the Quickhull reduction.
Table~\ref{tab:quickhull-K} lists the number of extreme vertices
$\widetilde{K}$ recovered by Quickhull for each raw vertex-set size
$K_{\mathrm{raw}} \in \{8,16,32,64,128\}$, at fixed $d=2$. As in the main
example, the raw point cloud consists of a fixed hexagonal boundary (six true
extreme points) plus $K_{\mathrm{raw}}-6$ redundant interior points, so
$\widetilde{K}=6$ irrespective of $K_{\mathrm{raw}}$.

\begin{table}[htbp]
  \centering
  \begin{tabular}{lccccc}
    \toprule
    $K_{\mathrm{raw}}$ & 8 & 16 & 32 & 64 & 128 \\
    \midrule
    $\widetilde{K}$    & 6 & 6  & 6  & 6  & 6   \\
    \bottomrule
  \end{tabular}
  \caption{Extreme vertices $\widetilde{K}$ recovered by Quickhull as
  $K_{\mathrm{raw}}$ grows, at fixed $d=2$.}
  \label{tab:quickhull-K}
\end{table}

For the dimension sweep, a fresh random drift matrix $A_d \in
\mathbb{R}^{d\times d}$ is generated for each $d \in \{2,4,6,8\}$, together
with a random polytope of $K_{\mathrm{raw}}=32$ vertices. To guarantee that
the uncontrolled system $\dot x = A_d x$ remains asymptotically stable
regardless of the random draw, $A_d$ is constructed by first sampling a
Gaussian matrix $G$ (i.i.d.\ entries $\mathcal{N}(0,0.3^2)$) and then
shifting its diagonal,
\[
  (A_d)_{ii} = G_{ii} - \Bigl(\textstyle\sum_{j=1}^{d}|G_{ij}| + 0.5\Bigr),
  \qquad
  (A_d)_{ij} = G_{ij} \ \ (j \neq i).
\]
By the Gershgorin circle theorem, every eigenvalue $\lambda$ of $A_d$ lies in
a disc centred at the real number $(A_d)_{ii}$ with radius
$R_i = \sum_{j\neq i}|(A_d)_{ij}| = \sum_{j\neq i}|G_{ij}|$, so
\[
  \operatorname{Re}(\lambda) \;\le\; (A_d)_{ii} + R_i
  \;=\; G_{ii} - |G_{ii}| - 0.5
  \;\le\; -0.5
\]
for every row $i$, since $G_{ii}\le |G_{ii}|$. Hence every eigenvalue of
$A_d$ has real part at most $-0.5$, for every random draw and every $d$: the
construction guarantees a uniform stability margin without ever having to
reject an unstable sample.

Table~\ref{tab:quickhull-d} shows that, unlike the $K$-sweep, $\widetilde{K}$
now grows with $d$: a random point cloud in higher dimensions has a larger
fraction of extreme points, and at $d=6,8$ all $32$ raw vertices are already
extreme, so no reduction is possible.

\begin{table}[htbp]
  \centering
  \begin{tabular}{lcccc}
    \toprule
    $d$             & 2  & 4  & 6  & 8  \\
    \midrule
    $\widetilde{K}$ & 12 & 23 & 32 & 32 \\
    \bottomrule
  \end{tabular}
  \caption{Extreme vertices $\widetilde{K}$ recovered by Quickhull as the
  state dimension $d$ grows, at fixed $K_{\mathrm{raw}}=32$.}
  \label{tab:quickhull-d}
\end{table}

All timings reported here are collected in the same single execution of the
replication script that produces every other number and figure of this
section: for each configuration, one discarded warm-up training run is
followed by $\LCtimingReps$ timed repetitions, and we report the mean
per-epoch time (error bars in Figure~\ref{fig:scalability-linear} show one
standard deviation across repetitions). The hardware and library versions
of the run are recorded in the accompanying manifest of the replication
package; absolute times are hardware-dependent, and only the \emph{relative}
comparison between the two variants is of interest.

Figure~\ref{fig:scalability-linear} reports the resulting per-epoch training
time with and without the offline Quickhull reduction. As a function of
$K_{\mathrm{raw}}$ (left panel), the Quickhull-reduced training time stays
essentially flat, between \LCtimeQHmin{} and \LCtimeQHmax, while the
unreduced QP grows from \LCtimeNoQHfirst{} at $K_{\mathrm{raw}}=8$ to
\LCtimeNoQHlast{} at $K_{\mathrm{raw}}=128$ -- a $\LCslowdown$
slowdown -- because the QP of Section~\ref{sec:projection-extreme-points}
then has $128$ variables instead of $6$. As a function of $d$ (right panel),
the benefit shrinks as $\widetilde{K}$ approaches $K_{\mathrm{raw}}$: at
$d=6$ and $d=8$, where Table~\ref{tab:quickhull-d} shows no reduction is
achieved, the Quickhull-reduced time (\LCtimeQHdSix, \LCtimeQHdEight) is
no longer lower than the unreduced time (\LCtimeNoQHdSix,
\LCtimeNoQHdEight). Since the Quickhull reduction is performed
\emph{offline} and does not contribute to the per-epoch time, when
$\widetilde{K}=K_{\mathrm{raw}}$ the two variants solve identical QPs, and
the residual differences between their timings are within run-to-run
measurement variability rather than a preprocessing cost.

\begin{figure}[htbp]
  \centering
  \includegraphics[width=\textwidth]{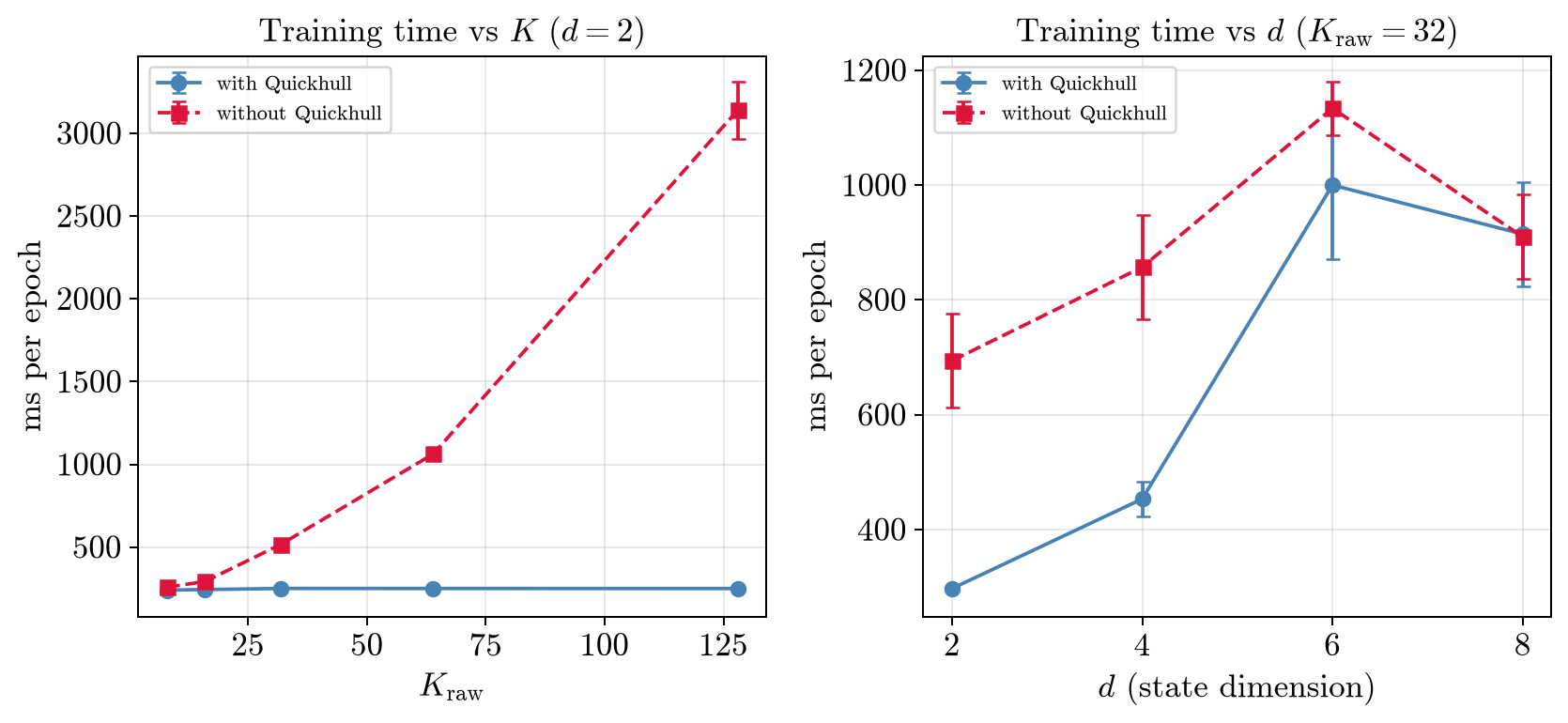}
  \caption{Per-epoch training time with (blue) and without (red) the offline
  Quickhull reduction; markers are means over $\LCtimingReps$ repetitions
  after a discarded warm-up run, error bars show one standard deviation.
  Left: as a function of $K_{\mathrm{raw}}$ at fixed
  $d=2$ (Table~\ref{tab:quickhull-K}). Right: as a function of the state
  dimension $d$ at fixed $K_{\mathrm{raw}}=32$ (Table~\ref{tab:quickhull-d}).}
  \label{fig:scalability-linear}
\end{figure}

\subsection{A Planar Differential Inclusion with a Rotating Ellipsoidal Control Constraint}
\label{subsec:rotating-ellipsoidal}

\paragraph{Problem formulation}
The two examples in Sections~\ref{subsec:linear-control}
and~\ref{subsec:parabolic-PDI} enforce the inclusion by \emph{minimizing}
the distance residual. This second ODE example illustrates the
complementary strategy anticipated in
Section~\ref{sec:description-formulation}: when the admissible set has a
tractable parametric description, admissibility can be built into the
parametrization itself, so that the distance residual vanishes
at every time-grid node of every candidate trajectory and the remaining
degrees of freedom can be used to select a specific element of the
solution set $\mathcal{S}(x_0)$ --- here, the trajectory steering the
state closest to a prescribed terminal point. This also reflects the
non-uniqueness discussed in Section~\ref{subsec:consistency}: the
distance-residual loss guarantees convergence to \emph{some} admissible
trajectory, whereas a hard-admissible parametrization allows one to
optimize \emph{which} one.

We consider trajectories $x:[0,T]\to\mathbb{R}^2$ of the differential
inclusion
\begin{equation}
	\label{eq:ellipse-inclusion}
	\dot{x}(t)\in F(t,x(t)) = g(t,x(t))+E(t),
	\qquad x(0)=x_0,
\end{equation}
whose admissible-velocity set is of Minkowski type: a nonlinear,
nonautonomous drift
\begin{equation}
	\label{eq:ellipse-drift}
	g(t,x)=
	\begin{pmatrix}
		\sin(x_1)+0.15\cos\bigl(\tfrac{2\pi t}{T}\bigr)\\[2mm]
		0.5\cos(x_2)-0.20\sin\bigl(\tfrac{2\pi t}{T}\bigr)
	\end{pmatrix},
	\qquad x=(x_1,x_2)^{\top},
\end{equation}
translated by a compact, convex, \emph{time-dependent} control set
$E(t)$, a rotating ellipse defined below. Equivalently, admissible
trajectories are exactly the solutions of the controlled system
\begin{equation}
	\label{eq:ellipse-controlled}
	\dot{x}(t)=g(t,x(t))+\xi(t),
	\qquad \xi(t)\in E(t)\ \text{ a.e.\ on } [0,T],
\end{equation}
where $\xi$ is a measurable selector of $t\mapsto E(t)$. Since
$E(t)$ does not depend on the state, this example falls into
regime~(a) of Section~\ref{sec:projection-extreme-points}; however,
unlike the polytopic set of Section~\ref{subsec:linear-control}, the
projection onto an ellipse admits no closed form (it requires the
solution of a scalar secular equation), so we also consider the
hard-admissible route here. We take $T=3$ and $x_0=(0,0)^{\top}$.

\paragraph{The rotating ellipsoidal control set}
For $t\in[0,T]$, let
\begin{equation}
	\label{eq:ellipse-axes}
	a(t)=0.3+0.4\Bigl|\sin\Bigl(\frac{\pi t}{T}\Bigr)\Bigr|,
	\qquad
	b(t)=0.15+0.25\Bigl|\cos\Bigl(\frac{\pi t}{T}\Bigr)\Bigr|,
	\qquad
	\varphi(t)=\frac{\pi t}{T},
\end{equation}
and let $R_{\varphi(t)}$ denote the rotation by the angle $\varphi(t)$.
The control set is the (filled) rotating ellipse
\begin{equation}
	\label{eq:ellipse-set}
	E(t)=\Bigl\{R_{\varphi(t)}
	\begin{pmatrix}u\\w\end{pmatrix}\in\mathbb{R}^2:
	\frac{u^2}{a(t)^2}+\frac{w^2}{b(t)^2}\leq 1\Bigr\},
\end{equation}
which is nonempty, compact, and convex for every $t$, with semi-axes and
orientation varying continuously in time; since $\varphi(0)=0$ and
$\varphi(T)=\pi$, the ellipse performs a half-turn over the horizon, and
the translated family $g(t,x(t))+E(t)$ traces a twisted tube of
admissible velocities along the drift curve
(Figure~\ref{fig:ellipse-tube}). Writing
$(u,w)^{\top}=R_{-\varphi(t)}\,\xi$ for the coordinates of
$\xi\in\mathbb{R}^2$ in the local frame of the ellipse, the
\emph{level function}
\begin{equation}
	\label{eq:ellipse-level}
	\ell(t,\xi)
	=\Bigl(\frac{u}{a(t)}\Bigr)^2+\Bigl(\frac{w}{b(t)}\Bigr)^2
\end{equation}
characterizes admissibility: $\xi\in E(t)$ if and only if
$\ell(t,\xi)\le1$, with $\ell=1$ on the boundary $\partial E(t)$.

\paragraph{Hard-admissible selector parametrization}
The selector is generated using $n_{\rm c}=24$ control nodes
$(p^j,\rho^j)\in\mathbb{R}^2\times\mathbb{R}$, $j=0,\dots,n_{\rm c}-1$,
linearly interpolated to the uniform time grid
$t_i=i\,\Delta t$, $i=0,\dots,n_t-1$, $\Delta t=T/(n_t-1)$, with
$n_t=120$; the low-dimensional nodal representation suppresses
artificial oscillations of the selector. From the interpolated values
$p(t_i)\in\mathbb{R}^2$, $\rho(t_i)\in\mathbb{R}$, define the unit-ball
direction and the sigmoidal radius
\begin{equation}
	\label{eq:ellipse-dir-radius}
	d(t_i)=\frac{p(t_i)}{\sqrt{\,|p(t_i)|^2+\varepsilon\,}},
	\qquad
	s(t_i)=\frac{1}{1+e^{-\rho(t_i)}}\in(0,1),
\end{equation}
with $\varepsilon=10^{-10}$ a normalization safeguard, and set
\begin{equation}
	\label{eq:ellipse-selector}
	\xi(t_i)
	=R_{\varphi(t_i)}
	\begin{pmatrix}
		s(t_i)\,a(t_i)\,d_1(t_i)\\
		s(t_i)\,b(t_i)\,d_2(t_i)
	\end{pmatrix}.
\end{equation}
Substituting \eqref{eq:ellipse-selector} into
\eqref{eq:ellipse-level} gives, for every grid point,
\begin{equation}
	\label{eq:ellipse-admissibility}
	\ell\bigl(t_i,\xi(t_i)\bigr)
	=s(t_i)^2\,|d(t_i)|^2
	\le s(t_i)^2<1,
\end{equation}
so the selector lies in the interior of $E(t_i)$ \emph{by construction},
for every value of the nodal variables: no penalty for constraint
violation is required, and the inclusion residual
$\dist\bigl(\dot x(t_i)-g(t_i,x(t_i)),E(t_i)\bigr)$ of
Section~\ref{sec:description-formulation} vanishes at every grid point
$t_i$. The guarantee \eqref{eq:ellipse-admissibility} is nodewise:
$\xi$ is recomputed from \eqref{eq:ellipse-selector} at the grid points.
Between nodes, the selector $\xi$ is linearly interpolated from the values
$\xi(t_i)$, and a convex combination of velocities admissible at different
times need not lie in the intermediate rotated and rescaled ellipse. For the reported optimized selector we
verified admissibility a posteriori on a dense grid of $\EllipNDense$
intermediate points: the maximum level is $\ell\approx\EllipDenseMaxL$,
with no violations. The dense-grid maximum may therefore slightly exceed the grid-node maximum
in Table~\ref{tab:ellipse-continuation}: the former also probes the
intermediate ellipses, which are not covered by the nodewise construction
\eqref{eq:ellipse-admissibility}.
An implementation that re-evaluates $d(t)$, $s(t)$, and $\xi(t)$ from
interpolated nodal values $p(t)$, $\rho(t)$ at every RK4 stage would make
the guarantee hold at every evaluation time by construction. This is
the analogue, for the inclusion constraint itself, of the hard-constraint
ansatz used for the initial condition in
Section~\ref{subsec:linear-control} and for the initial/boundary
conditions in Section~\ref{subsec:parabolic-PDI}.

\paragraph{Steering problem and discretization}
For a fixed selector, the state is propagated through
\eqref{eq:ellipse-controlled} by the classical fourth-order
Runge--Kutta scheme on the grid $\{t_i\}$, with $\xi$ evaluated between
grid points by linear interpolation. The nodal variables
$z=(p^0,\dots,p^{n_{\rm c}-1},\rho^0,\dots,\rho^{n_{\rm c}-1})
\in\mathbb{R}^{3n_{\rm c}}=\mathbb{R}^{72}$
are chosen by minimizing
\begin{equation}
	\label{eq:ellipse-cost}
	J(z)=\lambda_{\rm f}\,\bigl|x_z(T)-x_{\rm f}\bigr|^2
	+\lambda_{\xi}J_{\xi}(z)+\lambda_{\rm s}J_{\rm s}(z)
	+\lambda_{\rm c}J_{\rm c}(z)+\lambda_{\rm b}J_{\rm b}(z)
	+\lambda_{\rm u}J_{\rm u}(z),
\end{equation}
where $x_z$ denotes the discrete trajectory induced by $z$, the terminal
term steers the state to a prescribed target $x_{\rm f}$, and the
regularizers are the mean level
$J_{\xi}=\frac{1}{n_t}\sum_i\ell(t_i,\xi_z(t_i))$
(discouraging boundary-hugging selectors), the mean squared first and
second discrete differences of $\xi_z$
($J_{\rm s}$, $J_{\rm c}$: smoothness and curvature), the cubic soft
barrier $J_{\rm b}=\frac{1}{n_t}\sum_i\ell(t_i,\xi_z(t_i))^3$
(penalizing boundary saturation), and the mean squared differences of
the nodal variables ($J_{\rm u}$). The weights are
$\lambda_{\rm f}=2500$, $\lambda_{\rm s}=1$, $\lambda_{\rm c}=0.2$,
$\lambda_{\rm b}=10^{-2}$, $\lambda_{\rm u}=5\cdot10^{-2}$, and the
selector-regularization weight is driven along the continuation sequence
$\lambda_{\xi}\in\{10^{-3},10^{-2},10^{-1}\}$, each stage warm-started
from the previous one and solved by L-BFGS-B.

The target must of course be chosen inside the reachable set of
\eqref{eq:ellipse-inclusion}: the drift component
$g_2=0.5\cos(x_2)-0.2\sin(2\pi t/T)$ pushes the state upward with
magnitude $\approx0.5$ near $x_2=0$, while the vertical half-width of the
rotated ellipse $E(t)$,
\[
h_y(t)=\sqrt{a(t)^2\sin^2\varphi(t)+b(t)^2\cos^2\varphi(t)},
\]
stays between $\approx0.391$ and $0.7$ for the parameters
\eqref{eq:ellipse-axes} (the value $0.15$ is the minimal length of one
semi-axis, not the minimal vertical reach of the rotated ellipse), so
targets with markedly negative
$x_2$ are unreachable over $[0,T]$ (an extremal selector taking, at
every instant, the support point of $E(t)$ in the direction $-e_2$ only
attains $x_2(T)\approx\EllipExtremalXtwo$). We take
$x_{\rm f}=(1.8,\,0.1)^{\top}$, which lies inside but close to the
boundary of the reachable set: steering to it forces the selector to
operate near $\partial E(t)$ over most of the horizon, so the built-in
admissibility guarantee \eqref{eq:ellipse-admissibility} is genuinely
stressed.

\paragraph{Numerical results}
Table~\ref{tab:ellipse-continuation} reports, for each continuation
stage, the terminal mismatch and the mean and maximum of the level
function along the optimized selector. The terminal state matches the
target closely at every stage, while increasing
$\lambda_{\xi}$ pulls the selector towards the interior of
the ellipse (mean level decreasing from $\EllipMeanLA$ to $\EllipMeanLC$) at
essentially no cost in terminal accuracy. For a target close to the
reachable-set boundary, near-extremal velocities are needed over most of
the horizon.

\begin{table}[htbp]
	\centering
	\begin{tabular}{cccc}
		\toprule
		$\lambda_{\xi}$ & $|x_z(T)-x_{\rm f}|$ & mean $\ell$ & max $\ell$ \\
		\midrule
		$10^{-3}$ & $\EllipMismA$ & $\EllipMeanLA$ & $\EllipMaxLA$ \\
		$10^{-2}$ & $\EllipMismB$ & $\EllipMeanLB$ & $\EllipMaxLB$ \\
		$10^{-1}$ & $\EllipMismC$ & $\EllipMeanLC$ & $\EllipMaxLC$ \\
		\bottomrule
	\end{tabular}
	\caption{Continuation over the selector-regularization weight
		$\lambda_{\xi}$ (warm-started L-BFGS-B): terminal mismatch and level
		statistics of the optimized selector. Admissibility
		$\ell(t_i,\xi_z(t_i))<1$ holds at every grid point and every stage, by
		construction.}
	\label{tab:ellipse-continuation}
\end{table}

Figure~\ref{fig:ellipse-tube} displays the resulting geometry in the
space $(v_1,v_2,t)$ of velocities and time: the translated ellipses
$g(t_i,x_z(t_i))+E(t_i)$ form a twisted tube whose centerline is the
drift curve, and the selected velocity
$\dot{x}_z(t_i)=g(t_i,x_z(t_i))+\xi_z(t_i)$ runs inside the tube, close
to its wall, as quantified by the level statistics of
Table~\ref{tab:ellipse-continuation}. Figure~\ref{fig:ellipse-selector}
shows the same information in selector coordinates: the trajectory
$t\mapsto\xi_z(t)$ threads the rotating ellipse $E(t)$ centered at the
origin, with the color scale reporting the level
$\ell(t_i,\xi_z(t_i))$. Finally, Figure~\ref{fig:ellipse-state} shows
the optimized state trajectory from $x_0$ to $x_{\rm f}$ (left) together
with the admissibility level along the horizon (right), which stays
below the boundary value $\ell=1$ on the dense validation grid.

\begin{figure}[htbp]
	\centering
	\includegraphics[width=0.78\textwidth]{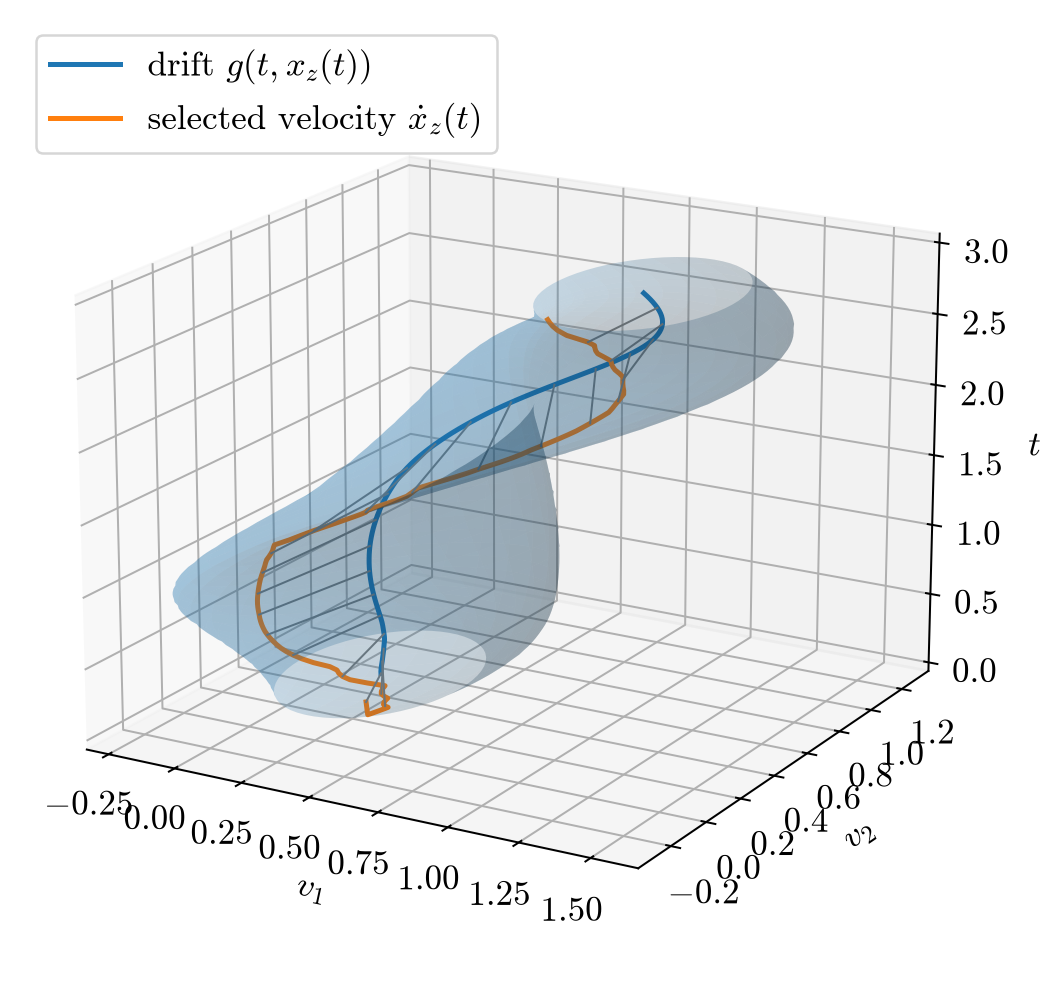}
	\caption{Admissible velocity tube of the inclusion
		\eqref{eq:ellipse-inclusion} in $(v_1,v_2,t)$ space. The surface is the
		boundary of the translated rotating ellipses $g(t,x_z(t))+E(t)$; the
		blue curve is the centerline $g(t,x_z(t))$, the orange curve the
		selected velocity $\dot{x}_z(t)$, and the grey segments the selector
		displacements $\xi_z(t_i)$.}
	\label{fig:ellipse-tube}
\end{figure}

\begin{figure}[htbp]
	\centering
	\includegraphics[width=0.78\textwidth]{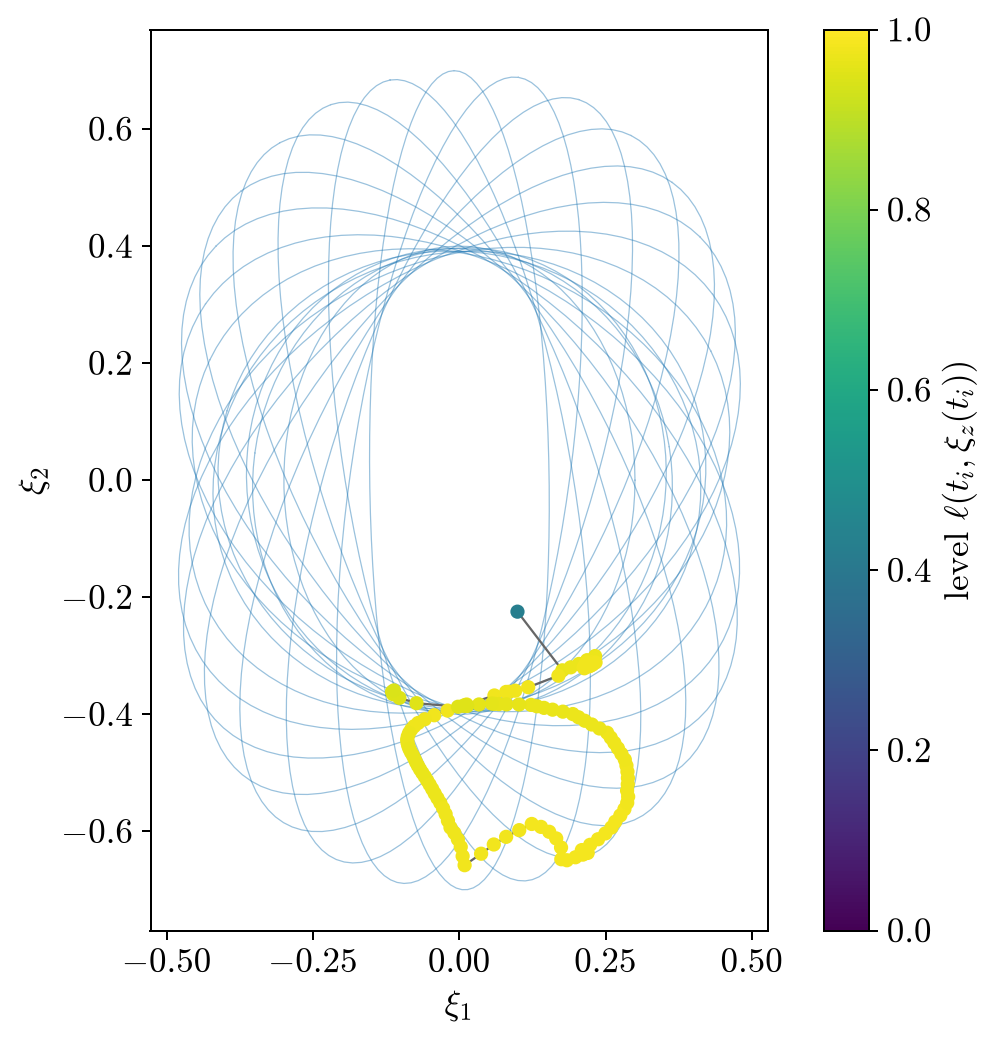}
	\caption{Optimized selector $\xi_z(t)$ inside the rotating control
		ellipse $E(t)$ (thin blue rings, drawn every sixth grid point). The
		color of the points encodes the level $\ell(t_i,\xi_z(t_i))$, which
		remains below $1$ at the plotted time-grid nodes by construction.}
	\label{fig:ellipse-selector}
\end{figure}

\begin{figure}[htbp]
	\centering
	\includegraphics[width=\textwidth]{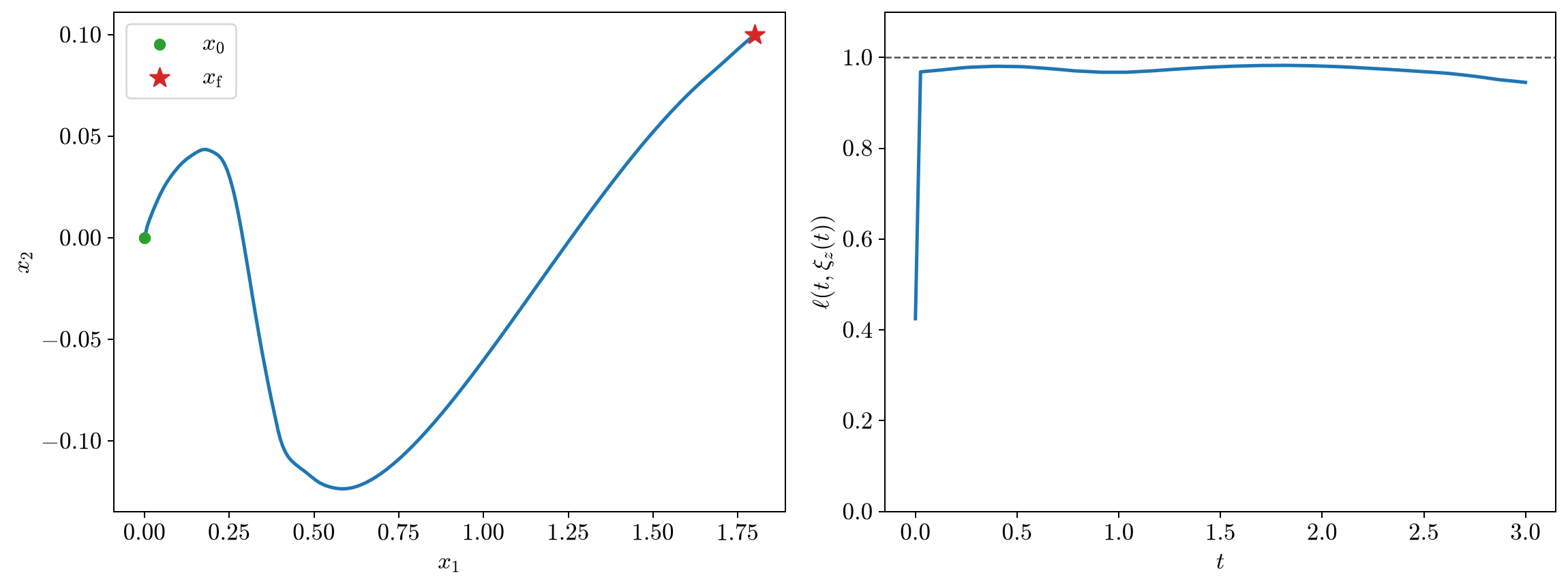}
	\caption{Left: optimized state trajectory of
		\eqref{eq:ellipse-controlled} from $x_0=(0,0)^{\top}$ to the target
		$x_{\rm f}=(1.8,0.1)^{\top}$ (final continuation stage
		$\lambda_{\xi}=10^{-1}$, terminal mismatch $\EllipMismC$). Right: admissibility level
		$\ell(t,\xi_z(t))$ along the horizon; the dashed line marks the
		boundary value $\ell=1$ of $\partial E(t)$.}
	\label{fig:ellipse-state}
\end{figure}
\paragraph{DR-PINN companion run}
The hard-admissible selector above is a constrained optimal-control
discretization: it propagates the state by RK4 and enforces the inclusion
at the time-grid nodes by construction. We therefore use a separate
DR-PINN run to test the distance-residual method on the same benchmark. Because the experiment is a steering problem with prescribed
initial and terminal states, we used the endpoint-hard neural ansatz
\begin{equation}
\label{eq:ellipse-endpoint-hard}
 x_\theta(t)
 =x_0+\frac{t}{T}(x_{\rm f}-x_0)+t(T-t)N_\theta(t),
\end{equation}
where $N_\theta$ is a $3\times64$ tanh network. Hence
$x_\theta(0)=x_0$ and $x_\theta(T)=x_{\rm f}$ exactly, and the training
objective is the pure distance residual
\begin{equation}
\label{eq:ellipse-drpinn-loss}
 \mathcal{L}(\theta)
 =\frac{1}{N}\sum_{i=1}^{N}
 \dist^2\!\bigl(\dot{x}_\theta(t_i)-g(t_i,x_\theta(t_i)),E(t_i)\bigr),
\end{equation}
with $N=512$ uniform collocation points.

The projection onto an ellipse is evaluated pointwise by Newton iteration on
the secular equation
$\sum_k(\sigma_k v_k^{\rm loc}/(\sigma_k^2+\mu))^2=1$ in the local frame of
$E(t)$, with $\sigma_1=a(t)$ and $\sigma_2=b(t)$. In the backward pass the
\emph{entire projected point} $\Pi_{E(t)}(v)$ is held constant, so automatic
differentiation returns the exact envelope gradient
$2(v-\Pi_{E(t)}(v))$. Holding only the multiplier $\mu$ constant would not
give this gradient; this distinction is important in an implementation with
a numerically computed projection. A finite-difference gradient check in the
replication code gives a relative error of $\EllipCompGradCheck$.

After $7000$ Adam epochs (hold-then-decay learning rate, double precision),
the collocation loss is $\EllipCompLoss$. On an independent dense grid of
$2000$ time points, the mean and maximum inclusion distances are respectively
$\EllipCompDenseMean$ and $\EllipCompDenseMax$, while both endpoint errors are zero
up to machine precision. The corresponding selector has mean level $\EllipCompMeanL$
and maximum level $\EllipCompMaxL$; any excess of the latter over $1$ is
consistent with the reported maximum inclusion distance and the numerical
projection tolerance. With this implementation, the DR-PINN attains the reported accuracy on
this benchmark.

The companion result is based on a single random seed. It establishes the
reported accuracy for this run but does not characterize training
robustness. A fuller robustness study would report medians and ranges over
$5$--$10$ seeds, unsuccessful-run frequency, and a comparison with a direct
parametrization of the control $\xi$ trained on the same steering
objective.
The hard-admissible selector is a complementary constrained discretization,
not a DR-PINN experiment. Conclusions about the distance residual rely on
the companion run and on the experiments of
Sections~\ref{subsec:linear-control} and~\ref{subsec:parabolic-PDI}.
The hard-admissible parametrization guarantees feasibility at the time-grid
nodes independently of training; feasibility between nodes for the reported
selector was also checked on a dense grid of $\EllipNDense$ points.

\subsection{A Reaction-Diffusion Inclusion with Relay (Sliding-Mode) Feedback}
\label{subsec:parabolic-PDI}

As a benchmark for the parabolic setting of Section~\ref{sec:pde},
we consider the reaction-diffusion inclusion
\begin{equation}
	\label{eq:relay-pde}
	\partial_t u(t,x) - \Delta u(t,x) \in \Phi(u(t,x)),
	\qquad (t,x)\in Q = (0,T)\times\Omega,
\end{equation}
with homogeneous Dirichlet boundary condition, where
the set-valued reaction term is the relay (signum) multifunction
\begin{equation}
	\label{eq:relay-Phi}
	\Phi(s) := -\lambda\,\Sign(s)
	=
	\begin{cases}
		\{-\lambda\}, & s>0,\\[2pt]
		[-\lambda,\lambda], & s=0,\\[2pt]
		\{\lambda\}, & s<0,
	\end{cases}
	\qquad \lambda>0.
\end{equation}
Physically, \eqref{eq:relay-pde}--\eqref{eq:relay-Phi} are the parabolic
analogue of the bang-bang differential inclusion of
Section~\ref{subsec:linear-control}. Indeed, at every point $(t,x)$, instead of a
smooth reaction term, the state is driven towards the equilibrium $u\equiv0$
by a switching (relay) feedback of constant magnitude $\lambda$, exactly as
a sliding-mode controller drives a state to a manifold at constant rate
regardless of its distance to that manifold. Equivalently, $\Phi=-\lambda\,\partial|\cdot|$, so that \eqref{eq:relay-pde}
may be read as the abstract evolution inclusion
\begin{equation}
	\label{eq:relay-gradient-flow}
	\partial_t u \in -\partial E(u),
	\qquad
	E(u) = \frac12\|\nabla u\|_{L^2(\Omega)}^2 + \lambda\|u\|_{L^1(\Omega)},
\end{equation}
associated with the Dirichlet energy regularized by an $L^1$ (total-mass)
penal\-ty, which connects this example to sparsity-promoting flows in
variational image processing and signal denoising.

We take $\Omega=(0,1)^2$ and
\begin{equation}
	\label{eq:relay-u0}
	u_0(x,y) = 16\,x(1-x)\,y(1-y),
\end{equation}
which is smooth, nonnegative, and satisfies $u_0\in H_0^1(\Omega)$, and take
$\lambda=0.6$ as a baseline value, together with a primary comparison sweep
$\lambda\in\{0.4,0.6,0.8\}$ to illustrate how the extinction time decreases
as $\lambda$ increases (cf.\ Proposition~\ref{prop:extinction} below). The time horizon is fixed at $T=0.3$, which
is large enough that all three relay cases have visibly reached extinction while
the $\lambda=0$ (classical heat equation) case has not, making the qualitative
contrast visible in a single plot (see Figure~\ref{fig:relay-extinction-ref} below).

We check that $\Phi$ in \eqref{eq:relay-Phi} satisfies \ref{P1}--\ref{P3} of
Section~\ref{sec:basicsPDI}, so that both the existence result
(Theorem~\ref{thm:existence-pde}) and the consistency result
(Theorem~\ref{thm:consistency-parabolic}) apply to this example.
\begin{itemize}
	\item[(P1)] For every $s\in\mathbb{R}$, $\Phi(s)$ is a singleton or the
	closed interval $[-\lambda,\lambda]$, hence nonempty, closed, and convex.
	\item[(P2)] $\Phi$ is upper semicontinuous: it is locally constant on
	$\mathbb{R}\setminus\{0\}$, and at $s=0$, $\Phi(s_n)\subset[-\lambda,\lambda]=\Phi(0)$
	for every sequence $s_n\to0$. This is the standard Filippov
	regularization of the signum function \cite{Filippov1988}.
	\item[(P3)] $\sup_{r\in\Phi(s)}|r| = \lambda \le \lambda(1+|s|)$ for every
	$s\in\mathbb{R}$, so \ref{P3} holds with $m_0=\lambda$.
\end{itemize}

\begin{remark}\label{rem:relay-wellposed}
	Existence of a strong solution for this example is guaranteed
	directly by Theorem~\ref{thm:existence-pde}, since \ref{P1}--\ref{P3}
	hold. Note that any selection of $\Phi$ in \eqref{eq:relay-Phi} is
	forced to equal $-\lambda$ for $s>0$ and $\lambda$ for $s<0$, hence
	is necessarily discontinuous at $s=0$; no reduction to a single-valued
	semilinear equation via a continuous selection is available, and the
	genuinely multivalued fixed-point argument of
	Section~\ref{sec:basicsPDI} is essential here. The example has, however,
	additional structure that yields more than existence --- namely
	uniqueness and explicit a priori bounds --- via the classical theory
	of evolution equations governed by maximal monotone operators
	\cite{Brezis1973}: rewriting \eqref{eq:relay-pde} as
	$\partial_t u + \bigl(-\Delta+\lambda\,\partial|\cdot|\bigr)(u)\ni 0$,
	we identify $-\Delta+\lambda\,\partial|\cdot|$ (with the Dirichlet
	Laplacian) as the subdifferential $\partial E$ of the functional $E$ of
	\eqref{eq:relay-gradient-flow}, extended by $E(u)=+\infty$ for
	$u\in L^2(\Omega)\setminus H_0^1(\Omega)$. Indeed, $E=\varphi_1+\varphi_2$
	with $\varphi_1(u)=\tfrac12\|\nabla u\|_{L^2(\Omega)}^2$ proper, convex,
	and lower semicontinuous on $L^2(\Omega)$, and
	$\varphi_2(u)=\lambda\|u\|_{L^1(\Omega)}$ convex and, since $\Omega$ is
	bounded, finite and continuous on all of $L^2(\Omega)$; the subdifferential
	sum rule (Moreau--Rockafellar, with the qualification condition satisfied
	because $\varphi_2$ is continuous) then gives
	$\partial E=\partial\varphi_1+\partial\varphi_2
	=-\Delta+\lambda\,\partial|\cdot|$. As the subdifferential of a proper,
	convex, lower semicontinuous functional, $\partial E$ is maximal monotone
	\cite{Brezis1973}. Here maximal monotonicity follows from the
	subdifferential representation; maximal monotonicity of the two summands
	alone would not suffice in general.
	
	The subgradient-flow theory provides existence in $L^2$ and also yields
	the regularity required for a strong solution in the sense of
	Definition~\ref{def:strong-solution-pde}, defined for every time
	horizon. Indeed, since $u_0\in H_0^1(\Omega)$ and $\Omega$ is bounded,
	$E(u_0)<\infty$, and the theory of subgradient flows with initial datum
	in the domain of the energy \cite{Brezis1973} produces, for every
	$T>0$, a unique function $u\in H^1(0,T;L^2(\Omega))$ with
	$u(t)\in D(\partial E)$ and
	$-\partial_tu(t)\in\partial E(u(t))$ for a.e.\ $t$. By the
	subdifferential sum rule above, for a.e.\ $t$ one has
	$u(t)\in H^2(\Omega)\cap H_0^1(\Omega)$ and there exists
	$\xi(t,\cdot)\in\partial\|\cdot\|_{L^1(\Omega)}(u(t))$, i.e.,
	$\xi(t,x)\in\Sign(u(t,x))$ a.e., such that
	\[
	\partial_tu-\Delta u+\lambda\,\xi=0
	\quad\text{a.e.\ on }Q.
	\]
	The selection is automatically measurable, since
	$\xi=(\Delta u-\partial_tu)/\lambda$ pointwise a.e., and
	$|\xi|\le1$ gives $\xi\in L^\infty(Q)\subset L^2(Q)$. Consequently,
	\[
	\Delta u=\partial_tu+\lambda\,\xi\in L^2(Q),
	\]
	and the $H^2$ elliptic regularity on convex or $C^{1,1}$ domains
	(see \cite[Theorem~3.2.1.2]{Grisvard1985} and recall that $\Omega=(0,1)^2$ is
	convex) yields
	$u\in L^2\bigl(0,T;H^2(\Omega)\cap H_0^1(\Omega)\bigr)$. Hence
	$u\in\mathcal{W}(0,T)$ is a strong solution of
	\eqref{eq:relay-pde}, recovering $\mathcal{S}(u_0)\neq\varnothing$
	independently of Theorem~\ref{thm:existence-pde}.
	
	The stated regularity can also be read off from explicit $L^2$ estimates.
	For gradient flows
	$-\partial_tu(t)\in\partial E(u(t))$ with $u_0\in D(E)$, the chain
	rule for convex subdifferentials \cite[Lemme~3.3]{Brezis1973} gives
	$\frac{d}{dt}E(u(t))=-\|\partial_tu(t)\|_{L^2(\Omega)}^2$ for a.e.\
	$t$, whence, integrating on $(0,T)$ and using $E\ge0$,
	\begin{align*}
	\int_0^T\|\partial_tu(t)\|_{L^2(\Omega)}^2\,dt
	& = E(u_0)-E(u(T))
	\le E(u_0)
	\\ &	
	= \tfrac12\|\nabla u_0\|_{L^2(\Omega)}^2
	+\lambda\|u_0\|_{L^1(\Omega)}.
	\end{align*}
	The pointwise bound $|\xi|\le1$ a.e.\ on $Q$ also gives
	\[
	\|\xi\|_{L^2(Q)}\le|Q|^{1/2}=(T\,|\Omega|)^{1/2}.
	\]
	Combining these bounds with the triangle inequality in $L^2(Q)$,
	\[
	\|\Delta u\|_{L^2(Q)}
	\le\|\partial_tu\|_{L^2(Q)}+\lambda\,\|\xi\|_{L^2(Q)}
	\le E(u_0)^{1/2}+\lambda\,(T\,|\Omega|)^{1/2},
	\]
	and the constant in the elliptic estimate
	$\|u(t)\|_{H^2(\Omega)}\le C(\Omega)\,\|\Delta u(t)\|_{L^2(\Omega)}$
	of \cite[Theorem~3.2.1.2]{Grisvard1985} depends only on the (convex)
	domain. All three memberships
	$\partial_tu\in L^2(0,T;L^2(\Omega))$, $\xi\in L^2(Q)$, and
	$u\in L^2(0,T;H^2(\Omega)\cap H_0^1(\Omega))$ therefore come with
	explicit bounds in terms of $E(u_0)$, $\lambda$, $T$, and $|\Omega|$
	alone.
	
	Uniqueness of strong solutions also holds. Let
	$u_1,u_2\in\mathcal{W}(0,T)$ be strong solutions with the same initial
	datum, with pointwise selections $\xi_i(t,x)\in\Sign(u_i(t,x))$ such
	that $\partial_tu_i-\Delta u_i=-\lambda\,\xi_i$ a.e.\ on $Q$.
	Subtracting the two equations and testing with $u_1-u_2$ gives, for
	a.e.\ $t\in(0,T)$,
	\[
	\frac12\frac{d}{dt}\|u_1-u_2\|_{L^2(\Omega)}^2
	+\|\nabla(u_1-u_2)\|_{L^2(\Omega)}^2
	+\lambda\int_\Omega(\xi_1-\xi_2)(u_1-u_2)\,dx
	=0.
	\]
	Since the graph of $\Sign$ is monotone,
	$(\xi_1-\xi_2)(u_1-u_2)\ge0$ a.e., whence
	$\frac{d}{dt}\|u_1-u_2\|_{L^2(\Omega)}^2\le0$; as $u_1(0)=u_2(0)=u_0$,
	we conclude $u_1=u_2$. Thus $\mathcal{S}(u_0)$ is a singleton, in sharp
	contrast with the genuinely multivalued solution sets of
	Section~\ref{subsec:linear-control}.
\end{remark}

For the finite-time extinction analysis, multiply \eqref{eq:relay-pde} by $u(t,\cdot)$ and integrate over $\Omega$.
Using that $s\cdot v = -\lambda|s|$ for every $v\in\Phi(s)$ and every
$s\in\mathbb{R}$, a formal computation (rigorous for strong solutions) gives
the energy identity
\begin{equation}
	\label{eq:relay-energy}
	\frac12\frac{d}{dt}\|u(t)\|_{L^2(\Omega)}^2
	+
	\|\nabla u(t)\|_{L^2(\Omega)}^2
	=
	-\lambda\,\|u(t)\|_{L^1(\Omega)}.
\end{equation}
The extra dissipation term $-\lambda\|u(t)\|_{L^1(\Omega)}$, absent in the
classical heat equation ($\lambda=0$), is responsible for a qualitatively
different phenomenon: the solution reaches the equilibrium $u\equiv0$ in
\emph{finite time} rather than merely decaying exponentially --- a behavior
familiar from evolution problems governed by dissipation terms that are
positively homogeneous of degree one, such as the total variation flow
\cite{AndreuCaselleMazon2004}. The energy identity
\eqref{eq:relay-energy} alone does not yield finite-time extinction. For
the nonnegative initial datum \eqref{eq:relay-u0}, extinction, an explicit
upper bound for the extinction time, and monotonicity with respect to
$\lambda$ follow from the comparison argument in the next proposition. Once $u(t,\cdot)$
vanishes on part of $\Omega$, the relay may ``slide'' there
($\Delta u(t,x)\in[-\lambda,\lambda]$
pointwise, consistently with $u\equiv0$), anchoring the solution at
equilibrium. A classical pointwise-residual PINN would have to choose one branch of
$\Phi$ in advance and would not represent this sliding behavior directly.

\begin{proposition}[Finite-time extinction and monotonicity in $\lambda$]
	\label{prop:extinction}
	Let $u_0\in H_0^1(\Omega)\cap L^\infty(\Omega)$ with $u_0\ge0$, and,
	for $\lambda>0$, let $u_\lambda$ denote the unique strong solution of
	\eqref{eq:relay-pde}--\eqref{eq:relay-Phi} provided by
	Remark~\ref{rem:relay-wellposed}, regarded as defined for all
	$t\ge0$. Then:
	\begin{enumerate}
		\item[\rm(a)] $u_\lambda(t,x)\ge0$ for a.e.\ $(t,x)$;
		\item[\rm(b)] $0\le u_\lambda(t,x)\le
		\bigl(\|u_0\|_{L^\infty(\Omega)}-\lambda t\bigr)_+$ a.e.; in
		particular, $u_\lambda(t,\cdot)\equiv0$ for every
		$t\ge\|u_0\|_{L^\infty(\Omega)}/\lambda$, so the extinction time
		$t^*(\lambda):=\inf\{t\ge0:u_\lambda(t,\cdot)\equiv0\}$ satisfies
		\[
		t^*(\lambda)\le\frac{\|u_0\|_{L^\infty(\Omega)}}{\lambda};
		\]
		\item[\rm(c)] if $\lambda_2>\lambda_1>0$, then
		$0\le u_{\lambda_2}\le u_{\lambda_1}$ a.e., and consequently
		$t^*(\lambda_2)\le t^*(\lambda_1)$.
	\end{enumerate}
\end{proposition}

\begin{proof}
	Throughout, $u=u_\lambda$ represents the strong solution and
	$\xi(t,x)\in\Sign(u(t,x))$ the associated selection with
	$\partial_tu-\Delta u=-\lambda\xi$ a.e.\ (cf.\
	Remark~\ref{rem:relay-wellposed}); all the test computations below are
	rigorous for strong solutions.
	
	(a) Let $u^-:=\max(-u,0)$, so that $u^-(t,\cdot)\in H_0^1(\Omega)$ for
	a.e.\ $t$. Multiplying the equation by $-u^-$ and integrating over
	$\Omega$ gives, for a.e.\ $t$,
	\[
	\frac12\frac{d}{dt}\|u^-(t)\|_{L^2(\Omega)}^2
	+\|\nabla u^-(t)\|_{L^2(\Omega)}^2
	=\lambda\int_\Omega\xi\,u^-\,dx
	=-\lambda\int_\Omega u^-\,dx\le0,
	\]
	where we used that $\xi=-1$ a.e.\ on $\{u<0\}$ and $u^-=0$ elsewhere.
	Since $u^-(0)=u_0^-=0$, we conclude $u^-\equiv0$, i.e., $u\ge0$.
	
	(b) Set $M:=\|u_0\|_{L^\infty(\Omega)}$, $t_0:=M/\lambda$, and
	consider the spatially constant supersolution
	$v(t):=M-\lambda t$ on $[0,t_0)$. Let $w:=u-v$. For $t\in(0,t_0)$ one
	has $v(t)>0$, hence $w<0$ on $(0,t_0)\times\partial\Omega$ and
	$w^+:=\max(w,0)$ satisfies $w^+(t,\cdot)\in H_0^1(\Omega)$. Moreover,
	$\partial_tv-\Delta v=-\lambda$, so
	\[
	\partial_tw-\Delta w=\lambda(1-\xi)
	\quad\text{a.e.\ in }(0,t_0)\times\Omega,
	\]
	and on $\{w>0\}$ we have $u>v>0$, hence $\xi=1$ and the right-hand
	side vanishes there. Testing with $w^+$ therefore gives
	\[
	\frac12\frac{d}{dt}\|w^+(t)\|_{L^2(\Omega)}^2
	+\|\nabla w^+(t)\|_{L^2(\Omega)}^2
	=\lambda\int_{\{w>0\}}(1-\xi)\,w^+\,dx=0,
	\]
	and since $w^+(0)=(u_0-M)^+=0$, we obtain $u\le v$ on $[0,t_0)$.
	Together with (a), this proves the pointwise bound in (b) on
	$[0,t_0)$; by continuity of $t\mapsto u(t)$ in $L^2(\Omega)$,
	$u(t_0)=0$. Finally, the identically zero function is a strong
	solution of \eqref{eq:relay-pde} with zero initial datum
	(indeed $0=\partial_t0-\Delta 0\in[-\lambda,\lambda]=\Phi(0)$), so by
	the uniqueness established in Remark~\ref{rem:relay-wellposed},
	restarted at time $t_0$, $u(t)\equiv0$ for all $t\ge t_0$.
	
	(c) Let $w:=u_{\lambda_2}-u_{\lambda_1}$, with selections
	$\xi_i\in\Sign(u_{\lambda_i})$. Then
	$\partial_tw-\Delta w=-\lambda_2\xi_2+\lambda_1\xi_1$, and on
	$\{w>0\}$ we have $u_{\lambda_2}>u_{\lambda_1}\ge0$ by (a), hence
	$\xi_2=1$ and
	$-\lambda_2\xi_2+\lambda_1\xi_1\le-\lambda_2+\lambda_1<0$. Both
	$u_{\lambda_i}$ vanish on $\partial\Omega$, so
	$w^+(t,\cdot)\in H_0^1(\Omega)$, and testing with $w^+$ gives
	\[
	\frac12\frac{d}{dt}\|w^+(t)\|_{L^2(\Omega)}^2
	+\|\nabla w^+(t)\|_{L^2(\Omega)}^2
	\le\int_{\{w>0\}}(\lambda_1\xi_1-\lambda_2)\,w^+\,dx\le0 .
	\]
	With $w^+(0)=0$ this yields $w^+\equiv0$, i.e.,
	$u_{\lambda_2}\le u_{\lambda_1}$, and the ordering of the extinction
	times follows immediately from the definition of $t^*$.
\end{proof}

For the initial datum \eqref{eq:relay-u0} one has
$\|u_0\|_{L^\infty(\Omega)}=1$, so Proposition~\ref{prop:extinction}(b)
gives the a priori bound $t^*(\lambda)\le1/\lambda$, i.e., $2.5$, $1.6\overline{6}$,
and $1.25$ for $\lambda\in\{0.4,0.6,0.8\}$. The bound is far from sharp
(the observed extinction times reported below are an order of magnitude
smaller), since the comparison argument discards the additional decay
produced by diffusion; its role is purely qualitative: it certifies
finite-time extinction and the monotone dependence on $\lambda$.
Identity~\eqref{eq:relay-energy} and
Proposition~\ref{prop:extinction} also provide solver-independent
numerical checks: the decay of $\|u_\theta(t)\|_{L^2(\Omega)}^2$ along the
trained trajectory should be strictly faster than that of the corresponding
heat equation ($\lambda=0$) with the same initial datum, and the detected
extinction times must be nonincreasing in $\lambda$ (strictly decreasing in
our experiments).

\paragraph{Adaptive collocation sampling}
Because $\Phi$ has a genuine jump discontinuity at $s=0$, the distance
residual is hardest to drive to zero in a neighborhood of the (a priori
unknown, time-dependent) extinction front $$\{(t,x):u_\theta(t,x)=0\}.$$ This
is precisely where a smooth network output must approximate a discontinuous
right-hand side, and our experiments confirm that the residual after
training is concentrated there rather than spread uniformly over $Q$.
Purely uniform collocation under-resolves this thin, moving region relative
to its share of the loss. We therefore complement uniform collocation with
a \emph{front-focused, state-based} adaptive sampling scheme inspired by
the RAR method of Lu, Meng, Mao, and Karniadakis
\cite{LuMengMaoKarniadakis2021DeepXDE}; unlike residual-based RAR, our
scheme selects points by the smallest $|u_\theta|$ (a proxy for proximity
to the extinction front) rather than by the largest residual. Precisely, at fixed intervals during training,
a large pool of candidate points is drawn uniformly over $Q$, the points
with the smallest $|u_\theta|$ are retained as an auxiliary ``hard pool''
tracking the current estimate of the extinction front, and a fraction of
every subsequent training batch is drawn from this pool — refreshed
periodically as the front moves — while the remaining points continue to be
sampled uniformly over $Q$ to preserve global coverage of the domain.
In our experiments the batch size is $N_{\rm col}=2000$ collocation points per
epoch; the adaptive fraction is $30\%$ (600 points), drawn from a hard pool
of $n_{\rm pool}=4000$ candidate points selected from a uniform draw of
$20\,000$ points over $Q$ by retaining those with smallest $|u_\theta|$.
The hard pool is refreshed every $500$ training epochs as the extinction
front advances; the remaining $70\%$ of each batch is drawn uniformly over
$Q$ to preserve global domain coverage. The campaign therefore includes a controlled sampler ablation
(front-based adaptive vs.\ residual-based RAR vs.\ uniform collocation) at
the baseline configuration. At these settings the three samplers are
statistically indistinguishable on every reported metric
(Table~\ref{tab:relay-ablation}); the front-based scheme is retained as a
safeguard for thinner bands rather than as a demonstrated speedup.

\paragraph{Reference solution}
Unlike the bang-bang example of Section~\ref{subsec:linear-control}, where
$F$ is genuinely multivalued and the ensemble/Hausdorff comparison of
Section~\ref{subsec:linear-control} is required, uniqueness of strong
solutions holds here, as proved in Remark~\ref{rem:relay-wellposed}:
$\mathcal{S}(u_0)$ is a singleton. A single trained
DR-PINN can therefore be validated directly against one reference solution,
obtained by a first-order operator-splitting scheme on a uniform
$49\times49$ interior grid with time step $\Delta t=5\times10^{-4}$.
Each time step alternates an implicit diffusion solve,
$v=(I-\Delta t\,\Delta_h)^{-1}u^n$ (with the 5-point finite-difference
Laplacian $\Delta_h$ and homogeneous Dirichlet boundary conditions), and an
exact proximal step,
\[
u^{n+1}=\Sign(v)\max(|v|-\lambda\Delta t,\,0),
\]
the pointwise proximal operator of $s\mapsto\lambda\Delta t\,|s|$.
This scheme is able to represent exact zeros: any grid point for which
$|v_i|\le\lambda\Delta t$ is set to exactly zero \emph{at that time
step} by the soft-thresholding operation (the subsequent diffusion
half-step may of course make it nonzero again). In particular, once the
whole field has decayed below the threshold, the scheme reproduces the
identically zero equilibrium without smoothing, so finite-time extinction
is captured sharply. Being a first-order splitting on a fixed grid, the
scheme carries discretization errors of order $O(\Delta t)+O(h^2)$ that
may systematically shift the detected extinction instants; a refinement
study therefore quantifies the reliability of the reported
$t^*_{\mathrm{ref}}$. Table~\ref{tab:relay-ref-convergence} reports, for
each solver configuration, the detected extinction time $t^*$ (first grid
time at which the discrete field is exactly zero; the zero state is
persistent up to $T$ in every run), its change $\Delta t^*$ with respect
to the baseline configuration, and the maximum change of the
$L^2(\Omega)$-norm decay curve and of the solution profile shortly before
extinction. Refining the mesh from $49\times49$ to $97\times97$ and
$193\times193$ at fixed $\Delta t$ leaves the detected $t^*$ unchanged
for all three values of $\lambda$, so the baseline mesh is already
converged in space; halving and quartering $\Delta t$ shifts $t^*$ by at
most $6.3\times10^{-4}$ --- about one baseline time step, which is also
the quantization of the detection itself, since $t^*$ is read off the
discrete time grid --- while the norm curves and pre-extinction profiles
change by at most $7\times10^{-4}$ and $2\times10^{-4}$, respectively.
The reference extinction times are therefore reliable to three decimal
places, with an uncertainty of order one baseline step in the fourth, and
the monotone ordering $t^*(0.8)<t^*(0.6)<t^*(0.4)$ is preserved at every
refinement level. The refinement study is part of the replication
pipeline (reference stage of this experiment) and is regenerated together
with all other reported numbers.

\begin{table}[htbp]
  \centering
  \small
  \setlength{\tabcolsep}{5pt}
\begin{tabular}{lcccccc}
\toprule
$\lambda$ & grid & $\Delta t$ & $t^*$ & $\Delta t^*$ & $\max|\Delta\|u\||$ & $\max|\Delta u_{\mathrm{pre}}|$ \\
\midrule
$0.4$ & $49\times49$ & $0.0005$ & 0.1845 & +0.0000 & $0$ & $0$ \\
$0.4$ & $49\times49$ & $0.00025$ & 0.1840 & -0.0005 & $4.5\times 10^{-4}$ & $1.7\times 10^{-4}$ \\
$0.4$ & $49\times49$ & $0.000125$ & 0.1839 & -0.0006 & $6.7\times 10^{-4}$ & $1.1\times 10^{-4}$ \\
$0.4$ & $97\times97$ & $0.0005$ & 0.1845 & +0.0000 & $5.1\times 10^{-5}$ & $4.1\times 10^{-5}$ \\
$0.4$ & $97\times97$ & $0.00025$ & 0.1840 & -0.0005 & $5.0\times 10^{-4}$ & $1.3\times 10^{-4}$ \\
$0.4$ & $193\times193$ & $0.00025$ & 0.1840 & -0.0005 & $5.1\times 10^{-4}$ & $1.2\times 10^{-4}$ \\
$0.6$ & $49\times49$ & $0.0005$ & 0.1645 & +0.0000 & $0$ & $0$ \\
$0.6$ & $49\times49$ & $0.00025$ & 0.1643 & -0.0003 & $4.3\times 10^{-4}$ & $8.9\times 10^{-5}$ \\
$0.6$ & $49\times49$ & $0.000125$ & 0.1640 & -0.0005 & $6.4\times 10^{-4}$ & $1.3\times 10^{-4}$ \\
$0.6$ & $97\times97$ & $0.0005$ & 0.1645 & +0.0000 & $5.3\times 10^{-5}$ & $5.3\times 10^{-5}$ \\
$0.6$ & $97\times97$ & $0.00025$ & 0.1643 & -0.0003 & $4.8\times 10^{-4}$ & $8.9\times 10^{-5}$ \\
$0.6$ & $193\times193$ & $0.00025$ & 0.1643 & -0.0003 & $5.0\times 10^{-4}$ & $8.7\times 10^{-5}$ \\
$0.8$ & $49\times49$ & $0.0005$ & 0.1505 & +0.0000 & $0$ & $0$ \\
$0.8$ & $49\times49$ & $0.00025$ & 0.1502 & -0.0003 & $4.1\times 10^{-4}$ & $1.2\times 10^{-4}$ \\
$0.8$ & $49\times49$ & $0.000125$ & 0.1501 & -0.0004 & $6.2\times 10^{-4}$ & $1.8\times 10^{-4}$ \\
$0.8$ & $97\times97$ & $0.0005$ & 0.1505 & +0.0000 & $5.6\times 10^{-5}$ & $6.5\times 10^{-5}$ \\
$0.8$ & $97\times97$ & $0.00025$ & 0.1502 & -0.0003 & $4.7\times 10^{-4}$ & $1.2\times 10^{-4}$ \\
$0.8$ & $193\times193$ & $0.00025$ & 0.1502 & -0.0003 & $4.8\times 10^{-4}$ & $1.1\times 10^{-4}$ \\
\bottomrule
\end{tabular}
{lcccccc}
  \toprule
  $\lambda$ & grid & $\Delta t$ & $t^*$ & $\Delta t^*$ & $\max|\Delta\|u\||$ & $\max|\Delta u_{\mathrm{pre}}|$ \\
  \midrule
  $0.4$ & $49\times49$ & $0.0005$ & 0.1845 & +0.0000 & $0$ & $0$ \\
  $0.4$ & $49\times49$ & $0.00025$ & 0.1840 & -0.0005 & $4.5\times 10^{-4}$ & $1.7\times 10^{-4}$ \\
  $0.4$ & $49\times49$ & $0.000125$ & 0.1839 & -0.0006 & $6.7\times 10^{-4}$ & $1.1\times 10^{-4}$ \\
  $0.4$ & $97\times97$ & $0.0005$ & 0.1845 & +0.0000 & $5.1\times 10^{-5}$ & $4.1\times 10^{-5}$ \\
  $0.4$ & $97\times97$ & $0.00025$ & 0.1840 & -0.0005 & $5.0\times 10^{-4}$ & $1.3\times 10^{-4}$ \\
  $0.4$ & $193\times193$ & $0.00025$ & 0.1840 & -0.0005 & $5.1\times 10^{-4}$ & $1.2\times 10^{-4}$ \\
  $0.6$ & $49\times49$ & $0.0005$ & 0.1645 & +0.0000 & $0$ & $0$ \\
  $0.6$ & $49\times49$ & $0.00025$ & 0.1643 & -0.0003 & $4.3\times 10^{-4}$ & $8.9\times 10^{-5}$ \\
  $0.6$ & $49\times49$ & $0.000125$ & 0.1640 & -0.0005 & $6.4\times 10^{-4}$ & $1.3\times 10^{-4}$ \\
  $0.6$ & $97\times97$ & $0.0005$ & 0.1645 & +0.0000 & $5.3\times 10^{-5}$ & $5.3\times 10^{-5}$ \\
  $0.6$ & $97\times97$ & $0.00025$ & 0.1643 & -0.0003 & $4.8\times 10^{-4}$ & $8.9\times 10^{-5}$ \\
  $0.6$ & $193\times193$ & $0.00025$ & 0.1643 & -0.0003 & $5.0\times 10^{-4}$ & $8.7\times 10^{-5}$ \\
  $0.8$ & $49\times49$ & $0.0005$ & 0.1505 & +0.0000 & $0$ & $0$ \\
  $0.8$ & $49\times49$ & $0.00025$ & 0.1502 & -0.0003 & $4.1\times 10^{-4}$ & $1.2\times 10^{-4}$ \\
  $0.8$ & $49\times49$ & $0.000125$ & 0.1501 & -0.0004 & $6.2\times 10^{-4}$ & $1.8\times 10^{-4}$ \\
  $0.8$ & $97\times97$ & $0.0005$ & 0.1505 & +0.0000 & $5.6\times 10^{-5}$ & $6.5\times 10^{-5}$ \\
  $0.8$ & $97\times97$ & $0.00025$ & 0.1502 & -0.0003 & $4.7\times 10^{-4}$ & $1.2\times 10^{-4}$ \\
  $0.8$ & $193\times193$ & $0.00025$ & 0.1502 & -0.0003 & $4.8\times 10^{-4}$ & $1.1\times 10^{-4}$ \\
  \bottomrule
  \end{tabular}}
  \caption{Refinement study for the reference solver: detected extinction
    time $t^*$ (first grid time with an exactly zero discrete field;
    persistent up to $T$ in all runs), its change $\Delta t^*$ with
    respect to the baseline configuration ($49\times49$,
    $\Delta t=5\times10^{-4}$), the maximum change of the
    $L^2(\Omega)$-norm decay curve, and the maximum change of the
    solution profile shortly before extinction (at $0.9\,t^*$,
    interpolated to the baseline grid). At fixed $\Delta t$ the detected
    $t^*$ is independent of the mesh; the variation under
    $\Delta t$-refinement is of the order of one baseline time step and
    is dominated by the temporal quantization of the detection.}
  \label{tab:relay-ref-convergence}
\end{table}

\paragraph{Network architecture and loss function}
The trial solution takes the hard-constrained form
\begin{equation}
	\label{eq:trial-parabolic}
	u_\theta(t,x,y)
	= u_0(x,y) + t\,x(1-x)\,y(1-y)\,N_\theta(t,x,y),
\end{equation}
where $N_\theta:\mathbb{R}^3\to\mathbb{R}$ is a fully connected neural
network. The parametrization~\eqref{eq:trial-parabolic} enforces the
initial condition $u_\theta(0,\cdot)=u_0$ and the homogeneous Dirichlet
boundary condition exactly for every $\theta$ (the polynomial factor
vanishes on $\partial\Omega$), so, as in
Section~\ref{subsec:linear-control}, the DR-PINN loss reduces to the pure
inclusion term: it is~\eqref{eq:pde-distance-loss} with
$\Phi$ as in \eqref{eq:relay-Phi} and $\lambda_0=\lambda_b=0$. Since $\Phi(s)$ is an interval of
$\mathbb{R}$, the projection $\Pi_{\Phi(u_\theta(t_i,x_i))}$ needed to
evaluate the distance residual is explicit, requiring no quadratic program:
\[
\dist^2\!\bigl(z,\Phi(s)\bigr)
=
\begin{cases}
(z+\lambda)^2, & s>0,\\[2pt]
\max\bigl(|z|-\lambda,\,0\bigr)^2, & s=0,\\[2pt]
(z-\lambda)^2, & s<0,
\end{cases}
\qquad
z:=\partial_t u_\theta-\Delta u_\theta.
\]
With the \emph{exact} multifunction $\Phi$ of \eqref{eq:relay-Phi}, the
interval branch of this formula is triggered only when $u_\theta$ is
exactly zero in floating-point arithmetic --- an event of probability
zero for the smooth ansatz \eqref{eq:trial-parabolic} --- so almost every
collocation point is forced onto one of the two sign branches, and the
training landscape is essentially that of a discontinuous target. Our
experiments retain this exact-relay configuration as a negative control,
but the primary experiments train against the \emph{$\varepsilon$-banded
relay}
\begin{equation}
	\label{eq:relay-Phi-eps}
	\Phi_\varepsilon(s) :=
	\begin{cases}
		\{-\lambda\}, & s>\varepsilon,\\[2pt]
		[-\lambda,\lambda], & |s|\le\varepsilon,\\[2pt]
		\{\lambda\}, & s<-\varepsilon,
	\end{cases}
	\qquad \varepsilon\ge0,
\end{equation}
which applies the interval branch on the whole band
$\{|s|\le\varepsilon\}$ (so the middle case of the projection formula
above is used whenever $|u_\theta|\le\varepsilon$) and reduces to $\Phi$
for $\varepsilon=0$.

For the banded experiments, the relaxed reaction term is $\Phi_\varepsilon$. The map $\Phi_\varepsilon$ satisfies \ref{P1}--\ref{P3} with the same constant
$m_0=\lambda$: its values are nonempty, closed, convex intervals; it is
upper semicontinuous (it is locally constant outside
$\{|s|\le\varepsilon\}$, and at $|s|\le\varepsilon$ its value is already the
full interval $[-\lambda,\lambda]$, which contains every nearby value);
and the growth bound is unchanged.
Theorem~\ref{thm:consistency-parabolic} therefore applies to the inclusion
with reaction term $\Phi_\varepsilon$.

Moreover, $\Phi(s)\subseteq\Phi_\varepsilon(s)$ for every $s$, hence
$\dist(z,\Phi_\varepsilon(s))\le\dist(z,\Phi(s))$ pointwise, every
solution of the exact relay inclusion is a solution of the
$\varepsilon$-banded one, and the trained loss is a lower bound for the
exact one. The converse fails: the $\varepsilon$-banded inclusion has, in
general, strictly more solutions, and --- unlike the exact relay, whose
strong solution is unique by the monotonicity argument of
Remark~\ref{rem:relay-wellposed} --- the graph of $-\Phi_\varepsilon$ is
not monotone for $\varepsilon>0$, so that argument does not carry over.
Remark~\ref{rem:relay-witness} below makes this loss of rigidity explicit
and quantifies when it matters.

Since a smooth surrogate cannot extinguish exactly, we compare the
solutions through the \emph{band-entry time}
\begin{equation}
	\label{eq:relay-tband}
	t_\varepsilon(u) := \inf\bigl\{\, t\in[0,T] :
	\|u(t,\cdot)\|_{L^\infty(\Omega)} < \varepsilon \,\bigr\},
\end{equation}
evaluated on a uniform time grid of step $2.5\times10^{-3}$, together
with its \emph{persistent} variant (the first grid time after which the
sup-norm stays below $\varepsilon$ up to $T$). For the reference
solution, $t_\varepsilon\to t^*$ as $\varepsilon\to0^+$, so the
band-entry time is a well-posed proxy for the extinction time that is
meaningful for smooth surrogates; the sup-norm (rather than an $L^2$
threshold) is used because the band $\{|u|\le\varepsilon\}$ is exactly
the region where $\Phi_\varepsilon$ differs from $\Phi$.

\begin{remark}[Steady sliding states of the $\varepsilon$-banded relay]
	\label{rem:relay-witness}
	Let $\varphi\in H^2(\Omega)\cap H^1_0(\Omega)$ solve
	$-\Delta\varphi=1$ in $\Omega$, $\varphi>0$ (the torsion function;
	for $\Omega=(0,1)^2$,
	$\|\varphi\|_{L^\infty(\Omega)}\approx0.0736$), and set
	$w_\lambda:=\lambda\,\varphi$. Then
	$\partial_t w_\lambda-\Delta w_\lambda=\lambda$, so $u=w_\lambda$ is a
	stationary strong solution of the $\varepsilon$-banded inclusion
	$\partial_t u-\Delta u\in\Phi_\varepsilon(u)$ if and only if
	$\|w_\lambda\|_{L^\infty(\Omega)}
	=\lambda\|\varphi\|_{L^\infty(\Omega)}\le\varepsilon$: in that case
	$|w_\lambda|\le\varepsilon$ everywhere, the band branch is applied, and
	$\lambda\in \Phi_\varepsilon (w_\lambda)= [-\lambda,\lambda]$. Consequently, when
	$\lambda\|\varphi\|_{L^\infty(\Omega)}\le\varepsilon$ the
	$\varepsilon$-banded dynamics does \emph{not} force extinction:
	nonzero steady states exist inside the band, and after band entry a
	candidate with vanishing $\Phi_\varepsilon$-residual may approximate
	any element of this sliding funnel, not necessarily $u\equiv0$. The
	validation below therefore reports, for each trained network, the
	distance of the post-band tail both to zero and to the extremal
	witness $w_\lambda$, and records whether the witness is itself
	admissible for the given pair $(\lambda,\varepsilon)$
	($\lambda\|\varphi\|_{L^\infty}\approx0.0295,\ 0.0442,\ 0.0589$ for
	$\lambda=0.4,\ 0.6,\ 0.8$). For $\varepsilon=0$ the funnel collapses,
	$w_\lambda$ is never admissible, and uniqueness holds as in
	Remark~\ref{rem:relay-wellposed}.
\end{remark}
We also inspect branch selection by plotting the empirical cloud
\[
\bigl(u_\theta(t_i,x_i),\,\partial_t u_\theta(t_i,x_i)-\Delta u_\theta(t_i,x_i)\bigr)
\]
over the three-branch graph of $\Phi_\varepsilon$, whose middle step has
width $2\varepsilon$. Concentration of this cloud on the graph shows which
admissible value is selected at the collocation points and provides a
diagnostic specific to the set-valued formulation.

\paragraph{Numerical results}

Figure~\ref{fig:relay-extinction-ref} displays the $L^2(\Omega)$-norm decay curves
for the reference solver with $\lambda\in\{0,0.4,0.6,0.8\}$ over $[0,T]$.
The classical heat equation ($\lambda=0$) decays exponentially and
remains well above zero at $t=T=0.3$, whereas all three
relay reference solutions reach exact numerical extinction within
$[0,0.19]$, with the extinction time decreasing monotonically as
$\lambda$ increases.

\begin{figure}[htbp]
  \centering
  \includegraphics[width=0.72\textwidth]{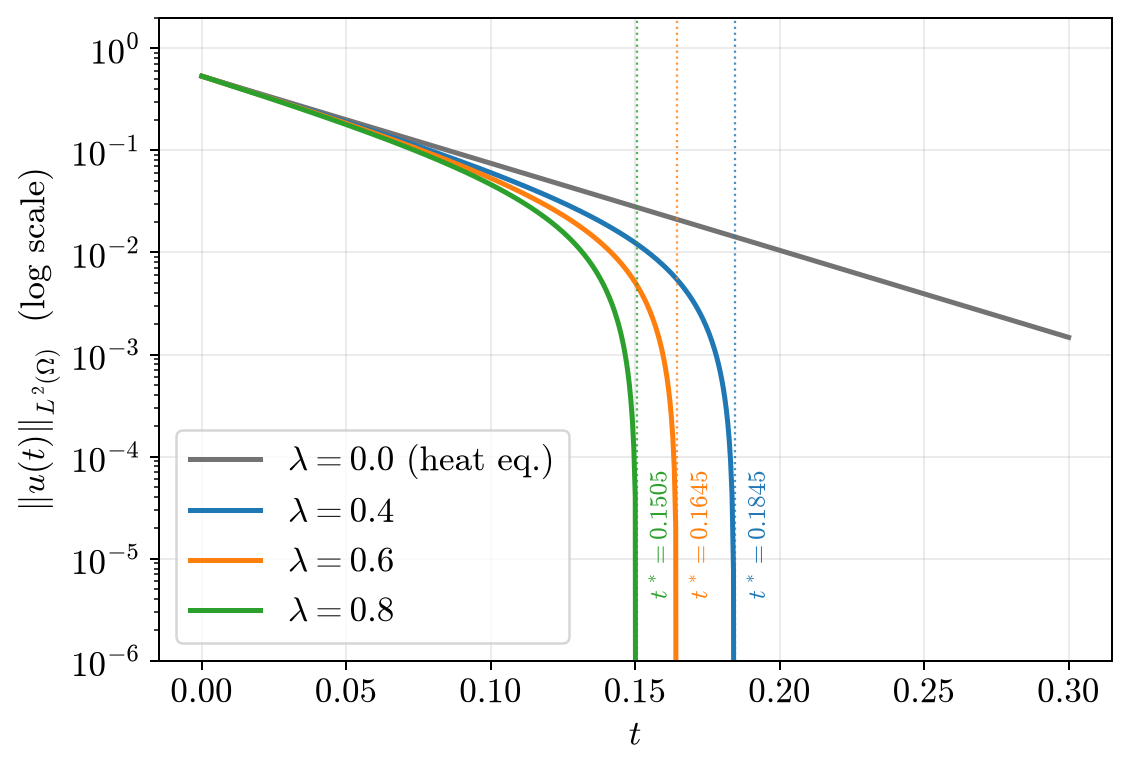}
  \caption{Reference solution: $\|u(t)\|_{L^2(\Omega)}$ on a log scale for
    $\lambda\in\{0,0.4,0.6,0.8\}$. The vertical dotted lines mark the finite-time
    extinction instants $t^*$ of the three relay cases; the $\lambda=0$ heat-equation
    curve does not extinguish within $[0,T]$. Curves are clipped below $10^{-6}$
    for readability; after extinction the reference norm is exactly zero
    (drop below the axis).}
  \label{fig:relay-extinction-ref}
\end{figure}

The network is a fully connected $\tanh$
network with five hidden layers of width $128$ and Glorot-normal
initialization, trained for $40\,000$ epochs with the Adam optimizer
under a hold-then-decay learning-rate schedule: the rate is held at
$2\times10^{-3}$ for the first $5\,000$ epochs and then decayed
exponentially with rate $0.85$ per $3\,000$ steps. The same architecture
and schedule are used for every configuration of the campaign.

\paragraph{Experimental design}
The primary campaign comprises $45$ independent tra\-inings (with $5$ seeds
$\{11,23,47,59,71\}$ per configuration): a $\lambda$-sweep
$\lambda\in\{0.4,0.6,0.8\}$ at fixed $\varepsilon=0.05$; an
$\varepsilon$-sweep $\varepsilon\in\{0,\,0.01,\,0.02,\,0.05,\,0.1\}$ at
fixed $\lambda=0.6$, with $\varepsilon=0$ (the exact relay) serving as a
negative control; and a sampler ablation (front-adaptive vs.\
residual-based RAR vs.\ uniform) at
$(\lambda,\varepsilon)=(0.6,0.05)$. After training, each network is
evaluated on a dense space--time grid (time step $2.5\times10^{-3}$) and
on a fresh cloud of $4\,000$ collocation points: we record the band-entry
times \eqref{eq:relay-tband} (first and persistent), the
$\Phi_\varepsilon$ distance residual, the number of cloud points
violating the inclusion beyond a tolerance of $10^{-4}$, the sup-norm
discrepancy of the pre-band $L^2(\Omega)$-norm curve against the
reference, and the tail distances of Remark~\ref{rem:relay-witness}.
Reference band-entry times $t_{\varepsilon,\mathrm{ref}}$ are computed
from the operator-splitting solution of the exact relay (whose refinement
study, Table~\ref{tab:relay-ref-convergence}, is unchanged) on the same
evaluation grid.

Figure~\ref{fig:relay-training} shows the training and validation curves
of all $45$ runs. The qualitative shape of earlier single-run
experiments is reproduced uniformly: an initial plateau at
$\mathcal{L}\approx20$, an abrupt drop by roughly two orders of magnitude
between epochs ${\approx}5\,000$ and ${\approx}8\,000$, and a steady
decrease thereafter. The five $\varepsilon=0$ runs are clearly separated
from all banded runs, stagnating an order of magnitude higher (final
validation mean $\approx0.44$, seed median, vs.\ $\approx0.027$ at
$\varepsilon=0.05$). This separation is quantified below.

\begin{figure}[htbp]
  \centering
  \includegraphics[width=0.80\textwidth]{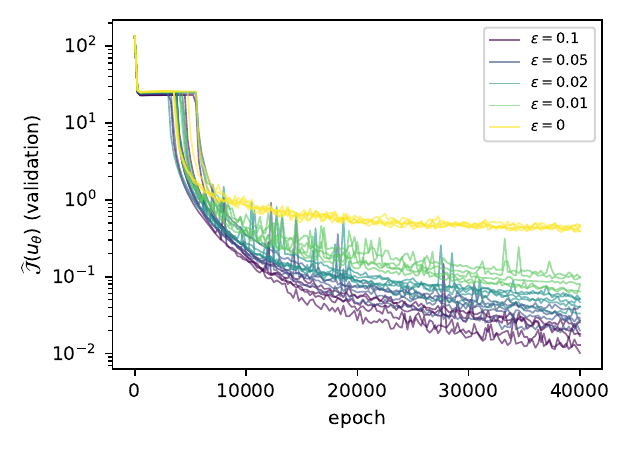}
  \caption{Training history of the inclusion loss for all $45$ runs of
    the campaign, colored by $\varepsilon$ (validation on a fixed batch
    of $4\,000$ points). All banded runs ($\varepsilon>0$) escape the
    initial plateau and continue to decrease; the five exact-relay runs
    ($\varepsilon=0$, top curves) stagnate an order of magnitude higher.}
  \label{fig:relay-training}
\end{figure}

\paragraph{Band-entry times: quantitative agreement with the reference}
Table~\ref{tab:relay-extinction} reports the $\lambda$-sweep at
$\varepsilon=0.05$. For every $\lambda$ and every seed the first and the
persistent band crossings coincide, and the seed-median band-entry time
matches the reference band-entry time \emph{exactly} to the resolution of
the evaluation grid, with seed ranges of at most one grid step
($2.5\times10^{-3}$). In particular, the monotone ordering
$t_\varepsilon(\lambda{=}0.8)<t_\varepsilon(0.6)<t_\varepsilon(0.4)$ of
the reference is reproduced by every individual run. The
$\varepsilon$-sweep (Table~\ref{tab:relay-eps}) shows the same agreement
across $\varepsilon\in\{0.02,0.05,0.1\}$ --- at $\varepsilon=0.1$ all five
seeds return the identical value $t_\varepsilon=0.105$ --- while at
$\varepsilon=0.01$ the agreement degrades gracefully: the median still
sits on the reference value, the seed range widens to four grid steps,
and for one seed the first crossing is transient (the sup-norm re-emerges
above the band, and the persistent crossing occurs only at $t=0.25$),
marking the resolution limit of the smooth surrogate. At $\varepsilon=0$
the negative control reproduces the failure mode of the exact relay
documented in earlier versions of this experiment: no run achieves a
persistent crossing, two of five never cross at all, and roughly
$3\,900$ of the $4\,000$ cloud points violate the inclusion beyond
$10^{-4}$, against ${\approx}900$ at $\varepsilon=0.05$. The transition
between the two regimes is sharp and occurs already at
$\varepsilon=0.02$, i.e.\ at a band two orders of magnitude smaller than
the amplitude of the initial datum. Figure~\ref{fig:relay-sweep}
visualizes the $\lambda$-sweep, and Figure~\ref{fig:relay-eps-sweep} the
residual statistics along the $\varepsilon$-sweep.

\begin{table}[htbp]
  \centering
  \small
  \setlength{\tabcolsep}{3pt}
  \begin{tabular}{ccccccc}
    \toprule
    $\lambda$ & $t_{\varepsilon,\mathrm{ref}}$ &
    \shortstack{$t_\varepsilon$ (DR-PINN)\\median [range]} &
    \shortstack{persistent\\crossings} &
    \shortstack{violations\\(of $4\,000$)} &
    \shortstack{tail dist.\\to $0$} &
    \shortstack{witness\\admissible} \\
    \midrule
    0.4 & 0.1350 & 0.1350 [0.1350, 0.1375] & 5/5 & 920 & $1.8\times10^{-3}$ & yes \\
    0.6 & 0.1275 & 0.1275 [0.1250, 0.1275] & 5/5 & 896 & $1.8\times10^{-3}$ & yes \\
    0.8 & 0.1200 & 0.1200 [0.1200, 0.1225] & 5/5 & 867 & $4.0\times10^{-3}$ & no \\
    \bottomrule
  \end{tabular}
  \caption{$\lambda$-sweep at $\varepsilon=0.05$ (front sampler, $5$
    seeds per row; violations and tail distances are seed medians). The
    DR-PINN band-entry time \eqref{eq:relay-tband} agrees with the
    reference band-entry time to within one evaluation-grid step for
    every seed, every first crossing is persistent, and the monotone
    ordering in $\lambda$ is preserved by every individual run --- in
    contrast to the raw $10^{-3}$ threshold crossings of the exact relay
    reported as the $\varepsilon=0$ control in
    Table~\ref{tab:relay-eps}. ``Witness admissible'' records whether
    $\lambda\|\varphi\|_{L^\infty}\le\varepsilon$
    (Remark~\ref{rem:relay-witness}); in all cases the trained tail is an
    order of magnitude closer to the zero state than to the extremal
    witness (cf.\ Table~\ref{tab:relay-eps}), so the networks select the
    extinct solution even where the $\varepsilon$-banded funnel is
    nontrivial.}
  \label{tab:relay-extinction}
\end{table}

\begin{table}[htbp]
  \centering
  \small
  \setlength{\tabcolsep}{3pt}
  \begin{tabular}{cccccccc}
    \toprule
    $\varepsilon$ & $t_{\varepsilon,\mathrm{ref}}$ &
    \shortstack{$t_\varepsilon$ (DR-PINN)\\median [range]} &
    \shortstack{persist.\\cross.} &
    \shortstack{final val.\\mean} &
    \shortstack{violations\\(of $4\,000$)} &
    \shortstack{tail dist.\\to $0$} &
    \shortstack{tail dist.\\to $w_\lambda$} \\
    \midrule
    $0$    & ---    & --- ($2/5$ no crossing)   & 0/5 & 0.437 & 3\,911 & $1.2\times10^{-3}$ & $2.6\times10^{-2}$ \\
    $0.01$ & 0.1550 & 0.1550 [0.1475, 0.1575]   & $4/5^{\dagger}$ & 0.081 & 1\,301 & $2.4\times10^{-3}$ & $2.4\times10^{-2}$ \\
    $0.02$ & 0.1475 & 0.1475 [0.1450, 0.1475]   & 5/5 & 0.047 & 1\,152 & $1.8\times10^{-3}$ & $2.3\times10^{-2}$ \\
    $0.05$ & 0.1275 & 0.1275 [0.1250, 0.1275]   & 5/5 & 0.027 & 896    & $1.8\times10^{-3}$ & $2.5\times10^{-2}$ \\
    $0.1$  & 0.1050 & 0.1050 [0.1050, 0.1050]   & 5/5 & 0.017 & 652    & $2.9\times10^{-3}$ & $2.5\times10^{-2}$ \\
    \bottomrule
  \end{tabular}
  \caption{$\varepsilon$-sweep at $\lambda=0.6$ (front sampler, $5$ seeds
    per row; final validation mean, violations, and tail distances are
    seed medians). The exact relay $\varepsilon=0$ is the negative
    control and reproduces the failure mode of the discontinuous target:
    no persistent crossings, an order-of-magnitude larger residual, and
    near-total violation of the pointwise inclusion on the evaluation
    cloud. From $\varepsilon=0.02$ on, band-entry times match the
    reference to within one grid step with persistent crossings for all
    seeds. $^{\dagger}$At $\varepsilon=0.01$ one seed's first crossing is
    transient; its persistent crossing occurs only at $t=0.25$. The tail
    is consistently an order of magnitude closer to $u\equiv0$ than to
    the extremal witness $w_\lambda$ of
    Remark~\ref{rem:relay-witness}.}
  \label{tab:relay-eps}
\end{table}

\begin{figure}[htbp]
  \centering
  \includegraphics[width=0.62\textwidth]{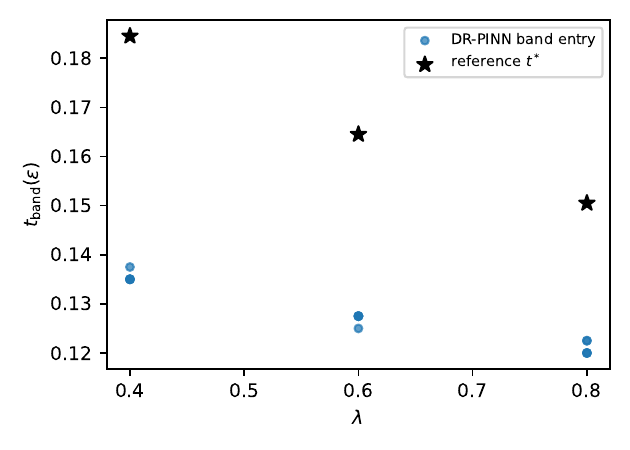}
  \caption{Band-entry times across
    $\lambda\in\{0.4,0.6,0.8\}$ at $\varepsilon=0.05$: DR-PINN values
    (five seeds per $\lambda$, dots) against the reference band-entry
    times and the reference extinction times $t^*$ (stars). The DR-PINN
    values sit on the reference band-entry values to within the
    evaluation-grid resolution and preserve the monotone ordering in
    $\lambda$ run by run.}
  \label{fig:relay-sweep}
\end{figure}

\begin{figure}[htbp]
  \centering
  \includegraphics[width=0.62\textwidth]{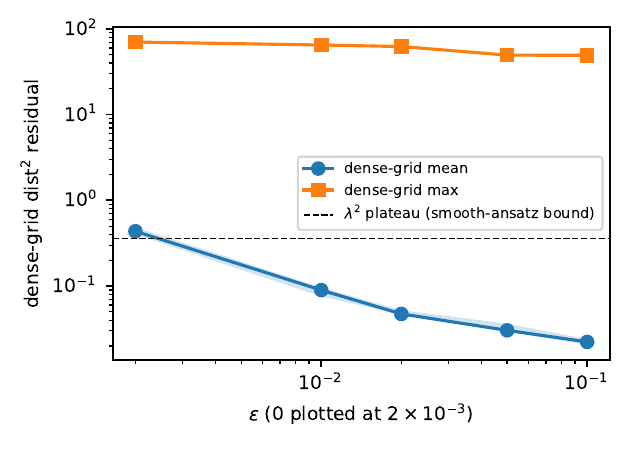}
  \caption{Dense-grid distance-residual statistics along the
    $\varepsilon$-sweep at $\lambda=0.6$ ($\varepsilon=0$ plotted at
    $2\times10^{-3}$ on the logarithmic axis). The mean residual
    decreases monotonically with $\varepsilon$ and drops below the
    $\lambda^2$ plateau of the smooth-ansatz bound as soon as the band
    opens, while the maximum residual --- attained ahead of the moving
    interface --- remains of order $\lambda^2$ times the squared
    operator overshoot for every $\varepsilon$.}
  \label{fig:relay-eps-sweep}
\end{figure}

Figure~\ref{fig:relay-vs-ref} overlays the $L^2(\Omega)$-norm decay
curves of all five seeds against the reference for the baseline
$\lambda=0.6$ (companion figures for $\lambda=0.4$ and $0.8$ are included
in the replication package). The curves coincide with the reference to
within ${\approx}2\%$ in sup norm over the pre-band interval, for every
seed and every $\lambda$ of the sweep. After band entry the network norms
stabilize at a small nonzero level of order $10^{-3}$--$10^{-2}$: a
single analytic tanh-network trajectory generally cannot be nonzero
before the extinction time and exactly zero on an open post-extinction
interval, so a small nonzero post-band norm is expected and is not
interpreted as exact extinction; the band-entry observable
\eqref{eq:relay-tband} is designed precisely so that this tail does not
contaminate the reported times.

\begin{figure}[htbp]
  \centering
  \includegraphics[width=0.72\textwidth]{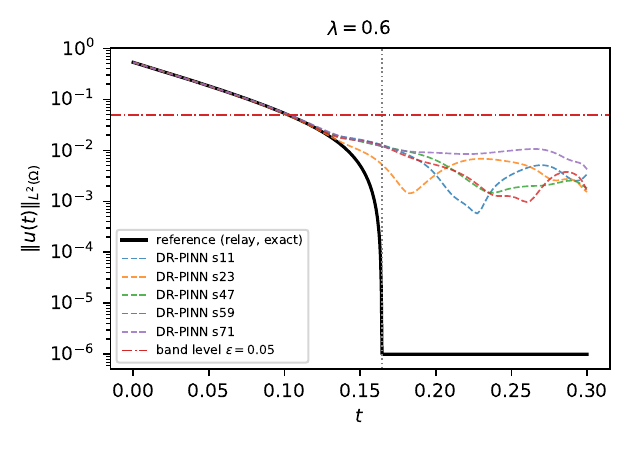}
  \caption{$L^2(\Omega)$-norm decay for $\lambda=0.6$,
    $\varepsilon=0.05$, on a log scale: reference solution (solid) and
    the five DR-PINN seeds (dashed). The horizontal dash-dotted line
    marks the band level $\varepsilon$, the vertical dotted line the
    reference extinction time $t^*_{\mathrm{ref}}=0.1645$. The curves
    track the reference to within ${\approx}2\%$ before band entry;
    after band entry the smooth surrogates retain a small nonzero tail,
    which the band-entry observable is designed to ignore.}
  \label{fig:relay-vs-ref}
\end{figure}

Figure~\ref{fig:relay-branch} shows the branch-selection diagnostic for a
representative baseline run. For state values
$s=u_\theta(t,x,y)>\varepsilon$, the cloud of operator values
$z=\partial_tu_\theta-\Delta u_\theta$ concentrates tightly on the branch
$z=-\lambda=-0.6$ of the graph of $\Phi_\varepsilon$, confirming that the
network selects the correct sign branch over most of $Q$; almost no
collocation points appear with $s<-\varepsilon$ (sign preservation by the
relay flow, up to a small smooth undershoot near the front). Inside the
band $|s|\le\varepsilon$ the wide step $[-\lambda,\lambda]$ absorbs the
operator values, and the residual there is numerically zero. The residual
mass that remains --- roughly $900$ of the $4\,000$ cloud points violate
the inclusion beyond $10^{-4}$ --- sits in a thin layer \emph{ahead} of
the moving interface, where $s>\varepsilon$ but the smooth output is
already transitioning: there the vertical scatter of $z$ reaches values
of order $5$, far above the graph. This scatter is a direct consequence
of the jump of the relay at the band edge: no smooth network output can
resolve a discontinuous right-hand side pointwise.

\begin{figure}[htbp]
  \centering
  \includegraphics[width=0.72\textwidth]{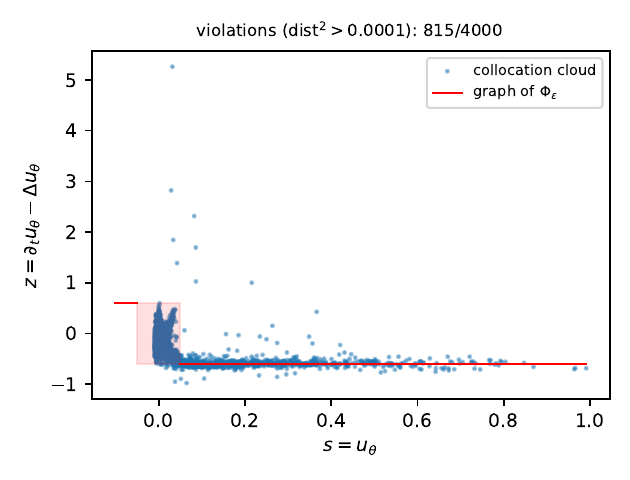}
  \caption{Branch-selection diagnostic ($\lambda=0.6$,
    $\varepsilon=0.05$): empirical cloud of
    $(u_\theta,\,\partial_tu_\theta-\Delta u_\theta)$ at $4\,000$
    post-training collocation points, overlaid on the graph of
    $\Phi_\varepsilon$ (red; the shaded step of width $2\varepsilon$ is
    the band). Points with $s>\varepsilon$ cluster tightly on the branch
    $z=-0.6$; inside the band the residual vanishes; the vertical
    scatter concentrates in a thin layer ahead of the moving interface,
    where the discontinuity of the relay limits pointwise accuracy. The
    inset magnifies the band region.}
  \label{fig:relay-branch}
\end{figure}

\paragraph{Sampler ablation}
Table~\ref{tab:relay-ablation} compares, at the baseline
$(\lambda,\varepsilon)=(0.6,0.05)$, the front-based adaptive sampler
against residual-based RAR \cite{LuMengMaoKarniadakis2021DeepXDE} and
purely uniform collocation, five seeds each. All three samplers produce
identical median band-entry times ($0.1275$, persistent for every seed),
and their residual and violation statistics overlap within seed-to-seed
variability. At these settings the adaptive schemes therefore provide no
measurable benefit over uniform sampling; we report this null result
explicitly and retain the front-based sampler in the primary sweeps only
for uniformity of protocol and as a safeguard for smaller $\varepsilon$,
where the band is thinner relative to the domain.

\begin{table}[htbp]
  \centering
  \small
  \setlength{\tabcolsep}{2.5pt}
  \begin{tabular}{lcccc}
    \toprule
    sampler &
    \shortstack{$t_\varepsilon$\\median [range]} &
    \shortstack{persistent\\crossings} &
    \shortstack{final val.\ mean\\median [range]} &
    \shortstack{violations\\median} \\
    \midrule
    front (adaptive) & 0.1275 [0.1250, 0.1275] & 5/5 & 0.027 [0.020, 0.041] & 896 \\
    RAR              & 0.1275 [0.1275, 0.1275] & 5/5 & 0.023 [0.017, 0.026] & 899 \\
    uniform          & 0.1275 [0.1250, 0.1300] & 5/5 & 0.026 [0.016, 0.027] & 895 \\
    \bottomrule
  \end{tabular}
  \caption{Sampler ablation at $(\lambda,\varepsilon)=(0.6,0.05)$, $5$
    seeds per sampler. The three schemes are statistically indistinguishable on every reported metric; the band-entry observable is insensitive to the collocation strategy at these settings.}
  \label{tab:relay-ablation}
\end{table}

\paragraph{Residual structure and remaining limitations}
The residual reduction is not uniform over $Q$. On
the $4\,000$-point evaluation cloud at
$(\lambda,\varepsilon)=(0.6,0.05)$, the $\Phi_\varepsilon$ distance
residual has seed-median mean ${\approx}0.030$ but maximum
${\approx}50$: inside the band and after band entry the residual is
numerically zero (post-band mean $\lesssim10^{-6}$), so the trained tail
is genuinely admissible, but ahead of the moving interface --- where
$|u_\theta|>\varepsilon$ and the smooth surrogate must still approximate
a sign branch across a thin transition layer --- the residual remains
concentrated and large in absolute terms. The pre-band
$L^2(\Omega)$-norm curves nevertheless track the reference to within
${\approx}2\%$ in sup norm. Thus the band-entry observable
\eqref{eq:relay-tband} reproduces the extinction dynamics quantitatively:
it is seed-stable, persistent, monotone in $\lambda$, and grid-exact against
the reference for $\varepsilon\ge0.02$. At the same time, the hypothesis
$\mathcal{J}(u_n)\to0$ of Theorem~\ref{thm:consistency-parabolic} is not
realized uniformly because the continuous residual remains dominated by
the pre-band interface layer. In addition, Remark~\ref{rem:relay-witness}
shows that at $(\lambda,\varepsilon)=(0.4,0.05)$ and
$(0.6,\{0.05,0.1\})$ the $\varepsilon$-banded inclusion admits nonzero
steady sliding states, so a vanishing post-band residual alone would not
identify the extinct solution; the recorded tail distances show that
every trained network is an order of magnitude closer to $u\equiv0$
(median ${\approx}2\times10^{-3}$) than to the extremal witness
(${\approx}2.5\times10^{-2}$), i.e.\ the networks select the solution of
the exact relay even where the relaxed problem does not force it.

\section{Conclusions}
\label{sec:conclusions}

We proposed  Distance-Residual Physics-Informed Neural Networks
(DR-PINNs) for set-valued differential laws. The usual pointwise residual
is replaced by the squared distance from the differential operator to the
admissible set; this distance vanishes exactly on admissible operator
values and gives the classical PINN loss when the set is a singleton. On the
theoretical side, we proved consistency results for both ordinary
differential inclusions (Theorem~\ref{thm:consistency}) and parabolic
partial differential inclusions
(Theorem~\ref{thm:consistency-parabolic}): any sequence of admissible
candidates whose \emph{continuous} distance-residual functional tends to
zero converges, along a subsequence, to an exact solution of the target
inclusion; relating the finite-collocation training loss to the continuous
functional requires the additional sampling and regularity assumptions
discussed in Remark~\ref{rem:collocation-gap}. These are conditional
consistency results for the continuous functional. They do not assert the
existence of a neural-network sequence with vanishing continuous residual,
convergence of a gradient-based optimizer, or validity of the bridging
assumptions for random or adaptive collocation, and they do not provide
convergence rates. Each of these issues requires additional analysis.
The proofs combine a priori estimates, compactness, weak convergence of the
differential operators, and the lower semicontinuity of normal convex
integrands; in view of the intrinsic non-uniqueness of solutions of
inclusions, this compactness-type statement is the natural analogue of
the error bounds available in the single-valued setting. On the practical
side, the projection defining the loss is computable: explicitly for
intervals, balls, and boxes; by a small quadratic program, combined with
an offline Quickhull reduction, for polytopic sets; and by a scalar
secular equation for ellipsoidal sets.

The numerical experiments cover three settings. For a
linear control system with a polytopic input set, the distance residual
is driven to machine-precision levels and the trained ensemble stays
close to an empirically sampled reachable tube. For a planar inclusion
with a rotating ellipsoidal constraint, the DR-PINN companion run reaches
a collocation loss of $\EllipCompLoss$ with endpoint constraints enforced
exactly, while a complementary hard-admissible parametrization
illustrates how the residual can instead be eliminated at the time-grid
nodes by construction (with feasibility between nodes confirmed a
posteriori on a dense validation grid), freeing the remaining degrees of freedom to select a particular element of the solution set. For a reaction-diffusion inclusion with relay feedback, training against an $\varepsilon$-banded relaxation $\Phi_\varepsilon\supseteq\Phi$ (which satisfies the same hypotheses \ref{P1}--\ref{P3}, so the same consistency theorem applies) turns a merely qualitative agreement into a quantitative one: for $\varepsilon\ge0.02$ the band-entry times of the trained networks are persistent, seed-stable (ranges of at most one evaluation-grid step over
five seeds), monotone in the relay strength $\lambda$, and agree with the
reference solver to the resolution of the evaluation grid, while the
exact relay $\varepsilon=0$ retains its failure mode and is reported as a
negative control. A sampler ablation shows the adaptive collocation
schemes to be statistically indistinguishable from uniform sampling at
these settings. The residual nonetheless remains concentrated in a thin
layer ahead of the extinction front, where a smooth surrogate must
approximate a sign branch of the relay; and the relaxed inclusion admits
nonzero steady sliding states whenever
$\lambda\|(-\Delta)^{-1}1\|_{L^\infty(\Omega)}\le\varepsilon$, yet the
trained tails are an order of magnitude closer to the extinct solution
than to the extremal sliding state, i.e.\ the networks select the
solution of the exact relay even where the relaxation does not force it.

Several directions remain open. Quantitative error estimates would require
additional structure, such as one-sided Lipschitz or monotonicity
conditions on $F$ or $\Phi$, beyond the compactness-based consistency used
here. State-dependent admissible sets also raise a separate backward-pass
question because the projection moves with the network state. The relay
experiment suggests combining the distance residual with hard-admissible
parametrizations, adaptive sampling, or nonsmooth network components near
discontinuities of the multifunction. Extensions to hyperbolic partial
differential inclusions, sweeping processes, and differential variational
inequalities are natural further directions.
Relaxing convexity of $F(t,x)$ and $\Phi(u)$ is also a natural but nontrivial open problem, since it would require prox-regularity of the admissible sets in place of convexity for the differentiability of the squared distance.

\section*{Code and data availability}

All numbers, tables, and figures of Section~\ref{sec:numerics} are
produced by a single execution of the replication pipeline referenced
throughout the text. The replication package is publicly available at
\url{https://github.com/CA24136PINN/pinns-for-differential-inclusions}. The package contains: the scripts
\texttt{run\_experiment\_61.py}, \texttt{run\_experiment\_62.py}, and
\texttt{run\_experiment\_63.py} generating every table and figure of
Sections~\ref{subsec:linear-control}--\ref{subsec:parabolic-PDI}
(including the automatically generated numerical macros \texttt{generated/results\_6x.tex} compiled into this manuscript); all
random seeds; the saved network checkpoints, in particular those
underlying the figures and tables of Section~\ref{subsec:parabolic-PDI}; the exact definitions and weights of all regularizers entering the cost \eqref{eq:ellipse-cost}; and a manifest
recording the hardware, operating system, and full library versions of
the runs whose outputs are reported here.

\section*{Acknowledgments}
This article is based upon work from COST Action InterCoML, CA24136, supported by COST (European Cooperation in Science and Technology).

\bibliographystyle{plainnat}
\bibliography{references}

\end{document}